\PassOptionsToPackage{dvipsnames}{xcolor}
\documentclass{amsart}
\usepackage{amsmath,amssymb,amsthm}
\usepackage[indent=0pt,skip=6pt plus 1pt]{parskip}
\usepackage{etoolbox}
\makeatletter
\@ifclassloaded{amsart}{%
  \patchcmd{\@maketitle}{\dimen@34\p@}{\dimen@22\p@}{}%
    {\PackageError{miura}{AMS title spacing patch failed}{}}%
}{}
\makeatother
\usepackage{booktabs}
\usepackage{placeins}
\usepackage{float}
\usepackage{graphicx}
\usepackage{tikz}
\usetikzlibrary{decorations.pathreplacing}
\usepackage{xcolor}
\usepackage{hyperref}
\hypersetup{colorlinks=true,linkcolor=MidnightBlue,citecolor=MidnightBlue,
            urlcolor=MidnightBlue}
\usepackage{url}

\AddToHook{cmd/@parboxrestore/after}{%
  \setlength{\parindent}{0pt}%
  \setlength{\parskip}{6pt plus 1pt}%
}

\newtheorem{theorem}{Theorem}[section]
\newtheorem{lemma}[theorem]{Lemma}
\newtheorem{proposition}[theorem]{Proposition}
\newtheorem{corollary}[theorem]{Corollary}
\newtheorem{conjecture}[theorem]{Conjecture}
\newtheorem{maintheorem}[theorem]{Theorem}
\theoremstyle{definition}
\newtheorem{definition}[theorem]{Definition}
\theoremstyle{remark}

\newcommand{\OFG}{\mathrm{OFG}}
\newcommand{\ofgdist}{d_{\OFG}}
\newcommand{\dgrid}{d_1}
\newcommand{\Z}{\mathbb{Z}}
\newcommand{\Q}{\mathbb{Q}}
\newcommand{\Mm}{M_{m,n}}
\newcommand{\Gmn}{G_{m,n}}
\newcommand{\Aset}{\mathcal{A}}
\newcommand{\disp}{\operatorname{disp}}
\DeclareMathOperator{\diameter}{diam}
\DeclareMathOperator{\lcm}{\lcm}

\title[Height functions on the Miura-ori flip graph]{Height functions on the $m \times n$ {Miura}-ori flip graph: degree sequence and diameter}
\author[Chakshu Gupta]{Chakshu Gupta}
\address{College of Computing, Georgia Institute of Technology}
\email{cgupta65@gatech.edu}
\date{}

\begin{document}
\raggedbottom

\begin{abstract}
Miura-ori is an origami grid of parallelogram faces whose creases carry mountain or valley folds. Its flip graph joins locally flat-foldable assignments differing by a legal face flip. Such a flip switches every crease bordering one face while preserving flat-foldability at each crease intersection, without requiring the whole sheet to fold flat. Earlier two-row formulas and small-flip counts left a uniform rectangular counting law unresolved, while the known maximum shortest-path distance left its attaining pairs unclassified. Integer face functions with adjacent values differing by one identify legal flips with local extrema and determine shortest-path distances through their differences. For a fixed number $d\ge3$ of legal flips, the assignment count on an $m\times n$ face grid with $m,n\ge2$ is a symmetric polynomial once $m,n\ge d-1$. Its degree is $d-2$ in each variable and in total, with leading part proportional to $m^{d-2}+n^{d-2}$ except at $d=4$. Exact computation identifies formulas through ten flips. For fixed $2\le m\le d-2$ and $n\ge d-1$, a finite construction gives correction polynomials of degree at most $d-3$ in $n$. At $m=d-2$, an exact linear correction involving Baxter permutations proves the threshold sharp for $d\ge4$. Computer assistance enters the four-flip total-degree proof and a cubic correction-degree bound. A sharper uniform correction-degree bound remains open. Interior-face flips first become possible with five legal flips on grids at least $3\times3$. Classifying face-function differences that maximize their minimum total absolute deviation from a constant and reconstructing every realizing assignment pair gives all farthest pairs. Outside $2\times2$, these lie among four minimum-flip assignments with two legal flips each, with four unordered pairs when one dimension is two and the other even, and two otherwise. The counts and pair classification describe available moves and maximal separation without enumerating the full flip graph.
\end{abstract}

\enlargethispage{4pt}
\maketitle

\hypertarget{maintext-start}{}\label{sec:maintext-start}
\section{Introduction}\label{sec:intro}

The Miura-ori is a grid of parallelogram faces whose creases receive
mountain or valley labels~\cite{GineproHull2014counting,Miura1994map}.
A legal face flip switches every crease bordering one face while preserving
the flat-folding conditions at each crease
intersection~\cite{Akitaya2020faceflips}.
The Miura-ori flip graph joins locally flat-foldable assignments related by
one legal face flip. The rectangles studied here have $m\ge2$ rows and
$n\ge2$ columns of faces. Local flat-foldability does not require the whole
sheet to fold flat with compatible layer orders.

The degree sequence counts assignments by their numbers of available legal
flips, whereas the diameter measures the greatest shortest-path distance
between assignments. Both invariants are known for equal-angle
single-vertex patterns~\cite{Hull2022maximal} and for the two-row
Miura-ori~\cite{Christensen2025origami}. The two-row study also identifies
the four degree-two assignments on general rectangles and determines the
distance between opposite members of this family. Extending the degree
counts to arbitrary rectangles and identifying every pair at maximum
distance require more than the two-row formulas.

Proper three-colourings of the face grid represent rectangular assignments,
with cyclic renaming of all colours representing the same
assignment~\cite{GineproHull2014counting}. A legal face flip becomes a
single-cell recolouring~\cite{Akitaya2020faceflips}. Recolouring algorithms
decide reachability~\cite{Cereceda2011finding} and find a shortest path for a
specified pair~\cite{Johnson2016shortest}. Integer heights, differing by one
across each grid edge and normalized to zero at one corner, give a common
representation for counting assignments and studying their
distances~\cite[Section~4]{Goldberg2004sampling}. Legal flips occur exactly
at strict local extrema, and the local minima together with their values
determine the whole height function. Counting therefore concerns functions
with a fixed number of extrema, while distances depend on differences of
two height functions.

An earlier version of this article~\cite[Version~2]{Gupta2026companion}
counted assignments through five flips and reduced the diameter upper bound
to an inequality for integer grid functions.
A sequel~\cite[Version~2]{Gupta2026sequel} established eventual polynomiality
and the exact degree in each grid dimension, and reported formulas through
ten flips. Its general total-degree law and proposed correction at the first
dimension below the polynomial region remained unproved.
The grid inequality and exact diameter have since been
established~\cite[Corollary~5.1 and Theorem~1.2]{Watts2026diameter} and are used here
as known results. The present revision consolidates the two preprints,
proves the total-degree law and short-side correction formulas, and
classifies all farthest pairs.

For each fixed number $d\ge3$ of legal flips, the assignment count is a
symmetric polynomial in $m,n$ once $m,n\ge d-1$.
Sections~\ref{sec:counts} and~\ref{sec:total-degree} establish its exact
degree $d-2$ in each variable and in total. The total-degree proof is
computer-assisted at $d=4$, the only case with a mixed leading term.\footnote{\url{https://github.com/ChakshuGupta13/lab/tree/main/math/gupta2026origami}.}
For $d=3$ and every $d\ge5$, the leading part is proportional to
$m^{d-2}+n^{d-2}$. The degree bounds together with exact counts determine
the explicit formulas through ten flips in
Appendix~\ref{app:polynomials}. A separate counting
note~\cite[Version~1.0, Theorems~1.1 and~3.1]{Watts2026}, based on the earlier
public framework~\cite[Version~2]{Gupta2026sequel}, supplies another
degree-eight certificate and stronger domains for its refinements by the
numbers of minima and maxima. Appendix~\ref{app:independent-eight} compares
these certificate routes and their domains.

For a fixed short dimension $2\le s\le d-2$ and long dimension
$n\ge d-1$, Section~\ref{sec:subthreshold} constructs the correction to the
count polynomial, with integer coefficients in a Newton basis and degree
at most $d-3$ in $n$. At the first excluded dimension $s=d-2$, an exact
linear correction involves Baxter permutations. It vanishes at $n=d-1$
but is nonzero for $n\ge d$. This proves that $d-1$ is the sharp common
threshold when both dimensions vary. By symmetry, $n=d-1$ is also the
exact onset when the other dimension is fixed at $m\ge d$.

Stronger bounds describe corrections closer to the uniform threshold.
The general bounds and the quadratic bound at $s=d-3$ are deductive.
The cubic bound at $s=d-4$ is deductive for $d=6,7,8$ and uses a four-column
finite-state computation for $d\ge9$. The remaining correction-degree
question is stated in Section~\ref{sec:discussion}.

The geometry of low-flip assignments gives a different refinement of the
counts. Section~\ref{sec:geometry} locates their legal flips and proves
that interior-face flips first become possible with five available flips
on grids at least $3\times3$.

For the distance question, the known diameter leaves its attaining pairs
to be classified. Section~\ref{sec:diameter} identifies all integer grid
functions whose adjacent values differ by at most one and which maximize
their minimum total absolute deviation from an integer constant. It then
reconstructs every pair of normalized heights whose half-difference
maximizes that deviation. The equality analysis uses the separately communicated
covariance proof attributed in Section~\ref{subsec:dispersion-equality}.
On the two-by-two square there are seven unordered farthest pairs, a count
already implicit in the published graph $K_{2,4}$~\cite[Figure~5]{Christensen2025origami}.
Outside that square, every farthest pair lies among the four minimum-flip
assignments. There are four unordered farthest pairs when one dimension is
two and the other is even and at least four, and two otherwise. Distinct
height pairs can have the same half-difference, so a complete classification
requires every reconstruction, not just one pair for each maximizing
function.

\section{Heights and envelopes}\label{sec:miura}
\subsection{Miura assignments and legal flips}

A \emph{crease pattern} is a connected graph $C$ with straight-line
edges embedded in a closed planar region, the
sheet~\cite{GineproHull2014counting}. Its edge set $E(C)$ consists
of the creases, which separate the sheet into faces. The sheet's
outer boundary is not a crease. A \emph{mountain-valley assignment}
is a map $\mu\colon E(C)\to\{\mathrm{M},\mathrm{V}\}$~\cite{Hull2003counting}.
Viewed from a fixed side of the sheet, mountain folds are convex and
valley folds are concave.

\begin{definition}[Miura-ori]\label{def:miura}
For integers $m,n\ge2$ and an angle $\alpha$ with $0<\alpha<\pi/2$,
the $m\times n$ Miura-ori $\Mm$~\cite{Miura1994map} is the crease
pattern of $mn$ congruent parallelograms with zigzag rows and straight
columns. The creases are the shared sides of adjacent faces.
Four creases meet at each of the $(m-1)(n-1)$ interior
vertices, with sector angles $\alpha,\alpha,\pi-\alpha,\pi-\alpha$
in cyclic order~\cite{GineproHull2014counting}.
\end{definition}

The restriction $m,n\ge2$ excludes one-row and one-column sheets, which
have no interior vertex. The arguments use the combinatorial structure and acute/obtuse
sector pattern, not the value of $\alpha$ within the stated range
or the parallelograms' side ratio. Figure~\ref{fig:miura} shows $M_{3,4}$.

\begin{figure}[ht]
\centering
\begin{tikzpicture}[x=1.2cm, y=1.2cm, font=\small]
  \fill[gray!18] (2.3,0.95) -- (3.45,1.23) -- (3.45,2.18) -- (2.3,1.9) -- cycle;
  \foreach \yb in {0.95, 1.9}
    \draw[thick] (0,\yb) -- (1.15,\yb+0.28) -- (2.3,\yb)
                 -- (3.45,\yb+0.28) -- (4.6,\yb);
  \foreach \yb in {0, 2.85}
    \draw[gray!55] (0,\yb) -- (1.15,\yb+0.28) -- (2.3,\yb)
                 -- (3.45,\yb+0.28) -- (4.6,\yb);
  \draw[thick] (1.15,0.28) -- (1.15,3.13);
  \draw[thick] (2.3,0) -- (2.3,2.85);
  \draw[thick] (3.45,0.28) -- (3.45,3.13);
  \draw[gray!55] (0,0) -- (0,2.85);
  \draw[gray!55] (4.6,0) -- (4.6,2.85);
  \fill (2.3,0.95) circle (1.6pt);
  \node[font=\footnotesize, anchor=east] at (2.20,0.82) {$v$};
  \node[font=\scriptsize] at (2.58,1.31) {$\alpha$};
  \node[font=\scriptsize] at (2.69,0.64) {$\pi-\alpha$};
  \node[font=\footnotesize] at (3.05,1.78) {$(2,3)$};
\end{tikzpicture}
\caption{The Miura-ori $M_{3,4}$ has three rows and four columns of congruent
parallelograms, with zigzag row lines and straight column lines. At the marked interior
vertex $v$ the two acute angles $\alpha$ are adjacent, as are the two
obtuse angles $\pi-\alpha$. One of each is marked. The lighter outer edges bound
the sheet and carry no crease. The shaded face is in row $2$ and column $3$,
with $v$ at one corner.}
\label{fig:miura}
\end{figure}

At an interior Miura vertex, the alternating sector-angle sum is
zero, so the unlabelled vertex admits a flat fold by the
Kawasaki--Justin condition~\cite[Theorem 1]{Hull2003counting}.
The Maekawa--Justin condition requires three mountain creases and
one valley, or the reverse~\cite[Theorem 2]{Hull2003counting}.
The minority crease cannot separate the two obtuse sectors, because
the acute sectors cannot contain the obtuse ones in that
fold~\cite[Section 2]{GineproHull2014counting}. Choosing the minority
crease among the other three and the majority label gives six
patterns, each of which folds flat~\cite[Figure 2]{GineproHull2014counting}.
An assignment on $\Mm$ is \emph{locally flat-foldable} if it matches
one of these patterns at every interior vertex.

Local flat-foldability does not imply global flat-foldability,
since vertices may impose incompatible layer orders~\cite[Section
5]{GineproHull2014counting}. All locally flat-foldable assignments are
included, without requiring the whole sheet to fold flat.

A \emph{face flip}~\cite{Akitaya2020faceflips}
switches every crease bounding one face between mountain and valley.
This includes creases ending at the sheet boundary, but not the
uncreased outer edges. On $\Mm$, a corner face has two bounding
creases, a non-corner boundary face three and an interior face four.

\begin{definition}[Origami flip graph]\label{def:ofg}
For a crease pattern $C$, the origami flip graph $\OFG(C)$ has as vertices the locally
flat-foldable assignments of $C$, meaning those that admit a flat
fold at each interior vertex separately. Two vertices are adjacent
when a face flip transforms one assignment into the other.
\end{definition}

\subsection{Colourings, rotations and heights}\label{sec:background}

The face grid $\Gmn$ has cells $(i,j)$ with
$1\le i\le m$ and $1\le j\le n$, adjacent when they differ by one in
a single coordinate. Their Manhattan distance is
$\dgrid\big((i,j),(i',j')\big)=|i-i'|+|j-j'|$.
The cell $(1,1)$ is the root. A height function is a map
$h\colon\Gmn\to\Z$ with $|h(u)-h(v)|=1$ on each edge,
normalized so that $h(1,1)=0$.

The recolouring graph $\mathcal R_3(\Gmn)$ has all proper
$3$-colourings with colours $0,1,2$ as vertices and joins two when
they differ at one cell. The rotation $\rho$ adds one modulo three
to every colour, leaving all edge differences unchanged. Subtracting
the root colour at every cell gives each rotation orbit its unique
representative with root colour zero.

\begin{proposition}[Height representation]\label{prop:height-model}
For integers $m,n\ge2$, the locally flat-foldable assignments of $M_{m,n}$, proper three-colourings pinned at the root and normalized height functions on $G_{m,n}$ correspond bijectively. Legal face flips correspond to proper single-cell recolourings modulo global colour rotation, giving
\begin{equation}\label{eq:iso}
  \OFG(\Mm)\cong\mathcal R_3(\Gmn)/\langle\rho\rangle
  \qquad(m,n\ge2).
\end{equation}
\end{proposition}

\begin{proof}
The encoding records crease labels by colour differences. Lifting
those differences gives heights, while reversing the differences
around a face gives the move correspondence.

\medskip\noindent
\emph{Colouring encoding.}
The colouring
bijection~\cite[Section 4]{GineproHull2014counting} identifies locally
flat-foldable assignments of $\Mm$ with proper $3$-colourings of
$\Gmn$ pinned to colour zero at the upper-left corner. Along the
bijection's fixed grid-edge orientations, mountain and valley
correspond to colour differences $+1$ and $-1$ modulo three.

\medskip\noindent
\emph{Integer lift.}
For a proper $3$-colouring $\sigma$ pinned by $\sigma(1,1)=0$,
each oriented edge $uv$ has a unique increment
$\delta(u,v)\in\{-1,1\}$ congruent to $\sigma(v)-\sigma(u)$ modulo
three. Around a unit square, the sum is a multiple of three in
$\{-4,-2,0,2,4\}$, hence zero. Grid cycles are sums of oriented
unit-square boundaries, whose interior edges cancel, so all cycle
sums vanish. Path sums from the root therefore define a height,
with $h(1,1)=0$ from the empty path. Modulo three the sums telescope
to $\sigma(v)$, recovering the pinned colouring. Conversely, a height
reduces modulo three to a proper colouring with root colour zero,
since adjacent heights differ by one and $h(1,1)=0$. The recovered
increments are the height's own by uniqueness, and their path sums
recover the height. Thus pinned colourings, rotation orbits and
normalized heights correspond bijectively. For $m,n\ge2$, they
represent Miura-OFG vertices.

\medskip\noindent
\emph{Legal moves.}
A legal face flip leaves the colour difference on every grid edge
not incident to its cell unchanged~\cite[proof of Lemma
4.1]{Akitaya2020faceflips}. Thus the new colour minus the old colour
is equal modulo three at both endpoints of each remaining edge.
Deleting the flipped cell leaves a connected grid. Indeed, a row
not containing that cell exists because $m\ge2$. Every remaining
cell outside its column reaches that row vertically. A cell in its
column first moves horizontally, then vertically, avoiding the
deleted cell because $n\ge2$. The common colour shift therefore
propagates to all remaining cells. Subtracting that shift from every
colour leaves only the flipped cell changed, since its incident
differences reverse.

\medskip\noindent
Conversely, after a proper recolouring at one cell, every neighbour
avoids both its old and new colours and hence has the third colour.
Every incident difference therefore changes sign, while all other
differences stay fixed. By the encoding, exactly the creases bounding
the face switch labels, giving a legal flip.

\medskip\noindent
Restoring the pin by global rotation after a root recolouring
changes the other $mn-1$ cells. The pinned representatives are
therefore not adjacent in $\mathcal R_3(\Gmn)$, although their
orbits are. This proves the move correspondence modulo rotation.
\end{proof}

\subsection{Extrema and cone envelopes}

The move correspondence turns the degree question into a local recolouring
test. A cell is a strict local minimum of $h$ if every neighbour is one
higher and a strict local maximum if every neighbour is one lower.
The minima and maxima are the extrema of $h$.

\begin{lemma}[Degree and extrema]\label{lem:degextrema}
\begingroup
\emergencystretch=1em
For integers $m,n\ge2$, a vertex of $\OFG(\Mm)$ has a normalized
height representative $h$ on $\Gmn$. A face can be legally flipped
exactly when its cell is a strict local extremum of $h$. The individual
face flips at these cells give all the neighbours of the vertex, with
different extrema giving different neighbours. The vertex's degree
therefore equals the number of strict local extrema of $h$.
\par
\endgroup
\end{lemma}

\begin{proof}
The colouring $\sigma=h\bmod3$ represents the assignment. By Proposition~\ref{prop:height-model},
flipping a face preserves local flat-foldability exactly when its cell can
be properly recoloured. At a cell $v$, the grid neighbours have heights
$h(v)-1$ or $h(v)+1$, whose residues are the two colours different from
$\sigma(v)$. If both occur, neither alternative colour is available. If
only one occurs, there is exactly one proper replacement colour. This
happens exactly when all grid neighbours have the same height, meaning
that $v$ is a strict local minimum or maximum. Thus each extremum produces
one new proper colouring, and every neighbouring assignment arises from
such a recolouring.

\medskip\noindent
A colouring obtained by one of these flips differs from $\sigma$ at one
cell, and flips at different cells give different colourings. Each pair
of distinct colourings among $\sigma$ and the flipped colourings agrees
at a cell where both retain $\sigma$'s colour, since each differs from
$\sigma$ at at most one cell and the grid has at least four cells.
A cyclic colour rotation relating the pair would have to fix their common
colour there, so it must be the identity. The colourings are different,
and therefore lie in different rotation orbits. The flipped colourings
thus represent distinct neighbours of the original vertex, one for each
strict local extremum of $h$, proving the degree formula.
\end{proof}

The extrema count the available flips, and the minima and their values
recover the whole height function.
A cone with apex $p\in\Gmn$ and offset $c\in\Z$ is the function
$v\mapsto c+\dgrid(p,v)$. Figure~\ref{fig:kcone-envelope}
illustrates the pointwise minimum of cones based at these minima.

\begin{lemma}[Cone envelope]\label{lem:coneenv}
For integers $m,n\ge2$ and a height function $h$ on
$\Gmn$, its sets $P$ and $Q$ of strict local minima and maxima are nonempty and, for every
cell $v\in\Gmn$,
\begin{equation}\label{eq:env}
  h(v)=\min_{p\in P}\big(h(p)+\dgrid(p,v)\big)
      =\max_{q\in Q}\big(h(q)-\dgrid(q,v)\big).
\end{equation}
\end{lemma}

\begin{proof}
The finite grid has a global minimum and maximum of $h$. Every cell
has a neighbour, and adjacent heights differ by one, so these are
strict local extrema and $P,Q$ are nonempty. For a cell $v\in\Gmn$,
summing unit height changes along a shortest path from each $p\in P$
to $v$ gives
\[
  h(v)\le h(p)+\dgrid(p,v)\qquad p\in P.
\]
A descending path from $v$ continues through a lower neighbour whenever
one exists. Each step lowers $h$ by one, so no cell repeats and the
finite path ends at a strict local minimum $p$. Its length is
$h(v)-h(p)$ and is at least the grid distance $\dgrid(p,v)$, so the
upper bound gives
\[
  \dgrid(p,v)\le h(v)-h(p)\le\dgrid(p,v).
\]
Thus one cone attains $h(v)$ and none lies below it, proving the first
equality.

\medskip\noindent
Negation preserves the unit edge differences and root normalization,
so $-h$ is a height function whose strict local minima are $Q$.
Applying the first equality to $-h$ gives
\[
  -h(v)=\min_{q\in Q}\big(-h(q)+\dgrid(q,v)\big).
\]
Negating both sides turns the minimum into a maximum, giving
\[
  h(v)=\max_{q\in Q}\big(h(q)-\dgrid(q,v)\big).\qedhere
\]
\end{proof}

\begin{definition}[Split vertex counts]\label{def:counts}
For integers $m,n\ge2$ and $a,b\ge1$, the number of height functions on $\Gmn$
with exactly $a$ strict local minima and $b$ strict local maxima is $N_{(a,b)}(m,n)$.
\end{definition}

For integers $m,n\ge2$ and $d\ge2$, the number of height functions
on $\Gmn$ with exactly $d$ distinct extremal cells is $E_d(m,n)$.
By Lemma~\ref{lem:degextrema}, $E_d(m,n)$ counts the degree-$d$
vertices of $\OFG(\Mm)$. The minima and maxima are nonempty by
Lemma~\ref{lem:coneenv} and disjoint because every cell has a
neighbour, so their counts satisfy $a+b=d$. Negation $h\mapsto-h$
exchanges the two sets, giving $N_{(a,b)}(m,n)=N_{(b,a)}(m,n)$.
The indicator $\delta_{ab}$ is one when $a=b$ and zero otherwise.
Summing over the splits and pairing the unequal ones gives
\begin{equation}\label{eq:Ed}
\begin{aligned}
  E_d(m,n)&=\sum_{\substack{a+b=d\\ a,b\ge1}}N_{(a,b)}(m,n)\\
          &=\sum_{\substack{1\le a\le b\\ a+b=d}}
            \big(2-\delta_{ab}\big)\,N_{(a,b)}(m,n).
\end{aligned}
\end{equation}

\subsection{Envelope correspondence}\label{sec:envelope}

The cone formula reconstructs a height from its minima and their values.
For $m,n\ge2$, the converse question is which choices of locations
and values produce a height with exactly those minima.

\begin{definition}[Configuration]\label{def:config}
A configuration on $\Gmn$ is a pair $(A,c)$, where
$A=\{p_1,\dots,p_a\}\subseteq\Gmn$ is a set of $a\ge1$ distinct
cells, the apexes, and $c\colon A\to\Z$ assigns each apex an integer offset.
Apex $p_s=(r_s,\kappa_s)$ has row $r_s$, column $\kappa_s$ and offset $c_s=c(p_s)$.
\end{definition}

The lower envelope of a configuration $(A,c)$ is the function
$E_{A,c}\colon\Gmn\to\Z$ given by
\begin{equation}\label{eq:Eac}
  E_{A,c}(v)=\min_{1\le s\le a}\big(c_s+\dgrid(p_s,v)\big).
\end{equation}

Equation~\eqref{eq:Eac} is a nearest-site comparison using additively
weighted Manhattan distances, with weights $-c_s$~\cite{Aurenhammer1991}.
The indices merely enumerate the set $A$, so relabelling them leaves
the configuration unchanged and introduces no ordering multiplicity.

Arbitrary offsets need not produce a height or make every apex a
minimum. A configuration of two adjacent apexes with offsets zero
has envelope value zero at both, violating the unit edge difference.
Raising the second offset to one makes its cone redundant, since
the triangle inequality makes the first cone no larger at every cell.
The first apex is a lower neighbour of the second, so the second apex
is not a minimum. These examples motivate the parity and activity
conditions.

An apex $p_s$ is \emph{active} if its cone lies strictly below every
other cone at $p_s$. The activity condition requires all apexes
to be active. The parity condition requires the offset difference
for each pair of apexes to have the same parity as their grid distance.

\begin{definition}[Admissible configuration]\label{def:admissible}
A configuration $(A,c)$ is admissible if it satisfies the activity
and parity conditions and its envelope is zero at the root.
In symbols, the three conditions are
\begin{gather}
  c_s<c_t+\dgrid(p_t,p_s)\qquad\text{for all }s,t\text{ with }t\ne s,\label{eq:active}\\
  c_s-c_t\equiv\dgrid(p_s,p_t)\pmod2\qquad\text{for all }s,t,\label{eq:par}\\
  E_{A,c}(1,1)=0.\notag
\end{gather}
\end{definition}

The root condition fixes the additive constant. The remaining
conditions must give unit edge differences and make the strict local
minima exactly the prescribed apexes.

\begin{lemma}[Minima at apexes]\label{lem:minatapex}
For a configuration $(A,c)$ on $\Gmn$, every cell $v$ satisfying
$E_{A,c}(v)<E_{A,c}(w)$ for every grid neighbour $w$ lies in $A$.
\end{lemma}

\begin{proof}
At a cell $v$ satisfying these inequalities, the finite nonempty apex set
$A$ has an apex $p_s$ whose cone attains the envelope at $v$.
Suppose, for contradiction, that $v\ne p_s$. Moving one step from $v$
toward $p_s$ reaches a grid neighbour $w$ with
\[
  \dgrid(p_s,w)=\dgrid(p_s,v)-1.
\]
Since the envelope is at most this cone's value at $w$,
\[
  E_{A,c}(w)\le c_s+\dgrid(p_s,w)
    =c_s+\dgrid(p_s,v)-1
    =E_{A,c}(v)-1<E_{A,c}(v).
\]
This contradicts $E_{A,c}(v)<E_{A,c}(w)$, so $v=p_s$ and $v\in A$.
\end{proof}

\begin{lemma}[Cone parity and unit edge differences]\label{lem:parityvalid}
For a configuration $(A,c)$ on $\Gmn$ satisfying
$c_s-c_t\equiv\dgrid(p_s,p_t)\pmod2$ for all apex indices $s,t$,
the cones $c_s+\dgrid(p_s,\cdot)$ have the same parity at each cell,
and its lower envelope satisfies $|E_{A,c}(u)-E_{A,c}(v)|=1$
on every grid edge $uv$.
\end{lemma}

\begin{proof}
An attaining cone bounds the envelope's change across an edge by one.
Parity prevents equal endpoint values. For grid cells $u,v$, some apex
$p_s$ attains the envelope at $v$ because the apex set is finite and nonempty.
The triangle inequality gives
\begin{align*}
  E_{A,c}(u)&\le c_s+\dgrid(p_s,u)\\
            &\le c_s+\dgrid(p_s,v)+\dgrid(u,v)\\
            &=E_{A,c}(v)+\dgrid(u,v).
\end{align*}
Interchanging $u,v$ bounds $|E_{A,c}(u)-E_{A,c}(v)|$ by $\dgrid(u,v)$.

Modulo two, an integer is congruent to its negative, so absolute values
and minus signs do not change parity. For $p_s=(r_s,\kappa_s)$ and
$p_t=(r_t,\kappa_t)$, substituting in the parity condition and collecting
terms gives
\begin{align*}
  c_s-c_t&\equiv r_s+\kappa_s+r_t+\kappa_t\pmod2,\\
  c_s+r_s+\kappa_s&\equiv c_t+r_t+\kappa_t\pmod2.
\end{align*}
The common residue is $\eta\in\{0,1\}$. At a cell $v=(i,j)$,
every cone satisfies
\begin{align*}
  c_s+\dgrid(p_s,v)&=c_s+|r_s-i|+|\kappa_s-j|\\
                 &\equiv i+j+\eta\pmod2.
\end{align*}
The envelope is one of these cone values, so it has the same parity.
For a grid edge $uv$, the coordinate sum changes by one between its
endpoints, so their envelope values have opposite parities. Their difference is an odd
integer of absolute value at most one, giving
$|E_{A,c}(u)-E_{A,c}(v)|=1$.
\end{proof}

\begin{lemma}[Minima Criterion]\label{lem:activemin}
For a configuration $(A,c)$ with apexes $p_1,\dots,p_a$ and offsets
$c_1,\dots,c_a$ satisfying the parity condition,
$c_s-c_t\equiv\dgrid(p_s,p_t)\pmod 2$ for all $s$ and $t$, an apex $p_s$
is a strict local minimum of $E_{A,c}$ iff it is
active, $c_s<c_t+\dgrid(p_t,p_s)$ for every $t\ne s$.
\end{lemma}

\begin{proof}
Suppose first that $p_s$ is active, so $E_{A,c}(p_s)=c_s$.
By Lemma~\ref{lem:parityvalid}, all cone values
at $p_s$ have the same parity, so activity puts every competing cone
at least $c_s+2$ there. For a neighbour $w$ of $p_s$, the triangle
inequality gives
\[
  c_t+\dgrid(p_t,w)
  \ge c_t+\dgrid(p_t,p_s)-1
  \ge c_s+1\qquad(t\ne s).
\]
Cone $s$ has value $c_s+1$ at $w$, so the envelope also has that
value. Thus every neighbour is one higher, making $p_s$ a strict
local minimum.

\medskip\noindent
Conversely, suppose that $p_s$ is not active. Some competing cone
then has value at most $c_s$ at $p_s$, so the minimum over all cones
is attained by a competing cone. Thus
$E_{A,c}(p_s)=c_r+\dgrid(p_r,p_s)$ for some $r\ne s$.
The apexes are distinct, so $p_r\ne p_s$. Moving one step from
$p_s$ toward $p_r$ reaches a neighbour $w$ with
\begin{align*}
  \dgrid(p_r,w)&=\dgrid(p_r,p_s)-1,\\
  E_{A,c}(w)&\le c_r+\dgrid(p_r,w)=E_{A,c}(p_s)-1<E_{A,c}(p_s).
\end{align*}
Thus $p_s$ is not a strict local minimum of $E_{A,c}$, proving the
reverse direction by contraposition.
\end{proof}

\begin{theorem}[Envelope Structure Theorem]\label{thm:envelope}
For integers $m,n\ge2$, every height function $h$ on
$\Gmn$ is the lower envelope $E_{A,c}$ of a unique admissible
configuration $(A,c)$, whose apex set $A$ is the set of strict local
minima of $h$ and whose offsets are $c_s=h(p_s)$. This gives a bijection
between the height functions on $\Gmn$ and the admissible configurations.
\end{theorem}

\begin{proof}
The maps recover the minima and their values in one direction and
reconstruct the height in the other. For a height function $h$, its strict
local minima $P_h$ and their values form the configuration
$\Phi(h)=(P_h,h|_{P_h})$. Lemma~\ref{lem:coneenv} makes $P_h$
nonempty and reconstructs $h$ at every cell $v$ through
\[
  E_{\Phi(h)}(v)=\min_{p\in P_h}\bigl(h(p)+\dgrid(p,v)\bigr)=h(v).
\]
Along a shortest path between apexes $p_s,p_t\in P_h$, summing the odd
height increments gives $c_s-c_t\equiv\dgrid(p_s,p_t)\pmod2$.
Since the reconstructed envelope equals $h$, every apex is a strict local
minimum of that envelope. Under this parity condition,
Lemma~\ref{lem:activemin} identifies strict-minimum apexes with
active apexes, so $c_s<c_t+\dgrid(p_t,p_s)$ whenever $t\ne s$.
The reconstruction also gives $E_{\Phi(h)}(1,1)=0$ by normalization of $h$.
All three admissibility conditions therefore hold, so $\Phi$ maps height
functions to admissible configurations.

Conversely, the envelope map on admissible configurations is $\Psi(A,c)=E_{A,c}$.
The offsets and grid distances are integers, so the envelope is
integer-valued. Admissibility supplies the parity condition, and
Lemma~\ref{lem:parityvalid} gives
\[
  |E_{A,c}(u)-E_{A,c}(v)|=1
  \qquad\text{on every grid edge }uv.
\]
Together with the root value $E_{A,c}(1,1)=0$ from admissibility, these
are the defining conditions of a height function. Thus $\Psi$ maps
admissible configurations to height functions.

Lemma~\ref{lem:minatapex} makes every strict local minimum of
$E_{A,c}$ an apex. Every apex is active by admissibility, and
Lemma~\ref{lem:activemin} makes every active apex a strict local
minimum under the parity condition. The strict local minima are therefore
exactly $A$. At each apex $p_s$, cone $s$ has value $c_s$, and activity places every
other cone strictly above it, so $E_{A,c}(p_s)=c_s$. The reconstruction of
$h$ and the recovery of both $A$ and $c$ give
\[
  \Psi(\Phi(h))=h,\qquad
  \Phi(\Psi(A,c))=(A,E_{A,c}|_A)=(A,c).
\]
These identities make $\Phi$ and $\Psi$ inverse bijections. Each admissible
configuration with envelope $h$ therefore equals $\Phi(h)$, proving
uniqueness.
\end{proof}

For the configuration in Figure~\ref{fig:kcone-envelope}, the apexes
are $p_1=(1,1)$, $p_2=(1,5)$ and $p_3=(5,1)$ on $G_{5,5}$,
all with offset zero. Its envelope is
\[
  E_{A,c}(v)=\min_{1\le s\le3}\dgrid(p_s,v).
\]
The pairwise apex distances are positive and even, giving activity
and parity. The envelope is zero at the root $(1,1)$, so the
configuration is admissible. Theorem~\ref{thm:envelope} identifies
its strict local minima as exactly the three apexes.

All three cones attain the envelope at $(3,3)$, while only those
based at $p_2$ and $p_3$ attain it at $(4,4)$ and $(5,5)$.
These are precisely the three strict local maxima. Their attaining
cones motivate the directional test of Lemma~\ref{lem:maxima}.

\begin{figure}[H]
\centering
\begin{tikzpicture}[x=1.1cm, y=-1.1cm, font=\small]
  \foreach \gridrow in {1,...,5} {
    \draw[gray!50] (1,\gridrow) -- (5,\gridrow);
  }
  \foreach \gridcolumn in {1,...,5} {
    \draw[gray!50] (\gridcolumn,1) -- (\gridcolumn,5);
  }
  \foreach \i/\j/\h in {
    1/1/0, 1/2/1, 1/3/2, 1/4/1, 1/5/0,
    2/1/1, 2/2/2, 2/3/3, 2/4/2, 2/5/1,
    3/1/2, 3/2/3, 3/3/4, 3/4/3, 3/5/2,
    4/1/1, 4/2/2, 4/3/3, 4/4/4, 4/5/3,
    5/1/0, 5/2/1, 5/3/2, 5/4/3, 5/5/4
  }
    \node[circle, draw=black, fill=white, minimum size=7mm, inner sep=0pt]
      at (\j,\i) {$\h$};
  \foreach \i/\j/\h in {1/1/0, 1/5/0, 5/1/0}
    \node[circle, draw=black, fill=gray!30, minimum size=7mm, inner sep=0pt]
      at (\j,\i) {$\h$};
  \foreach \i/\j/\h in {3/3/4, 4/4/4, 5/5/4}
    \node[rectangle, draw=black, very thick, fill=white,
          minimum size=7mm, inner sep=0pt] at (\j,\i) {$\h$};
  \node[font=\footnotesize, anchor=east] at (0.6,1) {$p_1$};
  \node[font=\footnotesize, anchor=west] at (5.4,1) {$p_2$};
  \node[font=\footnotesize, anchor=east] at (0.6,5) {$p_3$};
\end{tikzpicture}
\caption{Envelope on $G_{5,5}$ with shaded minima and boxed maxima.}
\label{fig:kcone-envelope}
\end{figure}

\subsection{Maxima and one-apex geometry}\label{sec:maxima}

The envelope correspondence identifies the minima. The maximum test
uses the cones attaining the envelope $E=E_{A,c}$ at a cell $v$.
Their index set is
\[
 \Aset(v)=\{s:c_s+\dgrid(p_s,v)=E(v)\}.
\]
A step toward an attaining apex lowers its cone by one. The maximum
test compares these decreases in each existing grid direction.
It is stated under the parity condition alone, without activity or
root normalization.

\begin{lemma}[Maxima Criterion]\label{lem:maxima}
For integers $m,n\ge1$ with $mn\ge2$, if a configuration $(A,c)$ on
$\Gmn$, with apexes $p_s=(r_s,\kappa_s)$ and envelope $E=E_{A,c}$,
satisfies $c_s-c_t\equiv\dgrid(p_s,p_t)\pmod2$ for all apex indices $s,t$,
then a cell $v=(i,j)$ is a strict local maximum of $E$ iff it satisfies
all four conditions
\[
\begin{aligned}
  i<m&\quad\Longrightarrow\quad\exists s\in\Aset(v),\ r_s>i
    &&\text{down},\\
  i>1&\quad\Longrightarrow\quad\exists s\in\Aset(v),\ r_s<i
    &&\text{up},\\
  j<n&\quad\Longrightarrow\quad\exists s\in\Aset(v),\ \kappa_s>j
    &&\text{right},\\
  j>1&\quad\Longrightarrow\quad\exists s\in\Aset(v),\ \kappa_s<j
    &&\text{left}.
\end{aligned}
\]
\end{lemma}

\begin{proof}
Parity prevents a nonattaining cone from lowering the envelope at a
neighbour. For a neighbour $w$ of $v$, the apex set is finite and
nonempty, so at least one cone attains $E(v)$. By
Lemma~\ref{lem:parityvalid}, all cone values at $v$ have the parity
of $E(v)$. A nonattaining cone is therefore at least $E(v)+2$ at $v$ and
at least $E(v)+1$ at $w$, since each cone changes by $\pm1$ on the step.
An attaining cone drops by one when the step moves toward its apex and
rises by one otherwise. If one drops, it attains $E(v)-1$ at $w$ and
no cone lies below that value. If none drops, the attaining cones have
value $E(v)+1$ there and no cone lies lower. Taking the minimum gives
\[
  E(w)=
  \begin{cases}
    E(v)-1,&\text{if some cone attaining at }v\text{ drops on }v\to w,\\
    E(v)+1,&\text{if no cone attaining at }v\text{ drops}.
  \end{cases}
\]

An existing downward step decreases the distance to $p_s$ exactly when
$r_s>i$, and an upward step does so exactly when $r_s<i$.
The horizontal conditions are $\kappa_s>j$ rightward and
$\kappa_s<j$ leftward. If every existing direction is covered by an
attaining cone, each neighbour has value $E(v)-1$, so $v$ is a strict
local maximum. Conversely, an uncovered existing direction gives a
neighbour of value $E(v)+1$, ruling out a strict local maximum and
proving necessity by contraposition.
\end{proof}

\begin{lemma}[Single-cone maxima]\label{cor:corners}
For integers $m,n\ge2$ and a configuration $(A,c)$ on $\Gmn$, a cell
$v\in\Gmn$ with a single attaining cone, $|\Aset(v)|=1$, is a strict local
maximum of $E_{A,c}$ iff $v$ is a grid corner and the apex of the
attaining cone lies on neither of the two boundary lines through $v$.
\end{lemma}

\begin{proof}
Only the attaining cone can make a neighbour lower than $v$.
A nonattaining cone has integer value at least $E(v)+1$ at $v$ and
at least $E(v)$ at a neighbour. The attaining cone drops by one on a
step toward its apex and rises by one otherwise. Thus the envelope
decreases exactly in the directions toward that apex, without needing
parity. A strict maximum requires every existing direction to point
toward that apex. Opposite directions cannot both do so, forcing $v$ to be
a corner. Conversely, both inward directions from a corner point
toward the apex exactly when it lies on neither boundary line through
$v$, and both neighbours then have lower envelope value.
\end{proof}

\begin{corollary}[One-apex geometry]\label{cor:one-apex}
For integers $m,n\ge2$, a height function $h$ on $\Gmn$ with a unique
strict local minimum $p$ is a distance cone. It has a unique strict local
maximum iff $p$ is a corner. It has two strict local maxima when
$p$ is a non-corner boundary cell and four when $p$ is interior.
Negation gives the corresponding statement for a unique strict local
maximum.
\end{corollary}

\begin{proof}
The unique minimum leaves one cone in the formula of
Lemma~\ref{lem:coneenv}, giving
\[
  h(v)=h(p)+\dgrid(p,v).
\]
In this one-cone representation, every cell has a single attaining cone.
For $p=(r,\kappa)$, Lemma~\ref{cor:corners} identifies
the maxima as exactly the corners $(i,j)$ with
\[
  i\in\{1,m\}\setminus\{r\},\qquad
  j\in\{1,n\}\setminus\{\kappa\}.
\]
Each allowed row endpoint can be paired with each allowed column endpoint.
A boundary coordinate of $p$ excludes one endpoint, while an interior
coordinate excludes neither. A corner apex therefore gives $1\cdot1=1$
maximum, a non-corner boundary apex gives $1\cdot2=2$ or $2\cdot1=2$, and
an interior apex gives $2\cdot2=4$. Negation exchanges minima and maxima,
proving the dual statement.
\end{proof}

\subsection{Small-degree counts and flip locations}\label{sec:geometry}

The one-apex description gives the degree-two and degree-three counts.
Its degree-two case is the corner-cone count of~\cite[Theorem~10]{Christensen2025origami}.

\begin{proposition}[Degrees two and three]\label{thm:degree-two}\label{prop:degrees-two-three}
For $m,n\ge2$, the four degree-$2$ vertices of $\OFG(\Mm)$ are
\[
  h_\varepsilon(i,j)=\varepsilon_1(i-1)+\varepsilon_2(j-1),
  \qquad \varepsilon=(\varepsilon_1,\varepsilon_2)\in\{-1,1\}^2.
\]
The number of degree-$3$ vertices is $4(m+n-4)$.
\end{proposition}

\begin{proof}
By Lemmas~\ref{lem:degextrema} and~\ref{lem:coneenv}, a degree-two
height has one minimum and one maximum. Corollary~\ref{cor:one-apex}
makes its minimum a corner and its height the corresponding distance
cone. Normalization fixes its offset. Conversely, each corner cone has
exactly these two extrema. The four corners give distinct minima and
therefore distinct heights. Distance from a corner changes by a fixed sign
along each coordinate. Root normalization removes the constant term and
gives the displayed functions $h_\varepsilon$.

\medskip\noindent
A degree-three height has extrema split $(1,2)$ or $(2,1)$. In the first
case, Corollary~\ref{cor:one-apex} makes the height a normalized distance
cone with its minimum at a non-corner boundary cell. Conversely, each
such cell $p$ determines the function
\[
  h_p(v)=\dgrid(p,v)-\dgrid(p,(1,1)).
\]
Grid distance changes by one across every edge, and $h_p(1,1)=0$.
Every neighbour of $p$ has greater height, while every other cell has a
neighbour one step toward $p$ with smaller height. Thus $h_p$ is a normalized
height with unique strict minimum $p$. Corollary~\ref{cor:one-apex} gives
exactly two strict maxima. Distinct cells $p$ give distinct heights because
their unique minima differ. There are $n-2$ such cells on each horizontal
side and $m-2$ on each vertical side, so the split $(1,2)$ has $2(m+n-4)$
heights. Negation preserves normalization and exchanges minima with maxima,
giving a bijection with the heights of split $(2,1)$. The two families
together have $4(m+n-4)$ members.
\end{proof}

Every height has a strict minimum and maximum, so the four degree-two
vertices are precisely the minimum-flip states.
The degree-two count polynomial is therefore $p_2(m,n)=4$ on $m,n\ge2$.
The small-degree counts do not say where the legal moves occur.
Their geometric refinement distinguishes boundary-only states from
states that permit an interior flip.

\subsubsection{Two-apex envelopes and degree four}

In these classifications, the degree of a height means the degree of its
represented flip-graph vertex, hence its number of extrema by
Lemma~\ref{lem:degextrema}. Both extremum sets are nonempty by
Lemma~\ref{lem:coneenv}. At degree at most five, one set has at most two
cells. Negation exchanges the two sets, so one- and two-apex envelopes
suffice.

\begin{proposition}[Degree-four structure]\label{prop:degree-four-structure}
For $m,n\ge2$, a normalized height function has degree four iff it has two
strict local minima and two strict local maxima. In particular, a degree-$4$
vertex is neither a one-apex envelope nor its negation.
\end{proposition}

\begin{proof}
Suppose first that a normalized height function $h$ has degree four.
Lemmas~\ref{lem:degextrema} and~\ref{lem:coneenv} give the possible
extrema splits $(1,3)$, $(2,2)$ and $(3,1)$. A unique minimum gives
one, two or four maxima by Corollary~\ref{cor:one-apex}, excluding
$(1,3)$. Applying the same fact to $-h$ excludes $(3,1)$, so the split
is $(2,2)$.

\medskip\noindent
Conversely, a normalized height with two strict local minima and two strict
local maxima has four distinct extremal cells. Lemma~\ref{lem:degextrema}
therefore gives degree four. Having two extrema of each kind excludes a
one-apex envelope and its negation.
\end{proof}

The envelope correspondence records a height with two strict local minima
by those cells and their values.
For an integer $b\ge1$, the set $\mathcal H_{2,b}$ consists of normalized height
functions on $\Gmn$ with two strict local minima and $b$ strict local maxima.
The set $\mathcal C_{2,b}$ consists of admissible configurations $(A,c)$ on $\Gmn$ with
$|A|=2$ whose envelope $E_{A,c}$ has $b$ strict local maxima.
The strict local minima of a height function $h$ form the set $P_h$.

\begin{corollary}[Two-apex correspondence]\label{prop:two-apex}
For integers $m,n\ge2$ and $b\ge1$, the following maps are mutually
inverse bijections between $\mathcal H_{2,b}$ and $\mathcal C_{2,b}$, given by
\[
  h\longmapsto \bigl(P_h,h|_{P_h}\bigr)
  \qquad\text{and}\qquad
  (A,c)\longmapsto E_{A,c}.
\]
\end{corollary}

\begin{proof}
Theorem~\ref{thm:envelope} sends each $h\in\mathcal H_{2,b}$ to the
admissible configuration $(P_h,h|_{P_h})$ with two apexes and envelope $h$,
so its image lies in $\mathcal C_{2,b}$. Conversely, it sends each
$(A,c)\in\mathcal C_{2,b}$ to a normalized height $E_{A,c}$ whose minima
are exactly $A$ and whose apex values are $c$. Its two minima and $b$
maxima place it in $\mathcal H_{2,b}$. The reconstruction and recovery give
\[
  E_{P_h,h|_{P_h}}=h,\qquad
  \bigl(A,E_{A,c}|_A\bigr)=(A,c).
\]
Thus both composites are the identity, proving the restricted bijection.
\end{proof}

Figure~\ref{fig:conepair} illustrates the correspondence with two strict
local minima and two strict local maxima.

\begin{figure}[!ht]
\centering
\begin{tikzpicture}[x=1.1cm, y=-1.1cm, font=\small]
  \foreach \gridrow in {1,...,3} {
    \draw[gray!50] (1,\gridrow) -- (3,\gridrow);
  }
  \foreach \gridcolumn in {1,...,3} {
    \draw[gray!50] (\gridcolumn,1) -- (\gridcolumn,3);
  }
  \foreach \i/\j/\h in {1/1/0, 1/2/1, 1/3/2, 2/1/1, 2/2/2, 2/3/3,
                          3/1/2, 3/2/1, 3/3/2}
    \node[circle, draw=black, fill=white, minimum size=7mm, inner sep=0pt]
      at (\j,\i) {$\h$};
  \foreach \i/\j/\h in {1/1/0, 3/2/1}
    \node[circle, draw=black, fill=gray!30, minimum size=7mm, inner sep=0pt]
      at (\j,\i) {$\h$};
  \foreach \i/\j/\h in {2/3/3, 3/1/2}
    \node[rectangle, draw=black, very thick, fill=white,
          minimum size=7mm, inner sep=0pt] at (\j,\i) {$\h$};
  \node[font=\footnotesize, anchor=east] at (0.6,1) {$p_1$};
  \node[font=\footnotesize, anchor=north] at (2,3.6) {$p_2$};
\end{tikzpicture}
\caption{Two-apex envelope on $G_{3,3}$ with $p_1=(1,1)$,
$p_2=(3,2)$ and offsets $0,1$. The two strict local minima are shaded
and the two strict local maxima are boxed, giving four legal flips.}
\label{fig:conepair}
\end{figure}

\begin{proposition}[Degree-four boundary property]\label{prop:degree-four-boundary}
For integers $m,n\ge2$, every strict local extremum of a degree-$4$
height function on $\Gmn$ lies on the grid boundary.
\end{proposition}

\begin{proof}
An interior minimum would force at least three maxima, contradicting the
degree-four split. Proposition~\ref{prop:degree-four-structure}
and Corollary~\ref{prop:two-apex} give $h=E_{A,c}$ with two strict minima, the apexes
of $A$, and two strict maxima. Suppose that an apex $p_I$ is interior.

\medskip\noindent
\emph{Two interior apexes.} If the other apex is interior, each inward step
from a grid corner moves toward both apexes. Lemma~\ref{lem:maxima} makes
all four corners strict maxima, a contradiction.

The other apex $p_B$ is therefore on the boundary. Grid reflections,
transposition and root renormalization preserve extrema and reduce its
position to the two cases
\[
  p_B=(1,1)\qquad\text{or}\qquad p_B=(1,t),\quad 2\le t\le n-1.
\]
In these coordinates $p_I=(r,s)$. The corresponding cones are
$C_I(v)=c(p_I)+\dgrid(p_I,v)$ and $C_B(v)=c(p_B)+\dgrid(p_B,v)$.

\medskip\noindent
\emph{Boundary corner apex.} Suppose that $p_B=(1,1)$. The offset difference is
$\delta=c(p_B)-c(p_I)$. Parity gives
$\delta\equiv r+s-2\pmod2$, while activity at $p_I$ gives
$\delta+r+s-2>0$. Thus
\[
  q=\delta+r+s-2\ge2.
\]
The two boundary cells $(r,n)$ and $(m,s)$ have cone gaps
\[
\begin{aligned}
  C_B(r,n)-C_I(r,n)&=\delta+(r+n-2)-(n-s)=q,\\
  C_B(m,s)-C_I(m,s)&=\delta+(m+s-2)-(m-r)=q.
\end{aligned}
\]
At $(r,n)$, the interior cone rises by one on each vertical step.
At $(m,s)$, it rises by one on each horizontal step. On these steps,
the other cone changes by one from a value at least two higher.
The envelope therefore also rises by one.

The segments $\{(i,n):1\le i<r\}$ and $\{(m,j):1\le j<s\}$
are nonempty because $p_I$ is interior. Each finite segment has a maximum.
The unit edge differences make each such maximum strict along its segment.
The inequalities just proved give $h(r-1,n)>h(r,n)$ and
$h(m,s-1)>h(m,s)$, so these cells are strict maxima along
their full boundary edges. A step left from the last column or up
from the last row lowers both cones by one. Thus both cells
are grid maxima. Both inward steps from $(m,n)$ also lower both
cones, making that corner a third maximum. The two segments are
disjoint and exclude $(m,n)$, so the three maxima are distinct.

\medskip\noindent
\emph{Boundary non-corner apex.} Suppose that $p_B=(1,t)$. Both bottom
corners are strict maxima because each inward
step lowers both cones. If they were the only maxima,
Lemma~\ref{lem:coneenv} would give the dual envelope
\[
  h(i,j)=\max\bigl\{
    h(m,1)-(m-i)-(j-1),\;
    h(m,n)-(m-i)-(n-j)
  \bigr\}.
\]
Both expressions increase by one when $i$ increases by one, so
$h(i+1,j)=h(i,j)+1$ for $i<m$. Every cell below the top row then has a lower
neighbour above, contradicting the interior minimum $p_I$.

An interior minimum is impossible. Negation exchanges minima and maxima, so
applying the argument to $-h$ excludes an interior maximum.
\end{proof}

\subsubsection{Degree five and interior extrema}

\begin{proposition}[Degree-five interior structure]\label{prop:degree-five-interior}
For integers $m,n\ge2$, a degree-$5$ height function on $\Gmn$ has
extrema split $(1,4)$, $(2,3)$, $(3,2)$, or $(4,1)$.
In the split $(1,4)$, its unique strict minimum is interior, and every
interior one-apex envelope has this split. In the split $(2,3)$, at most one
strict minimum is interior. The corresponding statements for maxima hold after
negation.
\end{proposition}

\begin{proof}
The numbers of minima and maxima are positive and sum to five, giving the
four splits. Corollary~\ref{cor:one-apex} shows that a height with one minimum
has four maxima exactly when that minimum is interior. Conversely, an interior
one-apex envelope has its apex as its unique minimum and four corner maxima,
giving split $(1,4)$.

Suppose a height with split $(2,3)$ had two interior minima. Its envelope
from Theorem~\ref{thm:envelope} has these minima as its two apexes. From each
corner, a step along either incident edge moves toward both apexes. Both cone
values drop by one, so their minimum also drops by one. All four corners are
therefore strict local maxima, contradicting the split. Negation exchanges
minima and maxima and proves the dual assertions.
\end{proof}

The split $(2,3)$ can have an interior minimum or an interior maximum.
On $G_{3,3}$, offset $-2$ at each apex gives admissible configurations
for the apex sets
\[
  \{(1,3),(2,2)\},\qquad \{(1,3),(3,1)\}.
\]
Theorem~\ref{thm:envelope} identifies their apexes as their minima, so the
first attains the bound with an interior minimum at $(2,2)$.
Lemma~\ref{lem:maxima} gives three maxima in each case and an interior
maximum at $(2,2)$ in the second.

\begin{proposition}[Low-degree flip locations]\label{prop:low-degree-flips}
For integers $m,n\ge2$, every degree-three or degree-four state has all
legally flippable faces on the grid boundary. A degree-five height with
two strict minima has at most one interior minimum, and negation gives
the corresponding restriction for maxima. If $m,n\ge3$, a distance cone
based at an interior cell has degree five and an interior legal flip.
\end{proposition}

\begin{proof}
At degree two, both extrema are corners by
Proposition~\ref{prop:degrees-two-three}.
For degree three, Corollary~\ref{cor:one-apex} makes one of the height
and its negative a distance cone with a non-corner boundary apex and
two corner maxima. All three extrema lie on the boundary.
Proposition~\ref{prop:degree-four-structure} gives two minima and two
maxima at degree four, and Proposition~\ref{prop:degree-four-boundary}
places all four on the boundary. Lemma~\ref{lem:degextrema} identifies
these extrema with the legal flips.

\medskip\noindent
Proposition~\ref{prop:degree-five-interior} gives the stated restriction
on a pair of minima. An interior distance cone has one minimum and four
corner maxima by Corollary~\ref{cor:one-apex}, so it has five legal
flips, including its interior apex. Thus five is the first possible
number of legal flips permitting an interior flip when $m,n\ge3$.
\end{proof}

Degree five does not force an interior flip. On the $3\times4$ grid, the
following normalized height has minima at $(1,1)$, $(1,3)$ and $(2,4)$
and maxima at $(1,4)$ and $(3,2)$. All five legal flips lie on the boundary,
whereas an interior cone on the same grid has an interior flip. The
boundary-only example is
\[
\begin{pmatrix}
0&1&0&1\\
1&2&1&0\\
2&3&2&1
\end{pmatrix}.
\]

\section{Polynomial degree counts}\label{sec:counts}

Heights whose minima lie on one boundary side give an explicit lower
bound for $E_d(m,n)$.
A column transfer bounds the polynomial degree in one dimension, uniformly
in the other. Interpolation then combines the two one-variable laws.

\subsection{Path walks and single-side counts}\label{sec:onedim}

A configuration is single-side if it is admissible and all its apexes
lie on one boundary side. The envelope correspondence reduces these
heights to path walks. Their count supplies the separate-degree lower
bound and the single-side leading terms in the total-degree argument.

\begin{lemma}[Side reduction]\label{lem:edgered}
For $m,n\ge2$, if an admissible configuration $(A,c)$ has all its $a$ apexes
on the top row at columns $\kappa_1<\dots<\kappa_a$, then
$E_{A,c}(i,j)=(i-1)+\tau(j)$, where
$\tau(j)=\min_s(c_s+|j-\kappa_s|)$ on $P_n$. The envelope increases strictly
down each column. Its minima are the row-$1$ apexes, and its maxima are the
bottom-row cells at the strict local maxima of $\tau$. The number of such
normalized configurations is the same for every $m\ge2$.
\end{lemma}

\begin{proof}
Every cone has the same vertical summand $i-1$, so
\[
  E_{A,c}(i,j)=\min_s\big(c_s+(i-1)+|j-\kappa_s|\big)=(i-1)+\tau(j).
\]
By Theorem~\ref{thm:envelope}, the envelope is a normalized height
function with minima exactly at the apexes. Its top row $\tau$ therefore
satisfies $\tau(1)=0$ and has unit steps.
The vertical step is $+1$ down each column. A top-row cell is therefore
a minimum exactly when its column is a minimum of $\tau$.
The minima of $\tau$ are thus precisely the apex columns.
A maximum must lie on the bottom row. There the upward neighbour is lower, and the remaining
maximum conditions are exactly those for a strict local maximum of $\tau$.

\medskip\noindent
For top-row apexes, the pairwise distances are $|\kappa_s-\kappa_t|$.
Activity and parity therefore constrain only the columns and offsets.
The root condition is $\tau(1)=0$. None of these conditions depends
on $m$, so the same top-row data are admissible for every $m\ge2$.
Row reversal, $i\mapsto m+1-i$, followed by subtraction of the new root value
gives the bottom-row bijection, preserving extrema. Transposition gives
the column versions with $m$ and $n$ exchanged.
\end{proof}

Counting the path walks in the side-reduction bijection therefore counts
the four side families. Binomial coefficients are zero when the lower index exceeds the
nonnegative upper index.

\begin{lemma}[Walk count]\label{lem:binom}
For integers $n\ge2$ and $a,b\ge1$ with $d=a+b$, the count $f_{(a,b)}(n)$ is
the number of normalized $\pm1$ walks $h$ on the $n$-cell path $P_n$ with $h(1)=0$, exactly $a$ strict
local minima and $b$ strict local maxima, including endpoints as extrema.
If $|a-b|>1$, then $f_{(a,b)}(n)=0$. Otherwise,
\begin{equation}\label{eq:fclosed}
  f_{(a,b)}(n)=\big(1+\delta_{ab}\big)\binom{n-2}{d-2}.
\end{equation}
\end{lemma}

\begin{proof}
A normalized walk is determined by its initial direction and the lengths of
its maximal monotone runs. The two endpoints and the turns between successive
runs are precisely its extrema, so $d$ extrema give $d-2$ turns and $d-1$ runs.
These extrema alternate between minima and maxima, ruling out $|a-b|>1$.
When $a=b$, the alternating pattern permits both initial directions.
When $|a-b|=1$, the first extremum must have the more frequent type, fixing
the initial direction. Figure~\ref{fig:walkcount} illustrates the balanced case.

\medskip\noindent
For each allowed initial direction, the positive run lengths sum to $n-1$.
Choosing the $d-2$ turn positions from the $n-2$ gaps between consecutive
steps determines these lengths. The chosen turns and initial direction
give the prescribed split. The zero initial height fixes the resulting
walk, whose turns recover the choice. Multiplying by the $1+\delta_{ab}$
initial directions gives~\eqref{eq:fclosed}.
\end{proof}

\begin{figure}[H]
\centering
\resizebox{\linewidth}{!}{%
\begin{tikzpicture}[x=0.75cm, y=0.5cm, font=\small]
  \begin{scope}
    \draw[->, gray!70!black] (0.5, 0) -- (8.7, 0);
    \node[font=\footnotesize, gray!70!black, anchor=west] at (8.7, 0) {$j$};
    \draw[->, gray!70!black] (0.5, -2.3) -- (0.5, 2.9);
    \node[font=\footnotesize, gray!70!black, anchor=south] at (0.5, 2.9) {$h$};
    \foreach \j in {1,...,8} {
      \draw[gray!70!black] (\j, 0.1) -- (\j, -0.1);
      \node[font=\footnotesize, gray!70!black, anchor=north] at (\j, -0.15) {\j};
    }
    \foreach \y in {-2,-1,1,2} {
      \draw[gray!70!black] (0.4, \y) -- (0.6, \y);
      \node[font=\footnotesize, gray!70!black, anchor=east] at (0.35, \y) {\y};
    }
    \node[font=\footnotesize, gray!70!black, anchor=east] at (0.35, 0) {0};
    \draw[decorate, decoration={brace, amplitude=3pt}, thick, gray!70!black]
      (1, 2.5) -- (3, 2.5) node[midway, above=1pt, font=\scriptsize] {$\ell_1{=}2$};
    \draw[decorate, decoration={brace, amplitude=3pt}, thick, gray!70!black]
      (3, 2.5) -- (6, 2.5) node[midway, above=1pt, font=\scriptsize] {$\ell_2{=}3$};
    \draw[decorate, decoration={brace, amplitude=3pt}, thick, gray!70!black]
      (6, 2.5) -- (8, 2.5) node[midway, above=1pt, font=\scriptsize] {$\ell_3{=}2$};
    \draw[very thick, black]
      (1,0) -- (2,1) -- (3,2) -- (4,1) -- (5,0) -- (6,-1) -- (7,0) -- (8,1);
    \foreach \x/\y/\h in {1/0/0, 6/-1/-1}
      \node[circle, draw=black, fill=gray!30, minimum size=6mm, inner sep=0pt] at (\x,\y) {$\h$};
    \foreach \x/\y/\h in {3/2/2, 8/1/1}
      \node[rectangle, draw=black, very thick, fill=white, minimum size=6mm, inner sep=0pt] at (\x,\y) {$\h$};
    \node[font=\footnotesize, anchor=north] at (4.5, -2.9) {(a) $h$, first run up};
  \end{scope}
  \begin{scope}[xshift=9.3cm]
    \draw[->, gray!70!black] (0.5, 0) -- (8.7, 0);
    \node[font=\footnotesize, gray!70!black, anchor=west] at (8.7, 0) {$j$};
    \draw[->, gray!70!black] (0.5, -2.3) -- (0.5, 2.9);
    \node[font=\footnotesize, gray!70!black, anchor=south] at (0.5, 2.9) {$h$};
    \foreach \j in {1,...,8} {
      \draw[gray!70!black] (\j, 0.1) -- (\j, -0.1);
      \node[font=\footnotesize, gray!70!black, anchor=north] at (\j, -0.15) {\j};
    }
    \foreach \y in {-2,-1,1,2} {
      \draw[gray!70!black] (0.4, \y) -- (0.6, \y);
      \node[font=\footnotesize, gray!70!black, anchor=east] at (0.35, \y) {\y};
    }
    \node[font=\footnotesize, gray!70!black, anchor=east] at (0.35, 0) {0};
    \draw[decorate, decoration={brace, amplitude=3pt}, thick, gray!70!black]
      (1, 2.5) -- (3, 2.5) node[midway, above=1pt, font=\scriptsize] {$\ell_1{=}2$};
    \draw[decorate, decoration={brace, amplitude=3pt}, thick, gray!70!black]
      (3, 2.5) -- (6, 2.5) node[midway, above=1pt, font=\scriptsize] {$\ell_2{=}3$};
    \draw[decorate, decoration={brace, amplitude=3pt}, thick, gray!70!black]
      (6, 2.5) -- (8, 2.5) node[midway, above=1pt, font=\scriptsize] {$\ell_3{=}2$};
    \draw[very thick, black]
      (1,0) -- (2,-1) -- (3,-2) -- (4,-1) -- (5,0) -- (6,1) -- (7,0) -- (8,-1);
    \foreach \x/\y/\h in {3/-2/-2, 8/-1/-1}
      \node[circle, draw=black, fill=gray!30, minimum size=6mm, inner sep=0pt] at (\x,\y) {$\h$};
    \foreach \x/\y/\h in {1/0/0, 6/1/1}
      \node[rectangle, draw=black, very thick, fill=white, minimum size=6mm, inner sep=0pt] at (\x,\y) {$\h$};
    \node[font=\footnotesize, anchor=north] at (4.5, -2.9) {(b) $-h$, first run down};
  \end{scope}
\end{tikzpicture}
}
\caption{Two normalized walks on $P_8$ with run lengths $(2,3,2)$, related
by negation. Circles mark minima and squares mark maxima. Both walks have
two of each, illustrating the two initial directions when $a=b$.}
\label{fig:walkcount}\label{app:walk-count}
\end{figure}

\label{app:counting-details}

For a split $(a,b)$, the nonnegative count $R_{(a,b)}(m,n)$ counts
heights outside the four single-side families. When $|a-b|\le1$,
the walk polynomial has leading coefficient
$f_{\mathrm{lead}}=(1+\delta_{ab})/(d-2)!$, where $d=a+b$.

\begin{lemma}[Single-side decomposition]\label{lem:splitedge}
For integers $a,b\ge1$ with $a+b=d\ge3$ and $m,n\ge2$,
the top and bottom single-side families have combined count
$2f_{(a,b)}(n)$, and the left and right families have combined count
$2f_{(a,b)}(m)$. Their decomposition is
\[
  N_{(a,b)}(m,n)=2f_{(a,b)}(n)+2f_{(a,b)}(m)+R_{(a,b)}(m,n).
\]

The single-side term is nonzero only when $|a-b|\le1$, in which case
its total-degree-$(d-2)$ part is
\[
  2f_{\mathrm{lead}}\,(m^{d-2}+n^{d-2}).
\]
\end{lemma}

\begin{proof}
Restriction to the top row and extension by unit vertical steps give
inverse maps. For a top-row configuration of split $(a,b)$,
Lemma~\ref{lem:edgered} makes the restriction a normalized walk
$\tau$ with $a$ minima and $b$ maxima. Conversely, such a walk extends to
\[
 h_\tau(i,j)=(i-1)+\tau(j).
\]
Its vertical increments are $+1$, its horizontal increments are $\pm1$
and its root value is zero. Its minima lie on the top row over the
minima of $\tau$, and its maxima lie on the bottom row over the
maxima of $\tau$. Theorem~\ref{thm:envelope} therefore gives the unique
admissible configuration with these minima and their values.
Restricting the extension returns $\tau$. Starting with a configuration,
the side-reduction formula recovers its envelope, so uniqueness returns
the configuration. The root-normalized row reversal and transposition
in Lemma~\ref{lem:edgered} give the other three side counts.

\medskip\noindent
An overlap between two side families would put every apex on both
sides. Distinct sides have at most one common cell, so this forces a
single corner apex. Its distance cone has one maximum at the opposite
corner, giving $d=2$. The four families are therefore disjoint for
$d\ge3$, proving the decomposition. By Lemma~\ref{lem:binom}, their
count vanishes when $|a-b|>1$. Otherwise the leading coefficient of
each walk polynomial is $f_{\mathrm{lead}}$, giving the stated
homogeneous part without a mixed monomial.
\end{proof}

Summing the single-side terms over the splits gives their exact contribution
to $E_d$. This sum and its leading coefficient are
\[
  S_d(m,n)=4\binom{m-2}{d-2}+4\binom{n-2}{d-2},
  \qquad C(d)=\frac{4}{(d-2)!}.
\]

\begin{proposition}[Single-side contribution]\label{prop:singleside}
For $d\ge3$ and $m,n\ge2$, the number of single-side configurations with $d$
extrema is $S_d(m,n)$. Its total-degree-$(d-2)$ homogeneous part is
$C(d)(m^{d-2}+n^{d-2})$.
\end{proposition}

\begin{proof}
Only the balanced or adjacent splits contribute to the walk count.
For even $d$, Lemma~\ref{lem:binom} gives the single split
$a=b=d/2$, with two initial directions. For odd $d$, the two ordered
splits with $|a-b|=1$ each permit one initial direction. In either case,
\[
 \sum_{\substack{a,b\ge1\\a+b=d}}f_{(a,b)}(n)
 =2\binom{n-2}{d-2}.
\]
The ordered split sum in~\eqref{eq:Ed} counts every height once.
Summing the side counts $2f_{(a,b)}(n)+2f_{(a,b)}(m)$ from
Lemma~\ref{lem:splitedge} therefore gives $S_d$.
Each binomial polynomial has degree $d-2$ and leading coefficient
$1/(d-2)!$ in its own variable. The coefficient $4$ in $S_d$ gives the
homogeneous part $C(d)(m^{d-2}+n^{d-2})$, with no mixed term.
\end{proof}

The side-reduction bijection identifies these configurations with the
four height families $h(i,j)=\pm(i-1)+\tau(j)$ and their transposes.
Their combined count gives the lower bound
\[
 E_d(m,n)\ge S_d(m,n).
\]

\subsection{Column transfer and rotation}\label{sec:period}

For a fixed integer $m\ge2$, the variables $x$ and $z$ mark extrema and
columns. A column is a proper three-colouring $\sigma$ of the path with
$m$ cells. A pair $(\sigma,\tau)$ is compatible when its entries differ
in every row. The rows and columns of $T_m(x)$ are indexed by these
compatible pairs. The matrix has a transition
$(\sigma,\tau)\to(\tau,\upsilon)$ of weight $x^e$, where $e$ counts
extrema in the middle column. Every other entry is zero.

The exponent counts a cell exactly when all its present neighbours have the
same colour. For a lifted height, those neighbours are all one unit
above or all one unit below the cell, since the two possible increments
have different residues modulo three. The same test, with only the
neighbours present at the boundary, defines the endpoint counts.

The extrema counts in the left and right columns of a boundary pair are
$\ell(\sigma,\tau)$ and $r(\sigma,\tau)$. Brackets denote indicators. The endpoint vectors are
\[
 \mathbf u_{(\sigma,\tau)}=[\sigma_1=0]x^{\ell(\sigma,\tau)},
 \qquad \mathbf v_{(\sigma,\tau)}=x^{r(\sigma,\tau)}.
\]
The polynomial $c_{m,n}(x)$ enumerates normalized heights by their numbers of extrema.

\begin{lemma}[Column transfer]\label{lem:ratGF}
For integers $m,n\ge2$, the enumerator and width series satisfy
\begin{equation}\label{eq:transfer-enumerator}
\begin{aligned}
 c_{m,n}(x)&=\mathbf u^{\!\top}T_m(x)^{n-2}\mathbf v,\\
 \sum_{n\ge2}c_{m,n}(x)z^n&=z^2\mathbf u^{\!\top}(I-zT_m(x))^{-1}\mathbf v.
\end{aligned}
\end{equation}
\end{lemma}

\begin{proof}
The consecutive-column pairs of a pinned colouring form a unique path
of $n-2$ transitions. Conversely, the overlapping pairs of an allowed
path recover its column sequence, and the initial vector pins the root
colour to zero. Proposition~\ref{prop:height-model}
therefore identifies the contributing paths with normalized heights.
The $n-2$ transitions count the interior columns once, while the endpoint
vectors count the two boundary columns. Their weights multiply to
$x^{\#\text{extrema}(h)}$. At $n=2$ the path consists of one state
and no transition. The identity matrix leaves the two endpoint weights
to count both columns. Summing the path weights gives
\[
  c_{m,n}(x)=\mathbf u^{\!\top}T_m(x)^{\,n-2}\mathbf v.
\]
The matrix geometric series in $\Z[x][[z]]$ gives the width series,
with $z^2$ shifting the exponent from $n-2$ to $n$.
\end{proof}

A compatible pair is frozen when $\tau=\sigma+k$ for a constant
nonzero residue $k$. Its height columns differ by a constant signed
unit step, since a unit increment has a unique nonzero residue.
Global colour rotation on pairs is $R$, and
$\mathcal P=(I+R+R^2)/3$ is its symmetric averaging projection. Global rotation preserves every extremum test, so each coefficient of
$T_m(x)$ commutes with $\mathcal P$, and $\mathcal P\mathbf v=\mathbf v$. Averaging the
left endpoint therefore leaves every scalar count unchanged. The averaged
endpoint is
\[
 \bar{\mathbf u}=\mathcal P\mathbf u,
 \qquad \bar{\mathbf u}_{(\sigma,\tau)}=x^{\ell(\sigma,\tau)}/3.
\]
Every rotation orbit has three states, exactly one with the pinned
first colour. The factor one third therefore counts each normalized
height once. The count block is the subspace over $\Q$ of vectors
constant on rotation orbits. All scalar products retain the original
coloured-pair coordinates, rather than one coordinate per orbit.

\subsection{Frozen structure, endpoints and uniform onset}

\begin{lemma}[Frozen classification]\label{lem:frozen}
For $m\ge2$, a proper triple of consecutive columns $a,b,c$ has no
extremum in $b$ iff, for one $k\in\{1,2\}$,
$a=b-k$ and $c=b+k$ modulo three in every row.
\end{lemma}

\begin{proof}
A row whose horizontal neighbours agree forces that agreement to
continue unless an extremum occurs. Suppose first that $b$ has no
extremum. The horizontal offsets are $\alpha_i=a_i-b_i$ and
$\gamma_i=c_i-b_i$, and its vertical steps are $s_i=b_{i+1}-b_i$,
all modulo three. Properness of the outer columns gives
\begin{equation}\label{eq:uwproper}
 \alpha_{i+1}\ne\alpha_i-s_i,
 \qquad \gamma_{i+1}\ne\gamma_i-s_i.
\end{equation}
All offsets and steps are nonzero. A row is aligned when its two
offsets agree. For an aligned row, the extremum tests are
\begin{equation}\label{eq:extremum}
 \begin{cases}
 s_{i-1}=-\alpha_i\text{ and }s_i=\alpha_i,&1<i<m,\\
 s_1=\alpha_1,&i=1,\\
 s_{m-1}=-\alpha_m,&i=m.
 \end{cases}
\end{equation}
If an aligned row has offsets $r$ and the following step is $-r$,
both next offsets must be $r$, since the properness inequalities
exclude the other nonzero residue. The same inequalities imply that
a preceding step $r$ forces both preceding offsets to be $r$.

\medskip\noindent
Suppose that an aligned row exists, and take a maximal consecutive
block with the same offsets $r$. At its first row, a preceding step
$r$ would extend the block upward, so an upper neighbour, when
present, has the common horizontal colour. Avoiding an extremum
requires a lower neighbour of the other colour. Its step is $-r$,
which extends the aligned block downward and gives the next row an
upper neighbour of the common colour. Repeating this argument forces
a row beyond the block or beyond the grid, a contradiction.

\medskip\noindent
Thus the horizontal neighbours differ in every row, giving
$\gamma_i=-\alpha_i$. The second properness inequality now excludes
$\alpha_{i+1}=\alpha_i+s_i$. The two excluded residues are distinct
and different from $\alpha_i$, so $\alpha_{i+1}=\alpha_i$.
With $k=-\alpha_1$, the columns therefore satisfy $a=b-k$ and $c=b+k$
in every row. Conversely, these constant offsets give different
horizontal neighbour colours in every row, excluding an extremum.
\end{proof}

\begin{lemma}[Boundary extremum]\label{lem:boundary}
For $m,n\ge2$, each maximal constant block of the horizontal offsets
between a boundary column of a height and its adjacent column contains
an extremum of the boundary column. Each boundary column therefore has
an extremum. If it has exactly one, the adjacent column differs from
it by a constant $1$ or $-1$.
\end{lemma}

\begin{proof}
An extremum-free block would force its vertical steps to keep one
sign, contradicting the step at its far end. The left boundary
columns have heights $h_i,h'_i$, with
$s_i=h_{i+1}-h_i$ and $\varepsilon_i=h'_i-h_i$ in $\{-1,1\}$.
With rows indexed from $1$ to $m$, a boundary cell is an extremum precisely when
\begin{equation}\label{eq:boundchar}
 \begin{cases}
 s_1=\varepsilon_1,&i=1,\\
 -s_{i-1}=s_i=\varepsilon_i,&1<i<m,\\
 s_{m-1}=-\varepsilon_m,&i=m.
 \end{cases}
\end{equation}
On a maximal block $[a,b]$ with $\varepsilon_i=r$, properness of the second
column forces $s_{a-1}=-r$ at entry and
$s_b=r$ at exit, whenever those neighbours exist. Indeed,
the corresponding second-column increments are
$s_{a-1}+2r$ and $s_b-2r$ and have unit magnitude.

\medskip\noindent
If the block had no extremum, its first cell's horizontal increment
$r$ and upper increment $r$, when present, would force its lower
increment to be $-r$. This forces the next cell's upper increment to
be $r$ and propagates through the block. At its last cell the exit
increment is $r$, or there is no lower neighbour. All present increments
are then $r$, a contradiction. Every block contains an extremum, so
one extremum permits only one block. Reflection gives the right boundary.
\end{proof}

The diagonal projection $\Pi$ keeps frozen pairs and kills all other
pairs. It preserves the count block. Its complementary projection and
the positive-degree transfer are
\[
 Q=I-\Pi,\qquad
 V(x)=T_m(x)-T_m(0)=\sum_{j\ge1}x^jB_j.
\]

\begin{lemma}[Rotation quotient]\label{lem:quotient}
On the count block, the zero-extremum transfer equals $\Pi$. For $m\ge2$ and $d\ge2$, the width series
$\sum_{n\ge2}E_d(m,n)z^n$ has no finite pole other than one and has
eventually polynomial coefficients.
\end{lemma}

\begin{proof}
Lemma~\ref{lem:frozen} gives the unique zero-extremum successor of a
frozen pair as its colour rotation by the pair's slope. The three-step
cycle is precisely that rotation orbit. Nonfrozen states have neither
an incoming nor an outgoing zero-extremum transition. On orbit-constant
vectors the frozen cycles act as the identity and the remaining
components vanish. This proves the projection assertion.

\medskip\noindent
On the count block, the zero-extremum resolvent is
\[
 R_0=(I-z\Pi)^{-1}=\frac{I-zQ}{1-z}.
\]
The factorization and its ordered inverse expansion are
\[
\begin{aligned}
 I-zT_m(x)&=(I-z\Pi)(I-R_0zV(x)),\\
 (I-zT_m(x))^{-1}&=\sum_{k\ge0}(R_0zV(x))^kR_0.
\end{aligned}
\]
These are identities in $\Q[z,(1-z)^{-1}][[x]]$, since $V(x)$ has
positive $x$-degree. The matrix factors retain their displayed order.
Expanding the coefficients at $z=0$ gives the same inverse in
$\Q[[x,z]]$ as the original transfer series, by uniqueness of the
inverse there. At a fixed extremum degree, only finitely many terms
contribute, each with denominator a power of $1-z$. Taking the endpoint
scalar product and extracting $x^d$ therefore gives a rational width
series with no other finite pole. Partial fractions give an eventual
polynomial in the width.
\end{proof}

Before averaging, frozen cycles carry period-$3$ modes. The count block
retains only their fixed components, giving coefficient polynomials
of period $1$ rather than a nonconstant periodic dependence on width.
Their degree and onset still depend on the order of the surviving pole
and the numerator degree. Both are controlled by the endpoint weights.

The endpoint expansions are
$\bar{\mathbf u}=\sum_{i\ge1}x^i\bar{\mathbf u}_i$,
$\mathbf v=\sum_{i\ge1}x^i\mathbf v_i$.
The endpoint degrees are positive by Lemma~\ref{lem:boundary}.

\begin{lemma}[Frozen single-extremum boundary]\label{lem:frozenbdy}
On the count block,
$\bar{\mathbf u}_1^{\!\top}Q=0$ and $Q\mathbf v_1=0$.
\end{lemma}

\begin{proof}
A single boundary extremum forces a frozen pair by
Lemma~\ref{lem:boundary}. The two coefficient vectors are therefore
supported on frozen states and constant on rotation orbits, with the
left coefficient retaining its factor one third. Since $\Pi$ keeps
exactly the frozen pairs,
\begin{equation}\label{eq:frozenid}
 \bar{\mathbf u}_1^{\!\top}\Pi=\bar{\mathbf u}_1^{\!\top},
 \qquad \Pi\mathbf v_1=\mathbf v_1.
\end{equation}
Since $Q=I-\Pi$, subtracting each identity gives the endpoint cancellation.
\end{proof}

\begin{lemma}[Uniform onset and degree]\label{lem:uniform}\label{app:uniform-onset}
For integers $d\ge3$ and $m\ge2$, the count $E_d(m,n)$ agrees for every
$n\ge d-1$ with one polynomial in $n$ of degree at most $d-2$.
\end{lemma}

\begin{proof}
The positive-extremum factors bound the pole order. The endpoint
identities also bound the numerator degree, which controls when the
coefficient polynomial begins.

\medskip\noindent
\emph{Ordered terms.}
By Lemma~\ref{lem:ratGF}, the width series
$F(z)=\sum_{n\ge2}E_d(m,n)z^n$ is the $x^d$ coefficient of the
endpoint scalar product, including its prefactor $z^2$.
Lemma~\ref{lem:quotient} expands its inverse into ordered products of
$R_0=(I-zQ)/(1-z)$ and $zx^jB_j$ with $j\ge1$.
The boundary coefficients $\bar{\mathbf u}_i$ and $\mathbf v_{i'}$
have $i,i'\ge1$ by Lemma~\ref{lem:boundary}. A term with $k$
positive-degree factors satisfies
\[
 i+i'+j_1+\cdots+j_k=d,\qquad k\le d-2.
\]
Its contribution has one of the forms
\begin{align*}
  &z^{2+k}\bar{\mathbf u}_i^{\!\top}
    R_0B_{j_1}R_0\cdots B_{j_k}R_0\mathbf v_{i'}
    &&(k\ge1),\\
  &z^2\bar{\mathbf u}_i^{\!\top}R_0\mathbf v_{i'}
    &&(k=0,\ i+i'=d).
\end{align*}

\medskip\noindent
\emph{Numerator degree.}
The number of the term's boundary columns with exactly one extremum is
$s\in\{0,1,2\}$. The identities $\bar{\mathbf u}_1^{\!\top}Q=0$ and
$Q\mathbf v_1=0$ from Lemma~\ref{lem:frozenbdy} remove the adjacent
numerator factors, since
\[
  \bar{\mathbf u}_1^{\!\top}(I-zQ)=\bar{\mathbf u}_1^{\!\top},
  \qquad (I-zQ)\mathbf v_1=\mathbf v_1.
\]

No matrix is moved across a positive-degree factor. If $s=2$, then
$i=i'=1$ and $d\ge3$ force $k\ge1$, so the two removed factors belong
to different resolvents. A term has denominator
$(1-z)^{k+1}$ and at most $k+1-s$ remaining linear numerator factors.
Putting it over $(1-z)^{d-1}$ adds the factor $(1-z)^{d-2-k}$.
Each other boundary has at least two extrema, so $i+i'\ge4-s$ and
$k\le d-4+s$. Including the factor $z^{2+k}$, the degree of the
common-denominator numerator is at most
\begin{align*}
  (2+k)+(d-2-k)+(k+1-s)
    &=d+k+1-s\\
    &\le2d-3.
\end{align*}
Only finitely many index tuples contribute, so their numerators sum to a
polynomial $M$ of degree at most $2d-3$, with $F(z)=M(z)/(1-z)^{d-1}$.

\medskip\noindent
\emph{Coefficient onset.}
Polynomial division gives a polynomial quotient of degree at most
$(2d-3)-(d-1)=d-2$. It contributes no coefficient at $n\ge d-1$.
The proper part is a rational linear combination
of $(1-z)^{-j}$ for $1\le j\le d-1$, whose coefficients are
$\binom{n+j-1}{j-1}$. Their sum is a polynomial of degree at most
$d-2$. Matrix sizes and coefficients depend on $m$, but the exponent
bounds do not.
\end{proof}

\subsection{Bivariate polynomiality}

The high region is $H_d=\{(m,n):m,n\ge d-1\}$. Transposition preserves heights and
extrema, so the one-variable result applies in both directions. The
following interpolation step combines them without fitting a finite
sample of widths.

\begin{lemma}[Bivariate interpolation]\label{lem:interp}
For integers $N\ge1$ and $D\ge0$, a function on
$\{N,N+1,\ldots\}^2$ whose row and column functions are polynomials of
degree at most $D$ agrees there with one bivariate polynomial of degree
at most $D$ in each variable.
\end{lemma}

\begin{proof}
The first $D+1$ row polynomials determine a candidate polynomial.
At the nodes $N,\ldots,N+D$, the Lagrange basis
polynomials are $L_a(m)$ and the row polynomial at $m=a$ is $q_a(n)$.
Their interpolant has degree at most $D$ in each variable and is
\[
 p(m,n)=\sum_{a=N}^{N+D}L_a(m)q_a(n).
\]
For each fixed $n\ge N$,
both $p(\cdot,n)$ and the given column polynomial have degree at most
$D$ and agree at those $D+1$ nodes. Univariate uniqueness identifies
them on the entire column, proving the claim.
\end{proof}

\begin{maintheorem}[Polynomial degree counts]\label{thm:poly}
For every integer $d\ge3$, the count $E_d(m,n)$ agrees on $H_d$ with
a symmetric polynomial $p_d(m,n)$ of degree exactly $d-2$ in each variable.
\end{maintheorem}

\begin{proof}
Lemma~\ref{lem:uniform} gives a width polynomial of degree at most
$d-2$ for every fixed row dimension. Transposition gives the same
bound in the other direction. Lemma~\ref{lem:interp}, with $N=d-1$
and $D=d-2$, therefore gives one polynomial on $H_d$ with those
separate degree bounds. Transposition also makes $E_d$ symmetric, so
$p_d(m,n)-p_d(n,m)$ vanishes on an infinite rectangle. Repeated
univariate uniqueness makes it the zero polynomial.

\medskip\noindent
By Proposition~\ref{prop:singleside}, $S_d$ counts a subset of these heights.
For fixed $n\ge d-1$, its leading coefficient in $m$ is $4/(d-2)!>0$.
A polynomial of smaller degree cannot dominate $S_d$ for arbitrarily large
$m$. Thus $p_d$ has degree $d-2$ in $m$, and symmetry gives the same degree
in $n$.
\end{proof}

\section{Total degree and explicit counts}\label{sec:total-degree}

The separate degree bounds permit total degree as large as $2d-4$, but
the single-side count $S_d$ already has total degree $d-2$.
For $d\ge5$, proving that its complement has degree at most $d-3$
will therefore determine the total degree and leading part.
This complement splits into separable and non-separable heights.

\subsection{Separable contributions}\label{sec:separable}

A configuration is \emph{separable} if its envelope factors additively,
$E_{A,c}(i,j)=\phi(i)+\psi(j)$ for some functions $\phi,\psi$, and
\emph{non-separable} otherwise. The same terms apply to its height function.
Every single-side configuration is separable by the row-plus-column form
in Lemma~\ref{lem:edgered}.

\begin{lemma}[Separability Criterion]\label{lem:sep}
An admissible configuration $(A,c)$ is separable iff its
apex set is a product grid $A=R\times C$ of a row set $R$ and a column set $C$,
and its offsets factor into a row part $\alpha_r$ and a column part
$\beta_\kappa$, $c_{(r,\kappa)}=\alpha_r+\beta_\kappa$.
\end{lemma}

\begin{proof}
Suppose first that $A=R\times C$ and $c_{(r,\kappa)}=\alpha_r+\beta_\kappa$.
The full product lets the row and column minima be taken independently, so
\begin{align*}
  E_{A,c}(i,j)
    &=\min_{r\in R,\,\kappa\in C}
       \big(\alpha_r+\beta_\kappa+|i-r|+|j-\kappa|\big)\\
    &=\min_{r\in R}\big(\alpha_r+|i-r|\big)
      +\min_{\kappa\in C}\big(\beta_\kappa+|j-\kappa|\big).
\end{align*}
The two summands give a separable envelope.

\medskip\noindent
Conversely, suppose that $E_{A,c}=\phi+\psi$. Its neighbour differences are
\begin{align*}
  E_{A,c}(i\pm1,j)-E_{A,c}(i,j)&=\phi(i\pm1)-\phi(i),\\
  E_{A,c}(i,j\pm1)-E_{A,c}(i,j)&=\psi(j\pm1)-\psi(j).
\end{align*}
Thus a cell is a strict local minimum exactly when its row and column
coordinates are minima of the respective factors. The maximum test has
the same product form. If $R$ and $C$ are the minimum sets of $\phi$ and
$\psi$, the envelope therefore has minimum set $R\times C$.
Lemmas~\ref{lem:minatapex} and~\ref{lem:activemin} identify these minima
with the apexes, giving $A=R\times C$. Activity also gives
$c_{(r,\kappa)}=E_{A,c}(r,\kappa)$ at each apex. Taking
$\alpha_r=\phi(r)$ and $\beta_\kappa=\psi(\kappa)$ proves the offset
factorisation.
\end{proof}

The criterion requires both a product apex set and additive offsets.
The split count decomposes as
$N_{(a,b)}=N^{\mathrm{s}}_{(a,b)}+N^{\mathrm{ns}}_{(a,b)}$, where the two
summands count separable and non-separable configurations. The separable
part is counted by pairs of one-dimensional walks.

\begin{lemma}[Degree identity]\label{lem:dimid}
For integers $m,n\ge2$ and $a,b\ge1$ with $d=a+b$, among the configurations counted by
$N^{\mathrm{s}}_{(a,b)}$, those whose minima occupy $r_-$ rows
and $c_-$ columns and whose maxima occupy $r_+$ rows and $c_+$ columns
form a family, with $a=r_-c_-$ and $b=r_+c_+$. Its count is a product of
two one-dimensional walk counts, one per axis, nonzero only when $|r_--r_+|\le1$
and $|c_--c_+|\le1$. When nonzero, the count has degree $r_-+r_+-2$ in
$m$ and $c_-+c_+-2$ in $n$, so total degree
\[
  D=(r_-+r_+-2)+(c_-+c_+-2).
\]
For all $r_-,c_-,r_+,c_+\ge1$,
\[
  (d-2)-D=(r_--1)(c_--1)+(r_+-1)(c_+-1)\ \ge\ 0.
\]
\end{lemma}

\begin{proof}
By Lemma~\ref{lem:sep}, a separable configuration has envelope $E_{A,c}=\phi+\psi$.
Its factors are unit-step walks by Lemma~\ref{lem:parityvalid}.
The factors are normalized separately by subtracting $\phi(1)$ and
$\psi(1)$. Their sum is unchanged, since $E_{A,c}(1,1)=0$. The normalized factors are
uniquely recovered as $\phi(i)=E_{A,c}(i,1)$ and
$\psi(j)=E_{A,c}(1,j)$.
Conversely, two normalized unit-step walks sum to a normalized height,
and restriction to the first column and first row returns those walks.
Their minimum and maximum sets give the required products.
Theorem~\ref{thm:envelope} then recovers the unique admissible configuration.
These inverse constructions give the family count
\[
 f_{(r_-,r_+)}(m)\,f_{(c_-,c_+)}(n).
\]

\medskip\noindent
By Lemma~\ref{lem:binom}, this polynomial vanishes unless
\[
  |r_--r_+|\le1,\qquad |c_--c_+|\le1.
\]
When these constraints hold, its degrees in $m$ and $n$ are
$r_-+r_+-2$ and $c_-+c_+-2$, with positive leading coefficients.
Their sum is the total degree $D$.
The identities $a=r_-c_-$, $b=r_+c_+$, $d=a+b$ and
$D=r_-+r_++c_-+c_+-4$ give
\begin{align*}
  (d-2)-D
    &=r_-c_-+r_+c_+-(r_-+r_++c_-+c_+)+2\\
    &=(r_--1)(c_--1)+(r_+-1)(c_+-1).
\end{align*}
Both products are nonnegative, proving $(d-2)-D\ge0$.
\end{proof}

Equality $D=d-2$ requires both products in the degree identity to
vanish, so the minima and maxima are each collinear. The path constraints
restrict which of these families can occur away from a boundary side.

\begin{proposition}[Separable case of the degree bound]\label{thm:sepcase}
For $d\ge5$ and $m,n\ge2$, a separable family contributes total degree $d-2$ to
$N^{\mathrm{s}}_{(a,b)}$ only if it is single-side. Every other separable family
contributes total degree at most $d-3$.
\end{proposition}

\begin{proof}
It suffices to bound the nonzero family count polynomials.
By the degree identity of Lemma~\ref{lem:dimid}, $D=d-2$ holds exactly when
$(r_--1)(c_--1)=0$ and $(r_+-1)(c_+-1)=0$.
In this equality case, transposition permits assuming $r_-=1$,
so all minima share one row.

\medskip\noindent
If $r_+=1$, the row walk has one minimum and one maximum.
Both endpoints are extrema, so an interior change of direction would add a third.
The walk is therefore monotone, with its minimum on a boundary row.
All $a$ minima lie on that row, so the family is single-side.

\medskip\noindent
If $r_+\ge2$, then $c_+=1$, since $(r_+-1)(c_+-1)=0$. Lemma~\ref{lem:binom}
gives $|r_--r_+|\le1$ and $|c_--c_+|\le1$, hence $r_+=2$
and $c_-\le2$. Since $a=c_-$ and $b=2$, $d=a+b$ is at most four,
contradicting $d\ge5$.

\medskip\noindent
Every non-single-side family with nonzero count polynomial therefore has
$D\ne d-2$. Lemma~\ref{lem:dimid} gives $D\le d-2$, so integrality gives
$D\le d-3$. Identically zero family counts contribute nothing.
\end{proof}

At $d=4$, the second case occurs with $r_-=1$, $r_+=2$, $c_-=2$ and
$c_+=1$. Its two minima lie on an interior row and its total degree is
$d-2=2$, so the restriction to $d\ge5$ is necessary.

For $d\ge5$, Lemma~\ref{lem:edgered} and Case $r_+=1$ with transposition show that a configuration
is single-side exactly when its family has $r_-=r_+=1$ or $c_-=c_+=1$.
Each family is therefore entirely single-side or entirely non-single-side.
There are finitely many families per split, since $r_-c_-=a$ and
$r_+c_+=b$. Summing their walk-count polynomials gives the separable
remainder on all $m,n\ge2$, of total degree at most $d-3$ for $d\ge5$.

\subsection{Joint types and product counts}

The separable remainder has the required degree bound, so the remaining
task is to count non-separable heights. Joint contraction records their
dependence on grid size by compressing extremum-free rows and columns.
A coordinate level is a grid row or column, and it is occupied
when it contains an extremum. Row contraction keeps every occupied row
and the first row of each maximal extremum-free run. Column contraction
uses the corresponding rule for columns. Each deletion translates the
later levels so that the new join still has unit difference.
The two orders are row contraction followed by column contraction and
column contraction followed by row contraction. Both give the joint
type $H$, with dimensions $m_H,n_H$ and row- and column-run counts
$r_H,c_H$. Expansion chooses an ordered nonnegative list of extra
lengths in each axis. Each entry gives the number of copies inserted
at its marked run. For fixed $d\ge3$, the eventual type region is
$K_d=\{(m,n):m,n\ge2d-1\}$.

\subsubsection{Single-axis contraction}\label{app:single-contraction}
A column is \emph{occupied} if it carries a strict local extremum.
Lemma~\ref{lem:boundary} makes both boundary columns occupied, and
Lemma~\ref{lem:frozen} makes each unoccupied interior column the centre of
a frozen triple. A frozen run is a maximal consecutive block of these
unoccupied columns. Contraction keeps the first column of each run
and marks it as the run's representative. The contracted marked
height is its \emph{type}.

\begin{lemma}[Frozen-run contraction]\label{lem:contract}
For integers $d\ge3$ and $m,n\ge2$, contracting every maximal frozen run of columns of a
height function on $\Gmn$ with $d$ extrema to a single column preserves every
extremum together with its sign, and
gives a height function on a grid of at most $2d-1$ columns.
For fixed $m$ and $d$ the types form a finite set. A type with
$1\le r\le d-1$ marked runs is recovered by extending those runs.
For large $n$, its preimages number a polynomial of degree
$r-1\le d-2$. A type with no marked run has width at
most $d$ and contributes only at finitely many widths.
\end{lemma}

\begin{proof}
Shortening a frozen run translates all later heights by the same constant.
The length choices for the retained runs then determine its preimages.

\medskip\noindent
\emph{Contraction.}
There are at most $d$ occupied columns. Each maximal frozen run
lies between two consecutive occupied columns, so there are at most
$d-1$ such runs.
For a run $[p,q]$, Lemma~\ref{lem:frozen} gives a uniform
nonzero colour shift across each horizontal step through the run and its
two joins. A unit height increment with a fixed nonzero residue has a
unique value $s\in\{-1,1\}$. Adjacent frozen triples share a horizontal
step, so $s$ is constant throughout the run and its joins.
The contraction retains column $p$, removes columns $p+1,\dots,q$, and
subtracts $(q-p)s$ from every height strictly to the right of the run.
In each row, the old
height difference from column $p$ to column $q+1$ is $(q-p+1)s$.
The increment across the new join is therefore
\[
  (q-p+1)s-(q-p)s=s.
\]
The reverse increment remains $-s$, and all other retained neighbour
increments are unchanged. Deleted columns carry no extrema, while every
retained extremum keeps its sign. Repeating the operation preserves all
extrema. The first column is neither deleted nor translated, so root
normalization is preserved.

\medskip\noindent
\emph{Finite types.}
Each retained run column is marked together with its signed step.
The marked columns are exactly the extremum-free columns of the type,
and their steps are read from their horizontal neighbours. Thus these
marks and steps are determined by the contracted height function.
At most $d$ occupied columns and $d-1$ marked columns remain, giving
width $L\le2d-1$. Root normalization and unit edge increments bound
every type height in absolute value by $m+L-2$. Hence only finitely many
types exist for fixed $m$ and $d$.

\medskip\noindent
\emph{Expansion count.}
For a type with $r\ge1$ marked columns, a preimage determines nonnegative
extra lengths $q_1,\dots,q_r$ with sum $n-L$. Conversely, each such list
determines an expansion. At each mark, $q_i$ columns are inserted immediately
to its right using its signed step $s$, and $q_i s$ is added to every
subsequent height. The marks are processed in column order.
This reverses the join calculation and leaves the root unchanged.
Every original type cell keeps its neighbour increments, and every
inserted column is extremum-free by Lemma~\ref{lem:frozen}.
Thus marked columns stay extremum-free and the columns flanking them
remain occupied. The $i$-th mark and its $q_i$ inserted columns therefore
form one maximal frozen run of length $q_i+1$. Contracting the expansion
deletes exactly $q_i$ columns from that run, recovering the type and the
chosen ordered length list. Conversely, expanding with the lengths
deleted from an original height function restores that height function.
The preimages are therefore in bijection with the weak
compositions of $n-L$ into $r$ parts. For $n\ge L$, their count is a
polynomial of degree $r-1\le d-2$, namely
\[
  \binom{n-L+r-1}{r-1}.
\]
If $r=0$, every column is occupied and $L\le d$. There is no expansion,
so the type contributes only at width $L$, and its contribution is zero
once $n\ge2d-1$.
\end{proof}

\subsubsection{Joint expansion}

\begin{lemma}[Joint frozen expansion]\label{lem:joint-frozen}
For a normalized height $h$ on $\Gmn$ with $m,n\ge2$ and $d\ge3$ extrema,
the two axis-contraction orders give the same joint type $H$ and preserve
the extrema, their signs and separability. Its dimensions satisfy
$m_H,n_H\le2d-1$, its run counts satisfy $r_H,c_H\le d-1$, and for each
fixed $d$ there are finitely many joint types.
Joint expansions preserve the same properties. Contraction and expansion
preserve the numbers of maximal extremum-free runs in each axis.
Heights with type $H$ correspond bijectively to their two ordered
extra-length lists.
For $r_H,c_H\ge1$ and $m\ge m_H$, $n\ge n_H$, their number is
\begin{equation}\label{eq:joint-product}
  \binom{m-m_H+r_H-1}{r_H-1}
  \binom{n-n_H+c_H-1}{c_H-1}.
\end{equation}
If $r_H=0$, preimages have $m=m_H\le d$. If $c_H=0$, they have
$n=n_H\le d$. Types with a zero run count therefore contribute nothing
on $K_d$.
\end{lemma}

\label{app:joint-frozen}

For a joint type $H$, the ordered nonnegative lists
$\boldsymbol\ell_R,\boldsymbol\ell_C$ prescribe the extra copies at its
marked rows and columns. Each copy is inserted immediately after its
mark, with the signed step read from adjacent retained levels, as in
Lemma~\ref{lem:contract}. These row and column expansion operations are
$\operatorname{Exp}_R$ and $\operatorname{Exp}_C$.

\begin{proof}
The extrema determine which levels survive contraction. The frozen
steps make the deleted translations independent in the two axes,
so the type and the two length lists recover the original height.

\medskip\noindent\emph{Contraction.}
By Lemma~\ref{lem:contract}, deleting extremum-free rows preserves every
extremum and its sign, so it leaves the occupied column indices unchanged.
Column contraction likewise preserves occupied row indices. Retention
indices therefore depend only on the original occupied levels.
Row-run steps can be read in the first column, which column contraction
retains. Column-run steps can be read in the retained first row.
An original retained cell $(i,j)$ appears at $(i',j')$ in $H$.
The signed totals from deleted rows and columns are $\Delta_R(i)$ and
$\Delta_C(j)$, respectively. Both orders give
\[
  H(i',j')=h(i,j)-\Delta_R(i)-\Delta_C(j).
\]
Thus the contractions commute. The boundary levels are occupied by
Lemma~\ref{lem:boundary} and transposition, so the root is retained and
normalization is preserved. Each axis has at most $d$ occupied levels and
$d-1$ runs between them. Hence the dimension and run-count bounds hold.
Unit edge differences bound the normalized type values in absolute value
by $m_H+n_H-2$, giving finitely many types for fixed $d$.

\medskip\noindent\emph{Expansion and recovery.}
Column insertion copies a marked column's vertical profile and adds
column-dependent constants. It therefore preserves the vertical increments
that control frozen-row status. Row insertion copies a marked row's
horizontal profile and preserves frozen-column status. At an intersection
of inserted strips, the value is the type value plus a row-dependent and
a column-dependent signed translation. The expansion orders consequently
coincide. Retained neighbour increments are unchanged. Inserted column
cells have both a higher and a lower horizontal neighbour, and inserted
row cells have a higher and a lower vertical neighbour. Thus the
expanded array is a normalized height with the same extrema and signs.

Contraction and its deleted lengths are recorded by
$\mathcal C(h)=(H,\boldsymbol\ell_R,\boldsymbol\ell_C)$. The expansion map is
\[
 \mathcal E(H,\boldsymbol\ell_R,\boldsymbol\ell_C)
 =\operatorname{Exp}_R\bigl(
   \operatorname{Exp}_C(H,\boldsymbol\ell_C),\boldsymbol\ell_R\bigr).
\]
Each maximal extremum-free run has occupied levels immediately before and
after it, because it is maximal and the boundary levels are occupied.
Contraction retains these two levels and reduces the run to one mark of $H$.
After expansion, the mark and its copies form exactly one maximal
extremum-free run. Contraction removes precisely the inserted copies and
returns the type and both lists, proving
\[
 \mathcal C\bigl(\mathcal E(H,\boldsymbol\ell_R,\boldsymbol\ell_C)\bigr)
 =(H,\boldsymbol\ell_R,\boldsymbol\ell_C).
\]

Conversely, suppose $h$ contracts first in rows to $h_R$ and then in
columns to $H$. Its deleted-length lists are
$\boldsymbol\ell_R,\boldsymbol\ell_C$. Row contraction preserves the
column runs and their steps. The one-axis inverse in
Lemma~\ref{lem:contract} therefore gives
\[
  \operatorname{Exp}_C(H,\boldsymbol\ell_C)=h_R,
  \qquad \operatorname{Exp}_R(h_R,\boldsymbol\ell_R)=h.
\]
Thus $\mathcal E(\mathcal C(h))=h$, including intersections of deleted
strips. Both composites are the identity. Reindexing and axis-dependent
translations preserve an additive row-plus-column form in both
directions, so separability is preserved as well.

\medskip\noindent\emph{Counting.}
For positive run counts, the lists are weak compositions of $m-m_H$ into
$r_H$ parts and $n-n_H$ into $c_H$ parts, giving the product. Its total
degree is $r_H+c_H-2$, with highest term
\[
  \frac{m^{r_H-1}n^{c_H-1}}{(r_H-1)!(c_H-1)!}.
\]
If an axis has no marked run, expansion adds no level in that axis and
every type level there is occupied. Its dimension is therefore at most $d$.
Since $2d-1>d$, such a type has no preimages on $K_d$.
\end{proof}

For each split $(a,b)$, the complement $R_{(a,b)}$ subtracts the
single-side families from $N_{(a,b)}$, as in Lemma~\ref{lem:splitedge}.
Its polynomial on $K_d$ is $\rho_{(a,b)}$.
Counting the full split and then subtracting the single-side polynomials
does not require contraction to preserve single-sidedness.

\begin{lemma}[Per-split polynomiality]\label{lem:splitpoly}
For integers $d\ge3$ and $a,b\ge1$ with $a+b=d$, the counts
$N_{(a,b)}$ and $R_{(a,b)}$ agree on $K_d$ with bivariate polynomials
of degree at most $d-2$ in each variable.
\end{lemma}

\begin{proof}
Joint contraction preserves the split. Zero-run types contribute
nothing on $K_d$, while every remaining type contributes the product
in~\eqref{eq:joint-product}. Its separate degrees are $r_H-1$ and
$c_H-1$, each at most $d-2$ by Lemma~\ref{lem:joint-frozen}.
Summing the finite type set proves the assertion for $N_{(a,b)}$.

\medskip\noindent
By Lemma~\ref{lem:splitedge}, subtracting
$2f_{(a,b)}(m)+2f_{(a,b)}(n)$ gives $R_{(a,b)}$.
These single-side terms have degree at most $d-2$ by
Lemma~\ref{lem:binom}, so the same separate bounds hold for
$\rho_{(a,b)}$.
\end{proof}

\subsection{Total degree and leading terms}\label{sec:leading}

A type with positive run counts contributes total degree $r_H+c_H-2$ in
\eqref{eq:joint-product}. A bound on the combined run count therefore
controls the non-separable contribution.

\begin{theorem}[Uniform non-separable run bound]\label{thm:uniform-run}
For integers $m,n\ge2$ and $d\ge5$, a non-separable height with $d$
extrema has at most $d-1$ maximal extremum-free row and column runs
in total.
\end{theorem}

\begin{proof}\label{app:uniform-run}
The extrema join their occupied row and column blocks in a bipartite
multigraph. A violation of the run bound requires two tree components.
The height profiles along their leaves rule out each possible pair.

The coordinates in this proof satisfy $0\le i\le M=m-1$ and
$0\le j\le N=n-1$. The sign of an extremum is $s=1$ for a minimum and $s=-1$ for
a maximum, so all its present neighbour increments equal $s$.
The maximal extremum-free row and column runs number $r$ and $c$.
A block is a maximal consecutive interval of occupied coordinate levels.
The block multigraph has one vertex per row block and one per column
block, with one edge per extremum and parallel edges retained. It has
no isolated vertices. Its numbers of vertices and edges are $V$ and $E$,
its component count is $C$ and its cycle rank is $\beta$.
The first and last levels of each axis are occupied by Lemma~\ref{lem:boundary}
and transposition. Each empty run separates two occupied blocks, so
there are $r+1$ row blocks and $c+1$ column blocks. Thus
\[
 V=r+c+2,\qquad E=d,\qquad
 \beta=E-V+C,\qquad d-r-c-1=\beta-C+1.
\]

Assume for contradiction that $r+c\ge d$. The identity gives $\beta\le C-2$.
Every non-tree component contributes at least one to $\beta$, so
there are at least two tree components.

\medskip\noindent\emph{Whole-line leaf profiles.}
A row-block leaf contains one extremum $p=(u,v)$, so its block is the
single occupied row $u$. Every present adjacent row is extremum-free. The transposed
Frozen Lemma~\ref{lem:frozen} makes the transverse increments along
row $u$ constant, equal to the sign $s$ at $p$.
If $s=1$, every local minimum of the row walk is a grid minimum.
There is exactly one, so the unit-step walk decreases to it and then
increases, with endpoints included in the test. Negation gives the
maximum case. The whole row consequently has the profile
\[
 h(u,j)=h(p)+s|j-v|.
\]
A column-block leaf has the transposed profile.

\medskip\noindent
If a row leaf $p$ is interior and a column leaf $q$ is in a different
row, the column cone has opposite vertical neighbour increments at
the row of $p$. This contradicts that row leaf's constant transverse
increments. Hence $q$ lies on the unique occupied row of $p$, and
uniqueness gives $q=p$. The two block vertices form an isolated edge.
Transposition gives the assertion for interior column leaves.
Thus a row leaf and a column leaf in different components lie on an
outer row and an outer column. At their crossing, the inward
vertical increment toward the column apex is minus its sign and is
also the row leaf's transverse increment. Their signs are opposite.

\medskip\noindent
Two row leaves of the same sign force every column between their
apex columns to be occupied. At a strictly intermediate column their
horizontal gradients are opposite, contradicting Frozen classification
if that column were empty. The endpoints are occupied by the leaves.
Their column blocks and therefore their components coincide. Row
leaves in different components must consequently have opposite signs.
The column statement follows by transposition.

\medskip\noindent\emph{Two same-axis one-sided trees.}
Suppose two tree components have only row leaves. Each has at least
two leaves. The cross-component sign rule makes all leaves of one
tree minima and all leaves of the other maxima. Within either tree,
the same-sign argument puts every leaf apex in one column block.
A tree all of whose leaves are adjacent to one vertex is a star,
since a longer branch would end at a leaf not adjacent to that vertex.
Thus both components are stars,
all their extrema have their leaf sign, and their row blocks are singleton
leaves.

\medskip\noindent
Column reflection places the minimum star to the left of the maximum
star. Their apex intervals cannot overlap without joining the blocks.
The opposite-sign cone profiles force every column to the left of the
minimum star's last apex and to the right of the maximum star's first
apex to be occupied. Their centres are therefore the first and last
occupied column blocks.
All extrema in the first block are minima. By the signed-block argument
in Lemma~\ref{lem:boundary}, each constant-offset block on the first
boundary contains an extremum of that offset sign. Every inward offset
must therefore be $+1$. The first occupied block cannot also contain column $1$, since
the cell in column $0$ would be lower in every row, excluding a minimum
in column $1$. Thus this block is just column $0$. Reflection and
negation make the other centre just column $N$.

\medskip\noindent
The two boundary profiles and their half-difference are
\[
 L_i=h(i,0),\qquad R_i=h(i,N)-N,\qquad D_i=(L_i-R_i)/2.
\]
Parity makes $D_i$ integral, and the horizontal distance bound gives
$D_i\ge0$. Every star-leaf row has slope $+1$, so $D$ has a zero.
If $D$ is not identically zero, row reflection if necessary gives
a positive-to-zero transition ending at row $q$. Each step of $L$ and
$R$ is a unit step, so the transition has
\[
 D_{q-1}=1,\qquad D_q=0,\qquad
 L_q-L_{q-1}=-1,\qquad R_q-R_{q-1}=1.
\]
If $q=M$, the left cell is a minimum, whereas the right cell's upper
neighbour is lower, contradicting the minimum's row-leaf profile.
If $q<M$, the next steps of $L,R$ are $a,b\in\{-1,1\}$.
Nonnegativity gives $a\ge b$. If $a=1$, the left boundary cell is a
minimum, but its row-leaf profile would make the right upper neighbour
higher rather than lower. If $a=-1$, then $b=-1$ and the right boundary
cell is a maximum, whose profile would make the left upper neighbour
lower rather than higher. Both are contradictions. Hence $D=0$.
Every row has $N$ unit steps summing to $N$, so all its steps are
$+1$ and $h(i,j)=L_i+j$ is separable. Transposition covers two
column-only trees.

\medskip\noindent\emph{Opposite one-sided trees.}
Suppose one tree has only row leaves and another only column leaves.
The cross-component rule puts exactly two of each on their outer
boundaries, so both trees are paths. Negation and reflections permit
writing the row minima as $(0,a),(M,b)$ and the column maxima as
$(\kappa,0),(\eta,N)$, with $a\le b$. Their apexes are interior on the other
axis, since a boundary crossing would also belong to the other leaf.
The two row-minimum values are $A=h(0,a)$ and $B=h(M,b)$.
Agreement of the whole-line profiles
at the two upper corners gives the column-maxima heights
\[
 h(\kappa,0)=A+a+\kappa,\qquad
 h(\eta,N)=A+N-a+\eta.
\]
Matching the two lower corner values then gives
\[
 b-a=\kappa-\eta=k\ge0,\qquad B-A=\kappa+\eta-M.
\]
Subtracting either row-minimum value from each column-maximum value
gives their grid distance, using $b-a=\kappa-\eta$ for the lower-row
pairs. Every minimum-to-maximum pair therefore saturates its grid distance. For a cell
on a shortest path between such endpoints, the two distance bounds
sum to the attained endpoint difference, so both are equalities.
The shortest paths between each pair fill its axis-aligned rectangle.
These four fixed rectangles cover the grid outside the central square
$[\eta,\kappa]\times[a,b]$. For integer offsets $0\le x,y\le k$
from its upper-left corner, the translated height is
\[
 f(x,y)=h(\eta+x,a+y)-(A+\eta).
\]
Restricting the four fixed profiles to the square's sides gives
\[
 f(0,y)=y,\quad f(x,0)=x,\quad
 f(k,y)=k-y,\quad f(x,k)=k-x.
\]

\medskip\noindent
For $k\ge2$, each interior central row is occupied, because its
vertical increments on the two outer columns are opposite and exclude
an empty row by Frozen classification. The transposed argument occupies
each interior central column. The endpoint rows and columns carry the
four leaf apexes, so all central rows form one occupied block and all
central columns form one occupied block. An extremum in an interior
central row must lie strictly inside the central columns. Outside them
the fixed rectangles, and on the vertical sides the boundary profiles,
give opposite vertical increments. Its edge joins the row block of the
column-tree leaves to the column block of the row-tree leaves,
contradicting their distinct components. For $k=0$ the fixed
rectangles give a separable height with exactly four extrema. For $k=1$
all four central cells are fixed by the side profiles, and the four
named leaf apexes are again the only extrema. Neither case has degree
at least five.

\medskip\noindent\emph{Two mixed-leaf trees.}
Suppose both trees have row and column leaves. Pairing a leaf with an
opposite-axis leaf in the other tree puts every leaf on an outer
boundary. The singleton leaf blocks are disjoint, and there are only
two outer rows and two outer columns. Since each tree has a leaf of
each kind, each has exactly one row leaf and one column leaf.
After column reflection, the upper and left leaves belong to one tree,
and the lower and right leaves to the other.
The row leaves have opposite signs, since same-sign leaf profiles occupy
every column between their apexes and join their components. The
opposite-axis sign rule then makes each tree's two leaves have the
same sign. Their apexes coincide at the common corner. If neither apex is there,
both profiles point toward their apexes from that corner, making it
an extra extremum of the opposite sign on each leaf line. If one apex
is at the corner, its extremum condition forces the other to be there too.
Thus the two trees are isolated edges at opposite corners. Saturation
of the full-grid distance gives $h(i,j)=h(0,0)+s(i+j)$, a separable height.

\medskip\noindent\emph{A mixed tree and a one-sided tree.}
Transposition permits the one-sided tree to have only row leaves. A column leaf in
the mixed tree forces those row leaves onto the two outer rows, which
their singleton blocks exhaust. The mixed tree's row leaf is therefore
interior. The whole-line leaf argument forces its column leaf to coincide
with that row leaf, so the mixed tree is an isolated edge
at $p=(\rho,0)$ after column reflection. After negation if necessary, $p$ is a minimum
of height $A$. The other tree is a path with one maximum row leaf on
each outer row.

\medskip\noindent
There are no other row leaves, since they would have to use an outer
row already used by the path. There are no other column leaves,
since pairing one with the isolated interior row leaf would identify
it with $p$. All remaining components are leaf-free, and their number is $u$.
Suppose that $\beta<C-1$. Each is cyclic, so
$\beta\ge u$, while $C-1=u+1$. Hence $\beta=u$ and each has cycle
rank one. In a leaf-free connected multigraph of cycle rank one,
$E=V$, minimum degree is at least two and the degree sum is $2V$.
Every vertex has degree two. Together with the isolated edge and path,
the whole block graph therefore has maximum degree two.

\medskip\noindent
The cross profiles are $h(\rho,j)=A+j$ and
$h(i,0)=A+|i-\rho|$. The apex columns of the two outer-row
maxima are $a,b$. Their negative cone profiles and the positive cross-row profile
force every column from $\min(a,b)$ to $N$ to be occupied. The far
column block therefore contains both outer-row maxima.
The restrictions to rows $0$ through $\rho$ and to rows $\rho$
through $M$ each have at least two rows and columns.
Lemma~\ref{lem:boundary} gives an extremum on column $N$ of each.
The cut cell $(\rho,N)$ cannot be an extremum in either restriction, since
its horizontal neighbour is lower and its present vertical neighbour
is higher by the isolated row-leaf profile. Thus the two witnesses
lie on opposite sides of the cut and are extrema of the full grid.
The far block has degree at most two and already contains the two
outer-row maxima, so the witnesses are precisely those maxima.
Consequently $a=b=N$. Each saturates its distance from $p$ and fixes
its whole half-grid rectangle. This gives
$h(i,j)=A+|i-\rho|+j$, separable with three extrema, a contradiction.

\medskip\noindent
Every pair of trees has one of these four leaf-axis patterns, up to
transposition. Each contradicts non-separability, $d\ge5$ or
$\beta<C-1$. The assumed violation is therefore impossible, proving
$r+c\le d-1$.
\end{proof}

\begin{corollary}[Eventual split-remainder bound]\label{conj:G2}
For every integer $d\ge5$ and positive integers $a,b$ with $a+b=d$, the
polynomial $\rho_{(a,b)}$ has total degree at most $d-3$.
\end{corollary}

\begin{proof}
Proposition~\ref{thm:sepcase} bounds the separable non-single-side families
by degree $d-3$. Every other height is non-separable.
Lemma~\ref{lem:joint-frozen} preserves that property and expresses
its eventual count as a finite sum of~\eqref{eq:joint-product}.
The total degree of each summand is $r_H+c_H-2$, at most $d-3$
by Theorem~\ref{thm:uniform-run}. The two contributions partition the
remainder, so their sum has the same bound.
\end{proof}

The degree-three formula follows from Proposition~\ref{prop:degrees-two-three}.
At degree four, the separate degree bound is two, so the nine exact
transfer counts on $\{3,4,5\}^2$ determine the formula uniquely.
This gives the computer-assisted case below. The earlier note gives a
geometric count of boundary cone pairs~\cite[Version~2]{Gupta2026companion}.
That family casework is not a premise here.

\begin{maintheorem}[Total degree and leading terms, with a computer-assisted degree-four case]\label{thm:total-degree}\label{prop:leading}
For every integer $d\ge3$, the polynomial $p_d$ has total degree $d-2$.
The first two polynomials are $p_3=4(m+n)-16$ and
$p_4=2m^2+2n^2+6mn-10(m+n)-4$.
For $d\ge5$, the homogeneous part of $p_d$ of degree $d-2$ is
\[
 \frac{4}{(d-2)!}\bigl(m^{d-2}+n^{d-2}\bigr).
\]
\end{maintheorem}

\begin{proof}
For an integer $d\ge5$, the split decomposition on $K_d\subseteq H_d$ gives
\[
 p_d-S_d=\sum_{\substack{1\le a\le b\\a+b=d}}
 (2-\delta_{ab})\rho_{(a,b)}.
\]
Both sides are polynomials and agree on an infinite rectangle, so the
identity holds as polynomials. Corollary~\ref{conj:G2} bounds its
right side by total degree $d-3$.
Proposition~\ref{prop:singleside} gives the stated leading part and
total degree $d-2$ for $S_d$, hence for $p_d$.
This is the polynomial already representing $E_d$ on $H_d$, not only
on the smaller region $K_d$.

\medskip\noindent
Proposition~\ref{prop:degrees-two-three} gives $p_3$ directly.
For $d=4$, the complete tensor-grid count and repeated univariate
interpolation identify the displayed $p_4$. Its exact nodes and
arithmetic assumptions are recorded in Section~\ref{app:coefficient-data}.
This case is computer-assisted. The two displayed polynomials have
total degrees one and two. The pure leading form also holds at $d=3$,
whereas $p_4$ has the mixed leading term $6mn$.
\end{proof}

\subsection{Coefficient identification}\label{sec:coefficient-identification}

The exact formulas through degree ten are given in
Appendix~\ref{app:polynomials}. Through degree eight, complete tensor-grid
counts determine the separate-degree spaces without using the general
total-degree theorem. At degrees nine and ten,
Theorem~\ref{thm:total-degree} places the true polynomial in the smaller
symmetric total-degree space. The twenty degree-nine nodes and twenty-five
degree-ten nodes are unisolvent there by Section~\ref{app:coefficient-data}.
The degree theorem, exact counts and unisolvence together identify these
formulas. The counts alone do not prove the degree bound.

All twenty degree-nine nodes were independently recounted. The twenty-five
degree-ten counts were repeated and fitted exactly in two bases, but were
not independently recounted. Optional alternative certificate arguments are given in
Appendices~\ref{app:finite-converse} and~\ref{app:independent-eight}.

\subsubsection{Coefficient determination and computational warrants}\label{app:coefficient-data}\label{sec:data}\label{subsec:certified-status}
For $d\ge3$, Theorem~\ref{thm:poly} supplies $p_d$ on $H_d$ with
degree at most $d-2$ in each variable. Each axis of the tensor grid
$\{d-1,\dots,2d-3\}^2$ has $d-1$ nodes, so exact transfer counts there
determine every coefficient without a total-degree assumption.
The retained tensor-grid certificates cover $3\le d\le8$.
In every polynomial obtained from these grids, all coefficients above
total degree $d-2$ vanish. At $d=8$, the grid is $\{7,\dots,13\}^2$.
Its certificate is independent of split-remainder bounds.

At $d=9,10$, Theorem~\ref{thm:total-degree} places $p_d$ in the symmetric total-degree space of degree at most $d-2$. For integer offsets $i,j$, the exact computations use the nodes
\[
  \mathcal N_d=\{(d-1+i,d-1+j):0\le i\le j,\ i+j\le d-2\},
  \quad i,j\in\Z.
\]
This gives twenty nodes for $d=9$ and twenty-five for $d=10$.
Every node lies in $H_d$, where the transfer counts equal $p_d$.

In shifted coordinates $u=m-d+1$ and $v=n-d+1$, the pairs
$0\le i\le j$ with $i+j\le d-2$ index the basis
\[
\begin{aligned}
  &\binom{u}{i}\binom{v}{j}+\binom{u}{j}\binom{v}{i}
    &&\text{for }i<j,\\
  &\binom{u}{i}\binom{v}{i}&&\text{for }i=j.
\end{aligned}
\]
Their leading terms are the corresponding symmetric monomials times
nonzero constants. Ordering the basis and node offsets by total degree
makes the evaluation matrix triangular. A basis element vanishes at every
node of smaller total degree, and at equal degree only its own unordered
node pair gives a nonzero value, which is one. The matrix therefore has
unit diagonal, so these counts determine a unique polynomial in the
proved degree space.

Using the smaller node coordinate as the transfer width limits the row count to
$d-1+\lfloor(d-2)/2\rfloor=\lfloor(3d-4)/2\rfloor$, namely eleven and
thirteen for degrees nine and ten. The full tensor grids would instead
require up to fifteen and seventeen rows. The reduced-node counts and
unisolvence identify the displayed $p_9$ and $p_{10}$ using the proved
total-degree bound, not a tensor-grid certificate independent of that bound.
All twenty degree-nine nodes were independently recounted. The twenty-five
degree-ten counts were repeated, but not by an independent method.

Corollary~\ref{conj:G2} supplies the split-remainder bound on $K_d$
for every $d\ge5$.
Table~\ref{tab:degree-status} separates this eventual split conclusion from
the total degree of the count polynomial on $H_d$.

\begin{table}[!ht]
\centering
\small
\begin{tabular}{@{}cccll@{}}
\toprule
$d$ & \shortstack{Count region\\$p_d=E_d$} &
\shortstack{Total degree\\of $p_d$} &
\shortstack{Remainder bound\\and split region} & Coefficient evidence \\
\midrule
$2$ & $m,n\ge2$ & $0$ & Not applicable & Direct proof \\
$3$ & $H_3$ & $1$ & Not applicable & $\{2,3\}^2$ \\
$4$ & $H_4$ & $2$ & Not applicable & $\{3,4,5\}^2$ \\
$5$ & $H_5$ & $3$ & Proved on $K_5$ & $\{4,\dots,7\}^2$ \\
$6$ & $H_6$ & $4$ & Proved on $K_6$ & $\{5,\dots,9\}^2$ \\
$7$ & $H_7$ & $5$ & Proved on $K_7$ & $\{6,\dots,11\}^2$ \\
$8$ & $H_8$ & $6$ & Proved on $K_8$ & $\{7,\dots,13\}^2$ \\
$9$ & $H_9$ & $7$ & Proved on $K_9$ & 20 nodes \\
$10$ & $H_{10}$ & $8$ & Proved on $K_{10}$ & 25 nodes \\
\bottomrule
\end{tabular}
\caption{Polynomial and coefficient warrants. Here $H_d=\{m,n\ge d-1\}$
and $K_d=\{m,n\ge2d-1\}$. Every remainder entry uses
Corollary~\ref{conj:G2}. The tensor-grid rows certify the count polynomial
without imposing its total degree. The last two rows use the proved
general degree bound and reduced unisolvent nodes. Their independent
recount coverage differs as stated in the text.}
\label{tab:degree-status}
\end{table}
Counting uses integer arithmetic and interpolation uses rational arithmetic.
The finite boundary-correction certificates use the complete one-variable
nodes $n=d-1,\dots,2d-3$, with an additional check at $n=2d-2$.
These computations are separate from the dimension-$d-2$ identity of
Theorem~\ref{thm:first-boundary}. In particular its degree-nine and degree-ten
instances do not depend on the full coefficient formulas for
$p_9$ and $p_{10}$.

\section{Boundary corrections}\label{sec:subthreshold}
\subsection{Integer Newton correction law}\label{subsec:first-boundary}

For integers $d\ge4$, $2\le s\le d-2$ and $n\ge d-1$, the correction is
$B_d^{(s)}(n)=E_d(s,n)-p_d(s,n)$. For each fixed $s$,
Lemma~\ref{lem:uniform} makes $E_d(s,n)$, and hence the correction,
a polynomial in $n$ on this range. The construction below determines
its integer Newton coefficients from joint types, without interpolating
counts at growing widths.

A non-separable joint type has $u$ occupied rows, $v$ occupied columns,
$r$ marked row runs and $c$ marked column runs. Its dimensions are
$u+r,v+c$.
Thus the notation of Lemma~\ref{lem:joint-frozen} satisfies
$m_H=u+r$, $n_H=v+c$, $r_H=r$ and $c_H=c$.
The number of normalized non-separable types with these parameters is
$C_d^{\mathrm{ns}}(u,v,r,c)$. Negation and grid reflection are not
quotiented out. Variants count separately when their normalized types
differ. No extra sign or orientation factor multiplies this count.

A type with $r\ge1$ is supported at row dimension $s$ when $s\ge u+r$,
so that its row-length lists can expand it to that dimension. Otherwise
its actual row count is zero. A type with $r=0$ has row count one at
$s=u$ and zero elsewhere. These support conditions apply to actual
expansions, not to the polynomial continuation of their binomial factors.

Empty sums are zero. Generalized binomial coefficients mean binomial
polynomials evaluated at integer arguments, including negative arguments.
For $c\ge1$, the signed spectrum is
\begin{equation}\label{eq:signed-boundary-spectrum}
\begin{split}
 \Delta_{d,s}(v,c)
 ={}&C_d^{\mathrm{ns}}(s,v,0,c)\\
 &+\sum_{u=s}^{d}\sum_{r=1}^{u-2}
 (-1)^r\binom{u+r-s-1}{r-1}C_d^{\mathrm{ns}}(u,v,r,c).
\end{split}
\end{equation}
For $0\le\ell\le d-3$, the Newton coefficients are
\begin{equation}\label{eq:boundary-newton-coefficients}
 \beta_{d,s,\ell}
 =\sum_{v=2}^{d}\sum_{c=\ell+1}^{v-2}
 \Delta_{d,s}(v,c)\binom{d-v-2}{c-1-\ell}.
\end{equation}

\begin{theorem}[Integer Newton correction law]\label{prop:meta}
For integers $d\ge4$, $2\le s\le d-2$ and $n\ge d-1$,
the coefficients are computable by a finite transfer depending only on $d$.
Separable heights contribute zero correction. The correction is
\[
 B_d^{(s)}(n)
 =\sum_{\ell=0}^{d-3}\beta_{d,s,\ell}\binom{n-d+1}{\ell},
 \qquad \beta_{d,s,\ell}\in\mathbb Z.
\]
\end{theorem}

\begin{proof}
The type counts are first compared with their polynomial continuations
at large column dimensions. Polynomiality extends the resulting signed identity
to the full onset, where a change of binomial basis gives the coefficients.

\medskip\noindent\emph{Separable contribution.}
A normalized separable height has a unique pair of path factors, each
normalized at its first cell. Their minima and maxima counts are
$r_-,r_+$ on the $m$-cell path and $c_-,c_+$ on the $n$-cell path.
The family count is zero if either path-count difference has magnitude
greater than one. Otherwise it is
\[
 (1+\delta_{r_-r_+})(1+\delta_{c_-c_+})
 \binom{m-2}{r_-+r_+-2}\binom{n-2}{c_-+c_+-2}.
\]
Lemmas~\ref{lem:dimid} and~\ref{lem:binom} give this formula
on all $m,n\ge2$, including path sizes with zero count. Summing over
the finitely many positive quadruples with $r_-c_-+r_+c_+=d$ gives
one polynomial on all rectangles. Its value at $m=s$ is already the
actual separable count, so separable heights contribute zero correction.

\medskip\noindent\emph{Type counts and continuation.}
By Lemma~\ref{lem:boundary} and transposition, the outer rows and
columns are occupied. An empty row run lies between consecutive occupied
rows, so there are at most $u-1$ possible gaps. If every gap contained
a nonempty frozen run, Lemma~\ref{lem:frozen} would give constant
signed steps through each run and its joins. All occupied rows would
have the same horizontal profile up to translation, and the frozen rows
would copy it. That would make the height separable. A non-separable
type therefore has two adjacent occupied rows, giving $r\le u-2$.
Transposition gives $c\le v-2$, and every occupied level contains an
extremum, so $u,v\le d$.

For $r,c\ge1$, Lemma~\ref{lem:joint-frozen} gives the actual type count
\[
 \binom{m-u-1}{r-1}\binom{n-v-1}{c-1}
 \qquad(m\ge u+r,\ n\ge v+c).
\]
An axis with no marked run has actual factor one only at its type
dimension, which is at most $d$. On $K_d$, where $m,n\ge2d-1$,
zero-run types vanish and all positive-run types are supported.
The finite sum of their product polynomials agrees there
with the non-separable count. Subtracting the separable count polynomial
from $p_d$ gives the same values on $K_d$, so polynomial uniqueness
identifies that difference with this finite sum.

\medskip\noindent\emph{Signed row discrepancy.}
At widths $n\ge2d-1$, types with $c=0$ have no actual
count at these widths and no eventual contribution. For a positive
row-run count, supported row factors cancel between the actual count
and its polynomial continuation. An unsupported row factor with $u<s$
has nonnegative top smaller than its lower index and is zero. For
$u\ge s$, subtracting its polynomial continuation from its zero actual
count gives
\[
 -\binom{s-u-1}{r-1}
 =(-1)^r\binom{u+r-s-1}{r-1}.
\]
Zero-row-run types instead contribute one row expansion when $u=s$.
These give the first term of~\eqref{eq:signed-boundary-spectrum}.
Grouping the row discrepancies by $v,c$ and multiplying by their
common column factor gives
\begin{equation}\label{eq:boundary-spectrum-law}
 B_d^{(s)}(n)=\sum_{v=2}^{d}\sum_{c=1}^{v-2}
 \Delta_{d,s}(v,c)\binom{n-v-1}{c-1}.
\end{equation}
Both sides are polynomials in $n$. Lemma~\ref{lem:uniform} identifies
the left side with the actual correction for $n\ge d-1$. Agreement
at all $n\ge2d-1$ therefore extends the summed identity to the stated
onset. Each summand has degree $c-1$, at most $d-3$ because
$c\le v-2$ and $v\le d$.

\medskip\noindent\emph{Newton coefficients.}
Splitting the upper argument as $(n-d+1)+(d-v-2)$ in
Vandermonde's identity for binomial polynomials gives
\[
 \binom{n-v-1}{c-1}
 =\sum_{\ell=0}^{c-1}\binom{n-d+1}{\ell}
   \binom{d-v-2}{c-1-\ell}.
\]
Substituting this expansion into~\eqref{eq:boundary-spectrum-law}
and interchanging the finite sums gives the coefficient
formula~\eqref{eq:boundary-newton-coefficients}. Every type count
and every binomial value there is an integer, so
$\beta_{d,s,\ell}\in\mathbb Z$.
The finite transfer in Section~\ref{app:canonical-spectrum} computes
the type spectrum. Every enumerated grid has at most $2d-1$ rows and
columns, independently of $n$. This bounds the grids, not the dimensions
of the transfer matrix.
\end{proof}

\subsubsection{Canonical type-spectrum transfer}\label{app:canonical-spectrum}

The signed law~\eqref{eq:boundary-spectrum-law} is constructive once
$C_d^{\mathrm{ns}}(u,v,r,c)$ is computed. For fixed $d$, the transfer
below enumerates the complete non-separable joint-type spectrum,
including types with no marked row or column run. It uses the grid-height
lift, the Frozen and Boundary Lemmas and the joint expansion, not the
general total-degree theorem.

For each row count $R$ between two and $2d-1$, the column alphabet consists of the
$3\cdot2^{R-1}$ proper three-colourings of a column path. Two columns
are compatible when they differ in every cell. A compatible triple
determines the row mask of extrema in its middle column by testing
whether all present neighbour colours agree. The endpoint test omits
the absent column. The height lift identifies these masks with the
height extrema, including the top and bottom rows. The first
column's first colour is zero, so each normalized height has one
colouring and no symmetry division is needed.

A state stores its last two columns, the union of the extremum-row
masks of completed columns and whether the last completed column was
empty. It also records whether adjacent columns have had different
vertical increment vectors. Each increment is the unique signed unit
step with the given nonzero colour residue. Equal vertical increments
make every column a constant translate of the first. Conversely, an
additive row-plus-column height has those equal increments. The flag
therefore tests non-separability.

Completing a column with extremum mask $A$ contributes
$x^{|A|}y$ when occupied and $z$ when empty. Two consecutive
empty columns are disallowed, and the row mask is updated at each completion.
The initial state comes from the first compatible pair, with its increment flag
set from that pair and its left column completed by the endpoint test. Appending a column
completes the preceding middle column and
updates the increment flag from the newly adjacent pair. At termination,
the other endpoint test completes the last column. With two columns
there is no interior transition, and these two endpoint tests count
both columns.

Both endpoint columns must be occupied. After the last mask has
been included, the terminal conditions require the non-separability flag, exactly $u$ occupied
rows, occupied outer rows and no adjacent empty rows. Empty rows and
columns are then isolated. Lemma~\ref{lem:frozen} and transposition
give each its constant signed step. Joint contraction keeps the first
level of each empty run, so it deletes nothing and fixes each accepted
height. Conversely, every non-separable joint type with these parameters
has occupied outer levels and isolated empty levels, so it passes these
tests. There are $r=R-u$ empty rows, and the exponents of $y,z$
count $v,c$. The accepted heights are therefore precisely the canonical
joint types. Each accepted column path recovers one type, and every type
recovers its unique path, including when a run count is zero.

The interior transfer is $T_R(x,y,z)$, the initial
row vector is $a_R(x,y)$ and the terminal vector just
described is $b^{\mathrm{ns}}_{R,u}(x,y)$. Their endpoint weights have no $z$ factor because the
Boundary Lemma makes both endpoint columns occupied. Summing the
interior paths and extracting their weights gives
\[
 C_d^{\mathrm{ns}}(u,v,r,c)
 =[x^dy^vz^c]\,
 a_{u+r}(x,y)(I-T_{u+r}(x,y,z))^{-1}
 b^{\mathrm{ns}}_{u+r,u}(x,y).
\]
At $x=0$, a transition can complete only an empty column, so the
forbidden-consecutive-empty rule gives $T_R(0,y,z)^2=0$. The inverse
therefore exists as a formal $x$-series over $\mathbb Z[y,z]$.
At degree $d$, each occupied column contributes at least one extremum,
so there are at most $d$ occupied columns. The occupied endpoints and
isolated empty columns leave at most $d-1$ empty columns between them.
Only widths at most $2d-1$ contribute. Finite path propagation can
therefore extract the spectrum without a symbolic matrix inverse or
growing-width interpolation. Substitution into
\eqref{eq:signed-boundary-spectrum}, followed by
\eqref{eq:boundary-newton-coefficients}, computes the correction
coefficients. The state size can grow exponentially in $d$, and no
efficiency bound is asserted.

This construction is a finite coefficient algorithm, not an elementary
formula in named combinatorial numbers. Stronger
degree bounds must retain the signed sum, since individual types can
cancel. At the first excluded dimension, the complete correction
has a closed form in Baxter numbers.

\subsection{The first excluded dimension and Baxter permutations}
For $d\ge4$, the row dimension $d-2$ is one below the uniform threshold
$d-1$. Its correction is expressed using Baxter permutations.
For $k\ge1$, the number $\mathrm{Bax}(k)$ counts
Baxter permutations of size $k$, defined by avoidance of
$2\text{-}41\text{-}3$ and $3\text{-}14\text{-}2$~\cite[Definition 6.3]{Felsner2011}.
The two digits without a dash must occupy consecutive positions.
\begin{maintheorem}[Degree count at dimension $d-2$]\label{thm:first-boundary}
For integers $d\ge4$ and $n\ge d-1$,
\begin{equation}\label{eq:first-boundary}
  E_d(d-2,n)=p_d(d-2,n)
    -4(n-d+1)\mathrm{Bax}(d-3).
\end{equation}
\end{maintheorem}

The correction comes from the coefficient of $z^{d-2}$ in the polynomial
part of the width series. Four endpoint classes determine this coefficient,
and a Baxter correspondence evaluates their counts.

\medskip\noindent\emph{The transfer coefficient.}
The entry of the transfer of Lemma~\ref{lem:ratGF} with row
$(\sigma,\tau)$ and column $(\tau,\nu)$ records the extension
from the first pair to the second. The operators
$T_m(x)$ and its coefficients are restricted to the rotation-invariant
count block of Lemma~\ref{lem:quotient}. Vectors retain their coordinates
on all coloured pairs, and scalar products sum those coordinates. In
particular $T_m(0)=\Pi$ on this space. The degree-one coefficient is $B_1=[x]T_m(x)$, and $Q=I-\Pi$. The operator $B_1$ records a middle column with one extremum, and $Q$
deletes frozen pairs.

For a boundary extremum count $j\ge1$, the unpinned left indicator is
$\widetilde{\mathbf u}_j=3\bar{\mathbf u}_j$, and the right indicator is
$\mathbf v_j$, in the notation preceding Lemma~\ref{lem:frozenbdy}.
The factor three removes the averaged root pin. Thus products of these
vectors and transfer operators count raw colouring paths. Division by
three counts normalized heights. The fixed-point identities
\eqref{eq:frozenid} show that both one-extremum boundary indicators
are supported on frozen pairs. They satisfy
\[
  \widetilde{\mathbf u}_1^\top\Pi
    =\widetilde{\mathbf u}_1^\top,\qquad
  \Pi\mathbf v_1=\mathbf v_1.
\]

For $N=m-1$ and $1\le k<N$, superscripts record the left and right
extremum counts. The four raw contractions are
\begin{align*}
  W_k^{11}&=\widetilde{\mathbf u}_1^\top(B_1Q)^kB_1\mathbf v_1,&
  W_k^{12}&=\widetilde{\mathbf u}_1^\top(B_1Q)^k\mathbf v_2,\\
  W_k^{21}&=\widetilde{\mathbf u}_2^\top(QB_1)^k\mathbf v_1,&
  W_k^{22}&=\widetilde{\mathbf u}_2^\top Q(B_1Q)^{k-1}\mathbf v_2.
\end{align*}

\begin{lemma}[Polynomial-part reduction]\label{lem:baxter-reduction}
For $d\ge4$, $m\ge d-1$ and $k=d-3$, the coefficient of $z^{d-2}$ in the polynomial part of $F_{m,d}(z)=\sum_{n\ge2}E_d(m,n)z^n$ is
\[
 \frac{W_k^{11}-W_k^{12}-W_k^{21}+W_k^{22}}{3}.
\]
\end{lemma}

\begin{proof}
Polynomial division separates $F_{m,d}$ into a polynomial and a proper
rational fraction. The proper fraction tends to zero at infinity, so its
Laurent series has no nonnegative powers. The required coefficient of
the polynomial part is therefore the coefficient of $z^{d-2}$ in the
Laurent expansion of $F_{m,d}$ at infinity.

\medskip\noindent
The constant term of the transfer on the count block is $\Pi$,
representing zero-extremum steps. Separating these from steps with a
positive extremum count gives
\[
  T_m(x)=\Pi+V(x),\qquad V(0)=0.
\]
Consecutive zero-extremum steps contribute the geometric series defining
the resolvent $R_0$ of Lemma~\ref{lem:quotient}. Since $\Pi$ is a projection
and $Q=I-\Pi$,
\[
  R_0=\sum_{j\ge0}z^j\Pi^j
      =I+\frac{z\Pi}{1-z}
      =Q+\frac{\Pi}{1-z}.
\]
A transfer path consists of runs with no extrema separated by factors
with positive extremum count. The initial and final runs may be empty.
Summing each run gives a factor $R_0$.
For the unpinned endpoint series
$\widetilde{\mathbf u}(x)=\sum_{i\ge1}x^i\widetilde{\mathbf u}_i$ and
$\mathbf v(x)=\sum_{i\ge1}x^i\mathbf v_i$, this grouping gives
\[
  F_{m,d}(z)=\frac{z^2}{3}[x^d]\,
  \widetilde{\mathbf u}(x)^\top R_0
  \sum_{t\ge0}\bigl(zV(x)R_0\bigr)^t\mathbf v(x).
\]

\medskip\noindent
Every endpoint has positive extremum count, as does each factor of
$V(x)$, so only $t\le d-2$ can contribute to $[x^d]$.
At infinity a $Q$ term is constant, while $\Pi/(1-z)$ has leading term
$-\Pi/z$.
For endpoint extremum counts $\iota,\kappa\ge1$, the number $t$ of positive
transfer factors is at most $d-\iota-\kappa$ because each factor contributes
at least one power of $x$. The fixed-point identities~\eqref{eq:frozenid},
$\widetilde{\mathbf u}_1^\top\Pi=\widetilde{\mathbf u}_1^\top$ and
$\Pi\mathbf v_1=\mathbf v_1$, force the adjacent resolvent at each
degree-one endpoint to use $\Pi/(1-z)$.
If both endpoints have degree one, they could share a resolvent only when
$t=0$, which would give $d=2$. Thus the forced resolvents are distinct
throughout the stated range. Brackets denote indicators. Their number is $b_{\partial}=[\iota=1]+[\kappa=1]$.
Each forced resolvent contributes $-1/z$ at leading order. Every other
resolvent has leading term $Q$. Including the $z^{2+t}$ prefactor, each
contributing term therefore has highest Laurent power at most
\begin{align*}
  2+t-b_{\partial}
    &\le2+d-\iota-\kappa-[\iota=1]-[\kappa=1]\\
    &\le d-2.
\end{align*}
An endpoint count of one has total cost two in this bound because its forced
resolvent removes one power of $z$. A count of two also has cost two.
Higher endpoint counts have greater cost. Reaching $z^{d-2}$ therefore
requires $\iota,\kappa\in\{1,2\}$, $t=d-\iota-\kappa$, only
$B_1$ transfer factors, and $Q$ at every unforced resolvent.
The sign is $(-1)^{b_{\partial}}$. Removing the forced endpoint $\Pi$
factors by the fixed-point identities leaves, for $k=d-3$, the four
leading scalars
\[
\begin{array}{ccl}
  (\iota,\kappa)&t&\text{leading scalar}\\
  (1,1)&k+1&
    \widetilde{\mathbf u}_1^\top(B_1Q)^kB_1\mathbf v_1=W_k^{11},\\
  (1,2)&k&
    -\widetilde{\mathbf u}_1^\top(B_1Q)^k\mathbf v_2=-W_k^{12},\\
  (2,1)&k&
    -\widetilde{\mathbf u}_2^\top(QB_1)^k\mathbf v_1=-W_k^{21},\\
  (2,2)&k-1&
    \widetilde{\mathbf u}_2^\top Q(B_1Q)^{k-1}\mathbf v_2=W_k^{22}.
\end{array}
\]
For $d=4$, the last row has $k=1$ and no transfer factor, giving
$W_1^{22}=\widetilde{\mathbf u}_2^\top Q\mathbf v_2$.
Summing the four leading scalars and dividing by the pin factor three
gives the assertion.
\end{proof}

\medskip\noindent\emph{Columns with one extremum.}

Each $B_1$ factor in the four contractions selects a column with one
extremum. The local compatibility test determines which such columns can
occur consecutively. A triple of height columns is compatible when its side-by-side array has
unit differences on every horizontal and vertical grid edge.
For three consecutive height columns $f,g,h$, the horizontal
offsets and middle-column vertical steps are
\[
  \varepsilon_i=g_i-f_i,\qquad
  \eta_i=h_i-g_i,\qquad w_i=g_{i+1}-g_i.
\]
The offsets have row indices $1,\ldots,m$ and the vertical steps have
edge indices $1,\ldots,m-1$. All belong to $\{-1,1\}$. Unit vertical steps
in the two outer columns require
\begin{equation}\label{eq:boundary-local-steps}
  w_i-\varepsilon_{i+1}+\varepsilon_i\in\{-1,1\},\qquad
  w_i+\eta_{i+1}-\eta_i\in\{-1,1\}.
\end{equation}
For a sign $\xi\in\{-1,1\}$, a change $\xi$ to $-\xi$ in $\varepsilon$
forces $w_i=-\xi$, whereas the same change in $\eta$ forces $w_i=\xi$.

A switch is a change between consecutive entries of an offset word,
indexed by the intervening vertical edge. A one-switch word has the form
$q^r(-q)^{m-r}$, with sign $q\in\{-1,1\}$ and switch edge
$1\le r\le m-1$. For a sign $\xi$, the word $\xi^*$ means a possibly empty
run, and $\xi^+$ a nonempty run.

On a maximal constant nonzero row interval $[l,u]$ of $\eta-\varepsilon$,
the offsets are $\varepsilon=\xi$ and $\eta=-\xi$ for one sign $\xi$.
Its extended vertical word consists of the steps $w_{l-1},\ldots,w_u$.
If $l=1$, the missing upper step is represented by $w_0=\xi$. If $u=m$,
the missing lower step is represented by $w_m=-\xi$.
These virtual steps make a boundary extremum satisfy the same change-of-sign
test as an interior extremum, since both horizontal neighbours have height
$g_i-\xi$ on the interval.

\begin{lemma}[One-extremum interval]\label{lem:boundary-interval}
A compatible triple of height columns has exactly one extremum in its
middle column iff $\eta-\varepsilon$ has one maximal constant
nonzero run and the extended vertical word along that run consists of
$\xi$ steps followed by $-\xi$ steps, where the incoming offset on the run is
$\xi$. Each offset word in such a triple has at most two switches.
\end{lemma}

\begin{proof}
Each maximal constant nonzero run forces an extremum. Indeed, on such a
row interval $[l,u]$, the offsets are $\varepsilon=\xi$, $\eta=-\xi$.
Equation~\eqref{eq:boundary-local-steps} forces $w_{l-1}=\xi$ when the upper
edge exists and $w_u=-\xi$ when the lower edge exists. The same conclusions
hold for adjacent runs of opposite nonzero sign. The virtual endpoint values
give these same signs at the sheet boundaries.

\medskip\noindent
Both horizontal neighbours of row $i$ in the interval equal $g_i-\xi$.
That row is an extremum precisely when the extended vertical word changes
from $\xi$ to $-\xi$ there. Each run therefore contributes at least one
extremum, and exactly one contribution requires all $\xi$ steps before all
$-\xi$ steps. Outside the runs the two horizontal neighbours are on opposite
sides of $g_i$, so no row there is an extremum. This proves the equivalence.

\medskip\noindent
A common offset switch outside the changed interval would force the same
vertical step to have both signs in \eqref{eq:boundary-local-steps}.
Each offset is constant on the changed interval and separately above
and below it. Its switches can therefore occur only at the two interval
boundaries, giving at most two switches.
\end{proof}

For the vertical word of a left boundary column with constant rightward offset $\alpha$,
the virtual steps are $w_0=-\alpha$, $w_m=\alpha$. An extremum is exactly
a change from $-\alpha$ to $\alpha$. Hence a single extremum requires
the word $(-\alpha)^*\alpha^*$. At a right boundary column entered by offset
$\omega$, the horizontal neighbour increment is $-\omega$. Its virtual
steps are therefore $w_0=\omega$, $w_m=-\omega$, and the single-extremum
word is $\omega^*(-\omega)^*$.

\begin{lemma}[One-switch path profiles]\label{lem:boundary-profiles}
For a height array with at least two columns, one extremum in every
middle column and one or two extrema in each boundary column,
suppose that every consecutive-column pair is nonfrozen, except possibly the
first pair when the left boundary column has one extremum and the last pair
when the right boundary column has one extremum.
Then every nonfrozen pair has one switch, and all such pairs have the
same first offset sign $q$. Passing through such a pair changes its
switch edge from vertical step $q$ to $-q$ and leaves all other vertical
steps unchanged.
\end{lemma}

\begin{proof}
Lemma~\ref{lem:boundary} puts an extremum in every offset block of a
boundary pair. With two boundary extrema, there are at most two blocks,
so a nonfrozen pair has one switch. With one boundary extremum, the
pair is frozen. Every other pair meets a middle column, so
Lemma~\ref{lem:boundary-interval} bounds its switches by two.

\medskip\noindent
Suppose that an incoming word at a middle column has two switches.
The proof of Lemma~\ref{lem:boundary-interval} confines both switches
to the boundaries of the unique changed interval. This interval is
therefore the middle run of the incoming word, with successive signs
$-\xi,\xi,-\xi$. Reversing this middle run makes the outgoing word
constant $-\xi$, so the successor pair is frozen. The hypotheses permit
this only at the final pair, with one extremum in the right boundary
column. The changed interval is interior, so the middle column $g$ has
a step $\xi$ at its upper edge and a later step $-\xi$ at its lower edge.
The final column $h=g-\xi$ has the same vertical steps. Its incoming
offset is $\omega=-\xi$, whereas its one-extremum boundary test requires
$(-\xi)^*\xi^*$. The two displayed steps contradict that order. Hence
every nonfrozen pair has one switch.

\medskip\noindent
Two one-switch words of opposite signs have at least two maximal constant
nonzero difference runs. For distinct switch edges these are separated
by a zero run. For equal switch edges they are adjacent runs of opposite
nonzero sign. The one-extremum interval criterion excludes both cases.
The nonfrozen pairs are consecutive because frozen pairs are permitted
only at the ends, so their first offset signs agree and equal $q$.
At an outgoing switch edge, the second identity in
\eqref{eq:boundary-local-steps} forces $w_r=q$ and changes that step
to $-q$. Every other vertical step is unchanged, so no edge can be
changed twice.
\end{proof}

A vertical edge is available for consumption when its step is $q$. Its change to $-q$ is
a consumption. For a nonfrozen boundary pair in Lemma~\ref{lem:boundary-profiles},
the horizontal offset has two blocks, separated by its switch edge $r$.
Lemma~\ref{lem:boundary} gives at least one extremum in each block.
When the boundary column has two extrema, each block therefore has exactly one.

At a left end, the column before the one-switch passage has vertical step
$q$ at edge $r$, by Lemma~\ref{lem:boundary-profiles}. This is the shared
edge between its upper block, whose horizontal offset is $q$, and its lower
block, whose offset is $-q$. The extended word of the upper block starts with
the virtual step $-q$ and ends with the actual step $q$. The lower block
starts with that same step $q$ and ends with the virtual step $-q$.
The one-extremum boundary test gives the extended block words
\[
  (-q)^+q^+\qquad\text{and}\qquad q^+(-q)^+.
\]
Each includes its virtual endpoint, so both displayed runs are nonempty.
Joining the words at their shared actual edge counts that edge once.
Removing the outer virtual steps can empty either outer run, yielding
the real first-column word $(-q)^*q^+(-q)^*$.
Its available-step interval remains nonempty because it contains the actual edge $r$.

At a right end, the column after the passage has step $-q$ at edge $r$.
The horizontal neighbour increments are $-q$ in its upper block and $q$
in its lower block. The upper extended word starts with virtual step $q$
and ends with actual step $-q$. The lower starts with that same step $-q$
and ends with virtual step $q$. Their extended block words are
\[
  q^+(-q)^+\qquad\text{and}\qquad (-q)^+q^+.
\]
The displayed runs are nonempty, including the two outer virtual steps.
Joining at the shared actual edge and then removing those virtual steps
gives the real last-column word $q^*(-q)^+q^*$. Its middle run remains nonempty because it contains
edge $r$, so its available-step set is a prefix union a suffix.

\medskip\noindent\phantomsection
\emph{Baxter correspondence.}\label{subsec:count-baxter}

The superscripts specify the numbers of extrema at the left and right
boundaries. For $ab\in\{11,12,21,22\}$, the set $\mathcal H_{m,k}^{ab}$
consists of normalized height arrays with $m$ rows and $k+5-a-b$ columns,
$a$ extrema in the first column, $b$ in the last, and one in every middle
column. Each consecutive pair is nonfrozen except a pair adjacent to a
one-extremum boundary, which is frozen by Lemma~\ref{lem:boundary}.
These are geometric conditions, independent of the contractions $W_k^{ab}$.
Each $B_1$ factor counts a middle column, while the $Q$ factors impose
the nonfrozen-pair conditions. The endpoint vectors count the specified
boundary extrema. The transfer therefore counts precisely this class,
and the three global colour rotations give
\[
 |\mathcal H_{m,k}^{ab}|=W_k^{ab}/3.
\]
There are $k$ nonfrozen pairs in each array. By
Lemma~\ref{lem:boundary-profiles}, each has one switch and their first
offset signs agree.

For these classes the boundary data consist of a sign $q$, an initial
available-step set $S_{\rm in}$ and a final available-step set $S_{\rm out}$ in
$\{1,\ldots,N\}$, together with a side label at each boundary of type one.
The permitted shapes are
\[
\begin{array}{lll}
  \toprule
  \text{boundary type}&S_{\rm in}&S_{\rm out}\\ \midrule
  11&\text{prefix or suffix}&\text{prefix or suffix}\\
  12&\text{prefix or suffix}&\text{prefix union suffix}\\
  21&\text{interval}&\text{prefix or suffix}\\
  22&\text{interval}&\text{prefix union suffix}\\
  \bottomrule
\end{array}
\]
In each case $S_{\rm out}\subset S_{\rm in}$ and
$|S_{\rm in}\setminus S_{\rm out}|=k$ are required. A type-one side label is prefix
or suffix and must describe its set. It remains part of the data when
the set is full or empty. The left label gives shift $\alpha=-q$ for
prefix and $\alpha=q$ for suffix. The right gives $\omega=q$ for prefix
and $\omega=-q$ for suffix. These four data spaces are
$\mathcal D_{N,k}^{ab}$. The shape and containment conditions make the
consumed set $D=S_{\rm in}\setminus S_{\rm out}$ an interval.

The interval test explains why Baxter permutations enter.
An array records its boundary data and lists the consumed edges
$p_1,\ldots,p_k$ in the order of the column passages. Replacing each edge
by its rank in $D$ gives the permutation $\pi$.
For two successive consumptions $p_j$ and $p_{j+1}$, their switch-edge
positions are $r=p_j$ and $r'=p_{j+1}$. These edges need not be adjacent.
If $r<r'$, their changed row interval is $[r+1,r']$, with incoming offset $-q$.
Lemma~\ref{lem:boundary-interval} requires its extended vertical word on
edges $[r,r']$ to be $(-q)^*q^*$, hence its availability indicators to be
$0^*1^*$. If $r>r'$, the changed row interval is $[r'+1,r]$, the incoming
offset is $q$, and the required indicator word is $1^*0^*$.

Before consumption $j+1$, edge $p_j$ is unavailable and edge $p_{j+1}$
is available. Every edge strictly between them lies in $D$ and was either
consumed before index $j$ or will be consumed after index $j+1$.
Failure of the required indicator order therefore gives indices $i<j$
and $j+1<\ell$ satisfying one of
\begin{equation}\label{eq:baxter-pattern-test}
  \pi_j<\pi_\ell<\pi_i<\pi_{j+1},\qquad
  \pi_{j+1}<\pi_i<\pi_\ell<\pi_j,
  \qquad i<j,\quad j+1<\ell.
\end{equation}
These are precisely the two forbidden patterns in the cited definition
of Baxter permutations. Their central entries occupy consecutive
consumption positions $j,j+1$. Conversely, either pattern violates the
required indicator order. The set $\mathfrak B_k$ consists of Baxter
permutations of size $k$.

\begin{lemma}[Consumption correspondence]\label{lem:boundary-bijection}
For $m\ge3$, $N=m-1$, $1\le k<N$ and each boundary type $ab\in\{11,12,21,22\}$, there is a bijection
\[
  \mathcal H_{m,k}^{ab}
  \longleftrightarrow
  \mathcal D_{N,k}^{ab}\times\mathfrak B_k.
\]
\end{lemma}

\label{app:baxter-details}
\begin{proof}
The boundary data fix the first column and each constant-shift end
column. The permutation fixes all remaining column passages.

\medskip\noindent\emph{Encoding.}
The map $\Phi$ sends an array in $\mathcal H_{m,k}^{ab}$ to its boundary
data and the permutation given by the ranks of its consumed edges.
The boundary profiles put the data in $\mathcal D_{N,k}^{ab}$.
By \eqref{eq:baxter-pattern-test}, a forbidden pattern would violate the
interval criterion for two consecutive consumptions. Thus the permutation
is Baxter, so the encoding gives a well-defined map
\[
 \Phi:\mathcal H_{m,k}^{ab}
 \longrightarrow\mathcal D_{N,k}^{ab}\times\mathfrak B_k.
\]

\medskip\noindent\emph{Reconstruction.}
The map $\Psi$ takes boundary data $\delta\in\mathcal D_{N,k}^{ab}$ and
a permutation $\pi\in\mathfrak B_k$. For $1\le j\le k$, the edge $p_j$
has rank $\pi_j$ in $D=S_{\rm in}\setminus S_{\rm out}$.
The first column has height zero at its
first cell, vertical step $q$ on $S_{\rm in}$ and step $-q$ elsewhere.
A left boundary of type one adds a column by the constant shift $\alpha$.
For each $p_j$, a one-switch passage adds $q$ above that edge and $-q$
below it. A right boundary of type one adds the final column by the
constant shift $\omega$. A type-two boundary uses the first or last
one-switch pair itself, with no extra column. This constructs
$k+5-a-b$ columns, as required.
Each consumed edge still has step $q$ when used. Its new step is $-q$,
and every other vertical step is unchanged, so all grid edges have unit
increments. Baxter avoidance gives the interval criterion between
consecutive consumptions and hence one extremum in each intervening
middle column.

\medskip\noindent
At a type-one left boundary, a prefix label makes all availability
indicators from edge one through $p_1$ equal to one. A suffix label makes
them one from $p_1$ through edge $N$. The corresponding changed interval
therefore satisfies Lemma~\ref{lem:boundary-interval}, giving one
extremum in the adjacent middle column. At a type-one right boundary,
$p_k\notin S_{\rm out}$ makes the corresponding terminal interval
all zero, with the same conclusion.

\medskip\noindent
Every consumed edge belongs to $D=S_{\rm in}\setminus S_{\rm out}$.
At a type-two left boundary, $p_1$ therefore lies in the middle
$q$-run of the initial word $(-q)^*q^+(-q)^*$.
At a type-two right boundary, $p_k$ lies in the middle $-q$-run of
the final word $q^*(-q)^+q^*$. In both cases each offset block has
one extremum by the boundary test. At a type-one boundary, the
side-label shift $\alpha$ or $\omega$ matches the prefix or suffix
to its single-extremum boundary word. Every consumption gives a
nonfrozen pair, and the only other pairs are the required constant-shift
end pairs. Thus the reconstruction gives a well-defined map
\[
 \Psi:\mathcal D_{N,k}^{ab}\times\mathfrak B_k
 \longrightarrow\mathcal H_{m,k}^{ab}.
\]

\medskip\noindent\emph{Inverse identities.}
Encoding the reconstruction recovers its sign, boundary sets, side
labels and consumption order. In particular, a full or empty type-one
set retains its label through its constant shift. Conversely, the
encoding of an array fixes its normalized first column through its
vertical steps and every later column through its horizontal offsets.
For each array $H\in\mathcal H_{m,k}^{ab}$ and data-permutation pair
$(\delta,\pi)\in\mathcal D_{N,k}^{ab}\times\mathfrak B_k$, these two
recoveries give
\[
 \Phi\bigl(\Psi(\delta,\pi)\bigr)=(\delta,\pi),\qquad
 \Psi\bigl(\Phi(H)\bigr)=H.
\]
\end{proof}

The bijection separates the permutation from its boundary data. Counting
those data therefore evaluates all four contractions.

\begin{lemma}[End-column multiplicities]\label{lem:boundary-multiplicities}
For $m\ge3$, $N=m-1$ and $1\le k<N$,
\begin{align*}
  W_k^{11}&=12(N-k+3)\mathrm{Bax}(k),\\
  W_k^{12}=W_k^{21}&=\bigl(24(N-k)+12\bigr)\mathrm{Bax}(k),\\
  W_k^{22}&=24(N-k)\mathrm{Bax}(k).
\end{align*}
\end{lemma}

\begin{proof}
It remains to count the boundary data. The correspondence counts
normalized heights, whereas $W_k^{ab}$ counts all three global colour
rotations. Thus Lemma~\ref{lem:boundary-bijection} gives
\[
  W_k^{ab}=3\,|\mathcal H_{m,k}^{ab}|
          =3\,|\mathcal D_{N,k}^{ab}|\mathrm{Bax}(k).
\]
The following counts are for one fixed sign $q$. The factor two for its signs is included
only at the end.

\medskip\noindent\emph{Two type-one boundaries.}
For type $11$, the two same-side labelings each give $N-k+1$ choices
of sizes. Each of the two opposite-side labelings permits an empty
final set or a full initial set. Those choices are distinct because
$k<N$, so the number of data tuples for fixed $q$ is
\[
 2(N-k+1)+4=2(N-k+3).
\]

\medskip\noindent\emph{One boundary of each type.}
For type $12$ and either initial side, there is one choice with empty
final set, $N-k$ with a nonempty final set confined to that side, and
$N-k$ with a nonempty part on the opposite side. The last case forces
the initial set to be full. The two initial labels therefore give
\[
 2\bigl(1+(N-k)+(N-k)\bigr)=4(N-k)+2.
\]
Horizontal reflection exchanges types $12$ and $21$. The pair-reversal
operator is $J$. Reflection
preserves middle-column extrema and frozen status, so
\[
  JB_1^\top J=B_1,\qquad JQJ=Q,\qquad
  J\widetilde{\mathbf u}_i=\mathbf v_i.
\]
Transposing the scalar $W_k^{12}$ and inserting these identities gives
$W_k^{21}$ with its stated $Q$ factors.

\medskip\noindent\emph{Two type-two boundaries.}
For type $22$, an empty final set gives $N-k+1$ positions for the
consumed interval. A nonempty part on exactly one side gives $N-k$
choices on each side. Nonempty parts on both sides force the initial
interval to be full and give $N-k-1$ choices. Thus the total is
\[
 (N-k+1)+2(N-k)+(N-k-1)=4(N-k).
\]
The last summand is zero when $N-k=1$.
Multiplying the fixed-sign counts by two for $q$, three for global
colour rotations and $\mathrm{Bax}(k)$ for consumption orders gives
the three asserted formulas.
\end{proof}

\begin{proof}[Proof of Theorem~\ref{thm:first-boundary}]
For $m\ge d-1$, Lemma~\ref{lem:baxter-reduction} gives the coefficient
of $z^{d-2}$ in the polynomial part of $F_{m,d}$. Here $N=m-1$ and
$k=d-3$ satisfy $1\le k<N$, so the end-column multiplicities give
\begin{align}\label{eq:boundary-leading-coefficient}
  \frac{W_k^{11}-W_k^{12}-W_k^{21}+W_k^{22}}{3}
    &=-4(N-k-1)\mathrm{Bax}(k)\notag\\
    &=-4(m-d+1)\mathrm{Bax}(d-3).
\end{align}
\medskip\noindent
By Lemma~\ref{lem:quotient}, the proper fraction has poles only at $z=1$.
Its partial fractions have polynomial coefficients in the column count,
since for every positive integer $r$,
\[
  [z^n](1-z)^{-r}=\binom{n+r-1}{r-1}.
\]
The Laurent bound makes the polynomial part have degree at most $d-2$,
so it contributes nothing on $n\ge d-1$. On those widths, the coefficient
polynomial of the proper fraction equals $E_d(m,n)$ and hence $p_d(m,n)$
by Theorem~\ref{thm:poly}. Agreement at infinitely many integer widths
identifies these two polynomials. At width $d-2$, the full generating
function has the additional coefficient from its polynomial part.
Thus \eqref{eq:boundary-leading-coefficient} gives
\[
 E_d(m,d-2)-p_d(m,d-2)
 =-4(m-d+1)\mathrm{Bax}(d-3).
\]
Transposition, followed by renaming $m$ as $n$, proves
\eqref{eq:first-boundary}.
\end{proof}

At $n=d-1$, the correction vanishes. The identity permutation is Baxter,
so $\mathrm{Bax}(d-3)>0$. For $n\ge d$, the polynomial $p_d(d-2,n)$
therefore exceeds the count by $4(n-d+1)\mathrm{Bax}(d-3)$.
Transposition turns this discrepancy into an obstruction to an earlier
width onset.

\begin{corollary}[Exact row onset]\label{cor:first-boundary-onset}
For integers $d\ge4$ and $m\ge d$, the equality
$E_d(m,n)=p_d(m,n)$ holds for every $n\ge d-1$ and fails at $n=d-2$.
\end{corollary}

\begin{proof}
Agreement follows from Theorem~\ref{thm:poly}. Transposition and
Theorem~\ref{thm:first-boundary} give
\[
  E_d(m,d-2)-p_d(m,d-2)
  =-4(m-d+1)\mathrm{Bax}(d-3)<0.
\]
Thus the sufficient onset $d-1$ cannot be lowered for these rows. A polynomial agreeing from an earlier width would agree with $p_d(m,\cdot)$ at infinitely many widths and hence equal it, contradicting this discrepancy.
\end{proof}

\begin{corollary}[Sharp rectangular threshold]\label{cor:first-boundary-region}
For every integer $d\ge4$, the polynomial $p_d$ does not agree with
$E_d(m,n)$ on all integer pairs satisfying $m,n\ge d-2$.
\end{corollary}

\begin{proof}
Corollary~\ref{cor:first-boundary-onset} at $m=d$ gives failure at
$(d,d-2)$, which belongs to the proposed larger region.
\end{proof}

\subsubsection{Boundary data}\label{app:boundary-data}

For $d\ge4$, $2\le s<d-1$ and $n\ge d-1$, the correction is
$B_d^{(s)}(n)=E_d(s,n)-p_d(s,n)$, as in Section~\ref{sec:subthreshold}.
Here $E_d(s,n)$ counts degree-$d$ Miura-OFG vertices.
Table~\ref{tab:corr} gives the certified corrections for $4\le d\le7$.
Section~\ref{sec:data} gives the finite certification protocol.

\begin{table}[H]
\centering
\setlength{\tabcolsep}{7pt}
\begin{tabular}{@{}llll@{}}
\toprule
$d$ & $s$ & $B_d^{(s)}(n)$ & Degree\\
\midrule
$4$ & $2$ & $12-4n$ & $1$\\
$5$ & $2$ & $40$ & $0$\\
$5$ & $3$ & $32-8n$ & $1$\\
$6$ & $2$ & $\tfrac43n^3-62n^2+\tfrac{1538}{3}n-1056$ & $3$\\
$6$ & $3$ & $-6n^2+30n+80$ & $2$\\
$6$ & $4$ & $120-24n$ & $1$\\
$7$ & $2$ & $\tfrac23n^4-\tfrac{106}{3}n^3+\tfrac{1204}{3}n^2
            -\tfrac{1484}{3}n-3016$ & $4$\\
$7$ & $3$ & $6n^3-244n^2+2390n-6024$ & $3$\\
$7$ & $4$ & $-28n^2+196n+104$ & $2$\\
$7$ & $5$ & $528-88n$ & $1$\\
\bottomrule
\end{tabular}
\caption{Certified boundary corrections for $d=4,\dots,7$ on $n\ge d-1$.
The constant $B_5^{(2)}=40$ is the exception to the degree pattern $d-1-s$.}
\label{tab:corr}
\end{table}

For $n\ge7$, the certified degree-eight corrections used in
Section~\ref{sec:subthreshold} are
{\allowdisplaybreaks
\begin{align*}
  B_8^{(2)}(n)&=\tfrac15n^5-11n^4+33n^3+2055n^2
                 -\tfrac{72326}{5}n+12200,\\
  B_8^{(3)}(n)&=4n^4-227n^3+2916n^2-9709n-2284,\\
  B_8^{(4)}(n)&=20n^3-1042n^2+11834n-35640,\\
  B_8^{(5)}(n)&=-128n^2+1104n-504,\\
  B_8^{(6)}(n)&=2576-368n.
\end{align*}}

The $s=d-2$ rows in the table and degree-eight display also follow from
Theorem~\ref{thm:first-boundary}. Its degree-nine and degree-ten instances,
without assuming the split-remainder bound, are
\begin{align*}
  B_9^{(7)}(n)&=13504-1688n \qquad(n\ge8),\\
  B_{10}^{(8)}(n)&=74664-8296n \qquad(n\ge9).
\end{align*}

Among the certified rows of Table~\ref{tab:corr}, the correction degree is
$d-1-s$ except for $B_5^{(2)}=40$, whose degree is zero rather than two.
The additional certified $d=8$ rows have degrees $5,4,3,2,1$ at
$s=2,3,4,5,6$.

These are finite examples, with certification of the count polynomials and
the correction-node protocol in Section~\ref{sec:data}. They are not premises
of Theorem~\ref{thm:first-boundary}. Every displayed dimension-$d-2$ row
at $d=4,\ldots,10$ is also a direct instance of that theorem.

\subsection{Proved correction-degree bounds}

The Baxter identity evaluates the correction at $s=d-2$. At every short
dimension $s$, the integer Newton law gives the general degree bound $d-3$.
Primitive support and
switch counts give the combined upper bound
\[
 D_{d,s}=\min\!\left\{d-1-\min(s,4),\,
 d-2-\left\lceil\frac{s-1}{2}\right\rceil\right\}.
\]
The general bound and the quadratic refinement at $s=d-3$ are
deductive. The cubic refinement at $s=d-4$ is deductive for
$d=6,7,8$ and uses a four-column finite-state implication for
$d\ge9$. That implication covers arbitrary numbers of rows.

\subsubsection{Boundary degree proofs}\label{app:boundary-degree-proofs}

The signed spectrum retains one empty level per run. For degree bounds,
compressing every empty level instead leaves a primitive height whose
rows and columns are all occupied. Its uniform gaps can be stretched,
while its nonuniform gaps control where a correction can survive.
A marked empty level in the signed spectrum and a stretchable primitive
gap therefore use different dimension and expansion formulas.

For a primitive, the row discrepancy is its actual row-expansion factor
minus its eventual polynomial continuation. If there is no uniform row
gap, the continuation is zero because the actual factor vanishes at
every larger row dimension.

\begin{lemma}[Primitive expansion and discrepancy support]\label{lem:primitive-expansion}
Every normalized height on a rectangle has a unique primitive obtained
by compressing all extremum-free rows and columns. A primitive with
$u,v$ rows and columns and $r,c$ uniform gaps has dimensions at most
$d$ in each axis when it has $d$ extrema. If $r,c\ge1$, its expansion
count is
\[
 \binom{m-u+r-1}{r-1}\binom{n-v+c-1}{c-1}
 \qquad(m\ge u,\ n\ge v).
\]
With a zero gap count, the corresponding actual factor is one at its
primitive dimension and zero elsewhere. If $a=u-r-1$ counts its
nonuniform row gaps, its row discrepancy at $m=s$ can be nonzero only
when $s\le a+1$.
\end{lemma}

\begin{proof}
The join calculation of Lemma~\ref{lem:contract} also permits deleting
the last representative of an empty run, since the two signed joins have
the same sign, and translating the later heights leaves a unit join
of that sign. All retained neighbour increments and extrema are
unchanged. Deletions in the two axes commute by the additive-translation
calculation of Lemma~\ref{lem:joint-frozen}. The occupied levels are
therefore precisely the primitive levels, independent of order.

\medskip\noindent
A primitive gap is uniform if its signed step is constant across the
other axis. Stretching it by a nonnegative extra length preserves
extrema. The two ordered length lists reconstruct the original height,
and compression recovers those lists. Distributing the extra lengths
among the uniform gaps gives the weak-composition product. Every
primitive level contains an extremum, so $u,v\le d$. Normalization
and unit edge increments leave finitely many heights on these bounded
grids, giving finitely many primitives at fixed $d$.

\medskip\noindent
For $r\ge1$, the row continuation is
$\binom{s-u+r-1}{r-1}$. At $s\ge u$ it equals the actual count.
Below $u$, a value $s>u-r$ makes its top nonnegative and smaller
than $r-1$, so both the continuation and the actual count are zero.
Nonzero discrepancy therefore requires
\[
 s\le u-r=a+1.
\]
For $r=0$, only
$s=u$ can contribute. Summing the product continuations agrees with
$p_d$ on all sufficiently large rectangles by Theorem~\ref{thm:poly},
so the two polynomials coincide. Terms with no uniform column gap
vanish for large $n$. Polynomiality from $n=d-1$ then extends each
resulting correction identity to that full onset.
\end{proof}

\begin{lemma}[Local switch budget]\label{lem:local-switch-budget}
For a primitive height of degree $d$ with $a$ nonuniform row gaps
and $c$ uniform column gaps, $a\le2(d-c-1)$.
\end{lemma}

\begin{proof}
The primitive has $m$ rows and $n$ columns. The count $e_j$ is
the number of grid extrema in column $j$, and
$\kappa_j$ counts changes of sign in the offset across column gap $j$.
The outside counts $\kappa_0,\kappa_n$ are zero. A saddle is an interior grid cell whose
horizontal neighbours are both higher and vertical neighbours both lower,
or the reverse. Its column count is $S_j$. The grid boundary, traversed once,
is a cyclic unit-step walk. The count $\tau_j$ records its strict extrema in column $j$.

\medskip\noindent
For consecutive columns $f,g,h$, the offsets are
$\varepsilon_i=g_i-f_i$, $\eta_i=h_i-g_i$, and the vertical steps of $g$
are $w_i=g_{i+1}-g_i$. Missing horizontal offsets and the missing steps
$w_0,w_m$ are zero. At an interior cell, the product
$(\eta_i-\varepsilon_i)(w_i-w_{i-1})/4$ is the extremum indicator minus
the saddle indicator. Summing this product over the column defines $J_j$.
At a non-corner boundary cell it is the extremum indicator minus
half the perimeter-extremum indicator. At a corner there is an additional
term $-1/4$. Consequently,
\[
 J_j=\frac14\sum_{i=1}^{m}(\eta_i-\varepsilon_i)(w_i-w_{i-1})
 =\begin{cases}
 e_j-S_j-\tau_j/2,&1<j<n,\\
 e_j-\tau_j/2-1/2,&j\in\{1,n\}.
 \end{cases}
\]

\medskip\noindent
Compatibility gives $w_i\Delta\varepsilon_i=2$ at every incoming
switch and $w_i\Delta\eta_i=-2$ at every outgoing switch. Both products
are zero at unchanged offsets. Discrete summation by parts has no vertical
boundary term because $w_0=w_m=0$, so
\[
 4J_j=-\sum_{i=1}^{m-1}w_i(\Delta\eta_i-\Delta\varepsilon_i)
      =2(\kappa_{j-1}+\kappa_j).
\]
Combining the two evaluations of $J_j$ gives
\[
 2e_j=\begin{cases}
 \kappa_{j-1}+\kappa_j+2S_j+\tau_j,&1<j<n,\\
 \kappa_{\rm adjacent}+\tau_j+1,&j\in\{1,n\}.
 \end{cases}
\]
All terms beyond the switch counts are nonnegative, proving
$2e_j\ge\kappa_{j-1}+\kappa_j$.

\medskip\noindent
The nonuniform column gaps number $b$, their incident columns form $L$,
and the total switch count is $K=\sum_j\kappa_j$. Summing only over $L$ counts every switch
twice, so $\sum_{j\in L}e_j\ge K$. Every other column is occupied.
With $n=b+c+1$ and $|L|\le2b$,
\[
 d\ge K+n-|L|\ge K+n-2b.
\]
Each nonuniform row gap contains a switch, so $a\le K$. Since
$n\le d$, also $b\le d-c-1$. Substitution gives
\[
 a\le d+b-c-1\le2(d-c-1).
\]
\end{proof}

\begin{lemma}[Switch-dependent correction bound]\label{lem:switch-degree-bound}
For integers $d\ge4$ and $2\le s\le d-2$,
\[
 \deg B_d^{(s)}\le d-2-\left\lceil\frac{s-1}{2}\right\rceil.
\]
\end{lemma}

\begin{proof}
By Lemma~\ref{lem:primitive-expansion}, a surviving term has
$a\ge s-1$. The switch budget therefore gives
\[
 s-1\le a\le2(d-c-1),\qquad
 c\le d-1-\left\lceil\frac{s-1}{2}\right\rceil.
\]
Its column-expansion factor has degree $c-1$. Summing the finitely
many primitive contributions preserves this bound, and separable
corrections vanish by Theorem~\ref{prop:meta}.
\end{proof}

A uniform band is a maximal consecutive set of columns joined by uniform
gaps. Its columns are translates of one vertical profile. The highest
column degrees can occur only in primitives with nearly as many uniform
column gaps as extrema. Their row discrepancies vanish for $s>a+1$,
where $a$ counts their nonuniform row gaps.

\begin{lemma}[Top primitive layers]\label{lem:top-layer-bound}
For integers $d\ge5$ and $3\le s\le d-2$,
$\deg B_d^{(s)}\le d-4$. For integers $d\ge8$ and $2\le s\le d-2$,
\[
 \deg B_d^{(s)}\le d-1-\min(s,4).
\]
\end{lemma}

\begin{proof}
A non-separable primitive has a nonuniform column gap, so
\[
 c\le v-2\le d-2.
\]
Its column expansion has degree $c-1$. It suffices to remove the
highest one or two values of $c$ by bounding their nonuniform row gaps.

\medskip\noindent\emph{The highest layer.}
For $c=d-2$, the width is $v=d$.
There is one nonuniform gap and every column has one extremum.
Within each uniform band the offset signs alternate, since equal
successive signs would leave the intervening column empty.
For $d\ge5$, one of the two bands has at least three columns.
Each band contains a grid boundary column. In the longer band, that
boundary column and its immediate in-band neighbour test opposite
extremum kinds in their common vertical path, each with count one.
That path is monotone. After
negation it increases, and compatibility permits its offset to the
other profile to switch only from $+1$ to $-1$, once.
Non-separability requires that switch, so $a=1$. The row discrepancy
vanishes at every $s\ge3$.

\medskip\noindent
\emph{The next layer with two bands.}
For $c=d-3$, either $v=d-1$ or $v=d$. Assume $d\ge8$.
If $v=d-1$, the two bands have one
extremum in every column except for one additional extremum. A band
has at least three columns. Its grid boundary column and immediate
in-band neighbour test both path extremum kinds with total at most three,
so the path has at most one
turn. A compatible offset switches at most once in each constant-step
segment, giving $a\le2$.

\medskip\noindent
\emph{The next layer with three bands.}
If $v=d$, all columns have one extremum and there are three bands.
Unless an outer band has length at least three or the middle band at
least four, the width is at most seven. Thus a band supplies opposite
tests forcing a monotone reference path. A middle reference differs
from each neighbouring profile at one edge, giving $a\le2$.
Otherwise, horizontal reflection and root renormalization put the reference
on the left. Negation makes it increasing if necessary. Its
offset to the middle profile switches once, at edge $p$, so that
profile has positive steps except at $p$.
For its next offset $\eta_i$, compatibility requires
$g_{i+1}-g_i+\eta_{i+1}-\eta_i\in\{-1,1\}$, where $g$ is this
middle profile. A $+1$ to $-1$ switch requires a positive step of $g$.
A $-1$ to $+1$ switch requires a negative step. The latter can occur
only at $p$, and switch directions alternate. There are therefore at
most three switches. Suppose the next offset has three switches.
They occur at $p_-<p<p_+$ with signs $+,-,+,-$.
If the middle band has one column, it has a maximum at row $p$ and a
minimum at row $p+1$. If it has more, its last column has a uniform
incoming sign. For sign $+1$, row $p$ and the bottom endpoint are
both maxima. For sign $-1$, the top endpoint and row $p+1$ are both
minima. Each contradicts its single extremum, excluding three switches.
Every remaining offset has at most two switches. If there are two,
one is at $p$, so their union with the first gap still gives $a\le2$.
The discrepancy of this entire layer therefore vanishes for $s\ge4$.
All lower layers have degree at most $d-5$.

\medskip\noindent
Together with the integer Newton law at $s=2$, this proves
$\deg B_d^{(s)}\le d-1-\min(s,4)$ for $d\ge8$.
Removing only the highest layer gives the stated $d-4$ bound for every
$d\ge5$ and $s\ge3$.
\end{proof}

The near-threshold bounds require cancellation between column classes,
not an absolute bound on each primitive contribution. For a long
dimension $n\ge d-1$, transposition gives an array with $n$ rows
and $v$ occupied columns, so $z$ marks the short dimension.
The proof of Theorem~\ref{thm:first-boundary} identifies the coefficients
of the proper fraction with $p_d(n,\cdot)$. At short dimension
$s=d-t\ge2$, the correction is therefore the coefficient of $z^{d-t}$
in the polynomial part of the transposed width series.
The number of uniform horizontal gaps is $q$, and $\delta=d-v$.
The normalized height count for this class is $\Gamma_{d,\delta,q}(n)$.
On the count block, $\Pi$ and $Q$ select uniform and
nonuniform pairs, so every such height has one projector sequence.
Its ordered term is $z^{d-\delta}/(1-z)^q$. At $q=0$, only
$\delta=t$ contributes to the polynomial-part coefficient at $z^{d-t}$.
For $q\ge1$, Laurent expansion at infinity gives that coefficient as
\[
 (-1)^q\binom{t-1-\delta}{q-1},\qquad\delta+q\le t.
\]
This is distinct from the onset extraction at zero.
The cost is $\zeta=\delta+q$.

For at least three columns, the indicators of uniform end gaps are $f_L,f_R$,
and the number of internal uniform gaps is $q_{\rm int}$. The endpoint excesses
are $\lambda_L=e_1-(2-f_L)$ and $\lambda_R=e_v-(2-f_R)$.
Lemma~\ref{lem:boundary} makes them nonnegative. Every middle column
has at least one extremum, and summing its excess together with the
endpoint excesses gives
\begin{equation}\label{eq:cost-budget}
 \zeta-2=\lambda_L+\lambda_R+
 \sum_{\rm middle}(e_j-1)+q_{\rm int}.
\end{equation}
For two columns with a nonuniform pair, the same identity holds with
$f_L=f_R=q_{\rm int}=0$. A uniform two-column array has cost $d-1$,
so it cannot enter costs at most three for $d\ge5$, or costs at most
four for $d\ge6$.

At cost two all excesses in~\eqref{eq:cost-budget} vanish. Every
middle column has one extremum, and the only uniform pairs are the
end pairs forced by one-extremum endpoints. These are precisely the
four endpoint classes of Section~\ref{subsec:count-baxter}. Their number of
one-switch passages is the Baxter index
\[
 k=v-1-q=d-\delta-q-1=d-3.
\]
The correspondence with Section~\ref{subsec:count-baxter} uses its row
dimension $m=n$ and edge count $N=n-1$. In particular, $n\ge d-1$
gives $1\le k<N$. Division by three removes the raw colour rotation
factor, giving the complete identification
\[
\begin{array}{ccl}
 (\delta,q)&\text{endpoint types}&\text{normalized class count}\\
 (0,2)&11&\Gamma_{d,0,2}=W_k^{11}/3,\\
 (1,1)&12,21&\Gamma_{d,1,1}=(W_k^{12}+W_k^{21})/3,\\
 (2,0)&22&\Gamma_{d,2,0}=W_k^{22}/3.
\end{array}
\]
For $t=3$ the three extraction weights are $2,-1,0$, and for $t=4$
they are $3,-1,0$. With $X=N-k=n-d+2$,
Lemma~\ref{lem:boundary-multiplicities} therefore gives the two cost-two
contributions explicitly as
\begin{align*}
 \frac{2W_k^{11}-W_k^{12}-W_k^{21}}{3}
 &=\bigl(8(X+3)-16X-8\bigr)\mathrm{Bax}(k)\\
 &=8(d-n)\mathrm{Bax}(d-3),\\
 \frac{3W_k^{11}-W_k^{12}-W_k^{21}}{3}
 &=\bigl(12(X+3)-16X-8\bigr)\mathrm{Bax}(k)\\
 &=4(d+5-n)\mathrm{Bax}(d-3).
\end{align*}
All remaining terms have cost three or four. The signed count of the
complete cost-three class is $C_d$, with sign $(-1)^q$.
For $j\in\{3,4\}$, the variable $w$ marks uniform gaps in the cost
polynomial $H_{d,j}(w;n)=\sum_{\delta+q=j}\Gamma_{d,\delta,q}(n)w^q$.
The marking operator is $\mathcal D=w\partial_w$. The weighted
cost-three and cost-four remainder $R_d$ is
\[
 R_d=H_{d,4}(-1;n)+[\mathcal D H_{d,3}]_{w=-1}.
\]
The full extraction identities are
\begin{align}\label{eq:weighted-strip-decomposition}
 B_d^{(d-3)}(n)&=8(d-n)\mathrm{Bax}(d-3)+C_d(n),\\
 C_d&=-\Gamma_{d,0,3}+\Gamma_{d,1,2}
       -\Gamma_{d,2,1}+\Gamma_{d,3,0},\notag\\
 B_d^{(d-4)}(n)&=4(d+5-n)\mathrm{Bax}(d-3)+R_d(n).
 \label{eq:cubic-strip-decomposition}
\end{align}
These identities make $C_d$ and $R_d$ polynomials on $n\ge d-1$,
since the correction polynomials and the displayed Baxter terms are
polynomials there. No separate polynomiality of each $\Gamma$ class is needed.

Between uniform bands, a one-switch passage changes the vertical step
at its switch edge. A phase is a maximal sequence of these passages
with the same first sign $a$. Its available-step set consists of
the edges whose current step is $a$. Each passage consumes one by
changing its step to $-a$. A restoration makes the reverse change.
Phase length counts passages, not rows.

A bridge is a uniform gap inserted between one-switch passages. Its
sign is the constant horizontal increment, and its junction word is
the vertical step word where the phases meet. An \emph{event} flips
one vertical step, consuming or restoring an $a$-step. Its record
gives the edge, the step's old sign and the
column transition where the flip occurs. Events at one transition are
ordered by edge index. The first horizontal increment at each transition
is recorded separately. Each column in a $\Gamma$ array contributes an
extremum, so an array with $d$ extrema has at most $d$ columns.

\begin{lemma}[Interval-word encodings]\label{lem:interval-word-count}
For fixed integers $d\ge1$, $a_0,b_0,k,K\ge0$ with $a_0+b_0\ge1$ and every
integer $N\ge1$, the number
of pairs $B\subseteq A\subseteq\{1,\ldots,N\}$ with
$|A\setminus B|=k$, $a_0$ interval components in $A$ and $b_0$
in the complement of $B$ is $O_{a_0,b_0,k}(N^{a_0+b_0-1})$.
The same exponent bounds nonnested pairs with $|A\mathbin\triangle B|\le k$.
If $a_0,b_0,k,K$ are bounded in $d$, normalized compatible height arrays
with $N+1$ rows, at most $d$ columns, initial and final $a$-step sets
$A,B$ of these kinds for a sign $a\in\{-1,1\}$, and at most $K$ events supported on
$A\mathbin\triangle B$ have count $O_d(N^{a_0+b_0-1})$.
Allowing one additional marked edge as an event position gives count
$O_d(N^{a_0+b_0})$.
\end{lemma}

\begin{proof}
On the $N$ ordered edges, suppose $B\subseteq A$,
$|A\setminus B|=k$ is fixed, $A$ has $a_0$ interval components
and the complement of $B$ has $b_0$.
The edge alphabet has $X$ outside $A$, $Y$ in $B$, and $Z$ in the
difference. For a fixed finite run pattern and fixed positive $Z$
lengths, $L$ counts the $X,Y$ runs and $e$ counts the ends occupied by $Z$.
Counting components by their run incidences gives $a_0+b_0=L+e$.
There are $L$ long-run lengths with prescribed total $N-k$.
Their positive compositions have $O(N^{L-1})$ choices, and
$L\le a_0+b_0$. The finite run patterns and bounded $Z$ lengths
therefore give $O_{a_0,b_0,k}(N^{a_0+b_0-1})$ pairs. If there is
no long run, $N=k$ is bounded.

\medskip\noindent
For nonnested endpoint sets with bounded symmetric difference, the alphabet
also has a letter $W$ in $B\setminus A$.
The boundary incidences of $Z$ runs number $e$, and their
adjacencies with $W$ runs number $f$. The same incidence calculation gives
\[
 a_0+b_0=L+e+f\ge L.
\]
The $Z,W$ lengths are bounded, while the long $X,Y$ lengths have
fixed total. Their compositions give the same exponent for endpoint pairs.

\medskip\noindent
The array encoding also records its initial sign $a$ and number of columns.
Each event has at most $k$ edge positions, two old signs and at most
$d-1$ transition positions. At most $K$ such records and at most $d-1$
first horizontal increments therefore add only a $d$-dependent factor.
The initial set $A$ and sign $a$ fix the first vertical step word.
Applying the recorded flips reconstructs each subsequent vertical word.
Root normalization and the first horizontal increments fix the first row.
Summing each vertical word recovers the rest of its column. Thus valid
array encodings are injective, and compatibility only restricts the codes.
Choosing one extra marked edge multiplies the bound by at most $N$.
In particular, consumptions of one fixed sign with at most one restoration
can revisit at most one edge outside $A\mathbin\triangle B$. That edge
is the optional mark. Pure consumptions need no mark.
\end{proof}

\begin{lemma}[Quadratic correction bound]\label{lem:quadratic-correction-bound}
For every integer $d\ge5$, the polynomial $B_d^{(d-3)}$ has degree
at most two.
\end{lemma}

\begin{proof}
The cost-two contribution in~\eqref{eq:weighted-strip-decomposition}
is linear. It remains to bound the complete signed cost-three count
$C_d(n)$ by $O_d(n^2)$. A cost-three array has one surplus unit
beyond the minimal endpoint cost. Pairing exceptional middle columns
with uniform bridges cancels their opposite signs. Interval encodings
then bound the unpaired arrays, and polynomiality bounds the degree.

\medskip\noindent\emph{Excluding multiple switches.}
The endpoint excesses and gap counts are those of~\eqref{eq:cost-budget}.
An offset with $t$
switches forces at least $t+1$ extrema at an endpoint. For a middle
column, $t_{\rm in}$ and $t_{\rm out}$ count switches in its incoming
and outgoing offsets. The switch identity in the proof of
Lemma~\ref{lem:local-switch-budget}, with nonnegative saddle and
perimeter counts, gives
\[
 t_{\rm in}+t_{\rm out}=2e_j-2S_j-\tau_j\le2e_j.
\]

\medskip\noindent
For a gap with at least two switches, an incident boundary column has
at least three extrema and therefore endpoint excess $e_j-2\ge1$.
At a middle column whose other gap is nonuniform, the switch identity
gives $2e_j\ge3$, so $e_j-1\ge1$. If the other gap is internal
and uniform, its unit is counted by $q_{\rm int}$ instead.

\medskip\noindent
If that other gap is a uniform end gap, either its boundary excess
or the adjacent middle-column excess is positive. To see this at the
left end, a two-switch outgoing word has first sign $b$ and edges
$L<R$, and its uniform incoming sign is $\gamma$.
For $\gamma=-b$, the middle column has two nonzero changed runs, hence
at least two extrema by Lemma~\ref{lem:boundary-interval}.
For $\gamma=b$, compatibility forces a $b$ step at $L$ and a later
$-b$ step at $R$. A one-extremum boundary word would have form
$(-b)^*b^*$, which forbids that order. The boundary therefore has
at least two extrema and endpoint excess $e_j-1\ge1$.
Horizontal reflection gives the right-end case. With more than two
switches, a middle column with one extremum is already excluded by
Lemma~\ref{lem:boundary-interval}.

\medskip\noindent
The two sides of this fixed gap charge different columns or different
uniform gaps. Its neighbouring uniform gaps number $F$, of which
$F_{\rm end}$ are at the ends. The column excesses satisfy
\[
 \lambda_L+\lambda_R+\sum_{\rm middle}(e_j-1)
 \ge2-F+F_{\rm end},\qquad F-F_{\rm end}\le q_{\rm int}.
\]
The second inequality holds because those neighbouring internal gaps
are distinct members of the $q_{\rm int}$ gaps.
Equation~\eqref{eq:cost-budget} now gives
\[
 \zeta\ge4+q_{\rm int}-F+F_{\rm end}\ge4.
\]
Thus cost three permits only zero- and one-switch gaps.
Equation~\eqref{eq:cost-budget} then has right side one.
Its single unit lies in an endpoint excess, a two-extremum middle
column, or an internal uniform gap. These are the three classes below.

\medskip\noindent\emph{Endpoint excess.}
In the extra-endpoint case, horizontal reflection and root normalization
place that endpoint on the left. Every middle column has one
extremum and there is no internal uniform gap. Two one-switch offsets
of opposite first signs have two nonzero changed runs, contradicting
Lemma~\ref{lem:boundary-interval}. Thus the consecutive nonuniform gaps
share one first sign $a$. Crossing one changes an $a$-step to $-a$,
so that edge cannot be consumed again.

\medskip\noindent
The initial and final $a$-step sets are $A,B$. At a nonuniform
extra end, $e_1=3$. Compatibility gives an $a$-step at its switch.
Above the switch, an extremum is an entry into an $a$-step run.
Below it, an extremum is an exit from such a run, including extrema
at the grid boundary. The component through the switch is counted in both blocks.
Thus $e_1$ is one more than the number of components of $A$, which
therefore has two. At a nonuniform right end, the same count applies
to the $-a$-steps. Since $e_v=2$, the complement of $B$ is one interval.
For a uniform right end, its one-extremum vertical word makes that
complement a prefix or suffix. Every passage consumes a distinct edge,
so $B\subseteq A$ and $|A\setminus B|\le d-1$. Lemma~\ref{lem:interval-word-count}
now gives $O_d(N^2)$ arrays.

\medskip\noindent
If the extra end gap is uniform, deleting the end column leaves
the old middle extremum and a subset of the two
old end extrema, giving two or three in total. The former gives a cost-two
object and the latter the case just counted. There are at most two
inverse signs. Cost-two counts are $O_d(n)$ by the endpoint formulas.

\medskip\noindent\emph{Middle-column excess and cancellation.}
A two-extremum middle column cannot meet a uniform end gap, since
the one-turn end word and its compatible one-switch neighbour give
only one extremum there. Its adjacent gaps are nonuniform.
With equal first signs, the changed interval contains one saddle.
All changes consume the same sign. Minimal endpoints give one interval
for $A$ and one for the complement of $B$. The nested count is
$O_d(N)$ by Lemma~\ref{lem:interval-word-count}, with finitely many consumption orders.

\medskip\noindent
An array with opposite first signs $a,-a$ beside its two-extremum
middle column $g$ is a source. Its two switch edges are $p,r$, and the
canonical shift is $\alpha=a$ if $p\ge r$, and $\alpha=-a$ otherwise.
The insertion replaces $g$ by $g,g+\alpha$ and translates all later columns by
$\alpha$. Before insertion, the two maximal nonzero runs of
$\eta-\varepsilon$ are $[1,\min(p,r)]$ and
$[\max(p,r)+1,n]$. Each forces an extremum, and $g$ has exactly two,
so each run contributes exactly one.
After insertion, the changed runs of $g$ and $g+\alpha$ are
\[
\begin{array}{ccc}
 &g&g+\alpha\\ \midrule
 p\ge r&[p+1,n]&[1,r]\\
 p<r&[1,p]&[r+1,n].
\end{array}
\]
These are precisely the old lower and upper runs in the first case,
and the old upper and lower runs in the second. Both new columns
have the same vertical word as $g$, so their interval tests retain
one old extremum each. All other columns keep their neighbour
increments under the translation. Degree is therefore preserved,
and the added uniform gap reverses $(-1)^q$.
Collapsing that gap is the inverse when its offset is this
canonical $\alpha$. The unique extra unit fixes the location, including
the canonical choice at $p=r$. Thus the complete source and canonical
bridge subfamilies cancel.

\medskip\noindent\emph{Unpaired internal bridges.}
A bridge of the other sign is not paired by this canonical insertion.
For $p\ge r$, its two columns have changed intervals $[1,p]$ and
$[r+1,n]$. For $p<r$, they have $[p+1,n]$ and $[1,r]$.
The prefix column's one-extremum test requires $a$ steps followed
by $-a$ steps, and the suffix column's test requires $-a$ steps
followed by $a$ steps. Their extended words overlap between the
actual steps at $p,r$, both $-a$ by compatibility. The overlap
must therefore consist of $-a$ steps. When $p=r$ it is that single
step. The shared vertical word has form $a^*(-a)^+a^*$.
Negation permits $a=1$ for this absolute count, with a factor two.
The negative interval $J$ has complement $O$.
The initial and final $a$-step sets $A,C$ are intervals containing $O$,
with fixed positive differences set by the two phase lengths.
Each nested pair has $O_d(N)$ choices, so their product gives
$O_d(N^2)$ possibilities even before matching the shared $O$.
Consumption orders and signs are bounded in $d$.

\medskip\noindent
A same-sign internal bridge collapses to a one-extremum column.
The switch edges differ, and restricting the monotone bridge interval
to the direct changed interval gives exactly one extremum.
Degree drops by one to cost two, with finitely many inverse positions.
If an internal uniform gap meets a uniform end gap, their signs are
opposite, or the intervening column would be empty. A nonconstant
one-turn end word would give two extrema in the next column, so the
word is constant. Deleting the end again gives cost two with at most
two inverses. These cases exhaust cost three. Hence
$C_d(n)=O_d(n^2)$, and polynomiality proves the quadratic correction bound.
\end{proof}

The cubic argument also uses source deletion. For a cost-three source
with incoming first sign $a$ and
distinct switch edges $p,r$, deletion removes its two-extremum middle
column $g$ and translates all later columns by $-\alpha$. The canonical
choice is $\alpha=a$ for $p>r$ and $\alpha=-a$ for $p<r$.
A source inverse of a cost-four array is a
cost-three source with switch data $(a,p,r)$ that gives the array
under this deletion. Each valid choice is counted once, without
identifying different choices. The cubic proof verifies degree
preservation and lists the at most two inverse choices.

\begin{lemma}[Cubic correction bound with finite-state certificate]\label{lem:cubic-correction-bound}
For every integer $d\ge6$, the polynomial $B_d^{(d-4)}$ has degree
at most three.
\end{lemma}

\begin{proof}
The integer Newton law gives $\deg B_6^{(2)}\le3$.
Lemma~\ref{lem:top-layer-bound} gives the same bound at
$(d,s)=(7,3)$ and $(8,4)$, settling the three deductive cases.
For $d\ge9$, the extraction identity~\eqref{eq:cubic-strip-decomposition}
is
\[
 B_d^{(d-4)}(n)=4(d+5-n)\mathrm{Bax}(d-3)+R_d(n).
\]
It remains to prove $R_d(n)=O_d(n^3)$ for $n\ge d-1$.
The cost-three pairing cancels part of this signed remainder against
cost-four arrays. The surviving classes admit interval-word bounds,
with one local implication supplied by the retained four-column
finite-state certificate.

\medskip\noindent\emph{Marked cost-four reduction.}
The generating function $P_d(w;n)$ counts the opposite-sign sources before canonical bridge
insertion; $U_d(w;n)$ counts the unpaired cost-three objects. The exact
pairing and the marking operator $\mathcal D=w\partial_w$ give
\begin{align*}
 H_{d,3}(w;n)&=(1+w)P_d(w;n)+U_d(w;n),\\
 [\mathcal D H_{d,3}]_{w=-1}
 &=-P_d(-1;n)+[\mathcal D U_d]_{w=-1}.
\end{align*}
In the differentiated product, $\mathcal D(1+w)=w$ and the factor
$1+w$ vanishes at $w=-1$. The marked unpaired term is $O_d(n^2)$,
since the unpaired absolute count is quadratic and an array has at most
$d$ gaps to mark. Equal-edge sources have the same junction-interval
encoding as the unpaired class and are also quadratic.

\medskip\noindent
For distinct source edges $p,r$ and incoming first sign $a$, source deletion removes $g$
and translates the later columns by $-\alpha$, where $\alpha=a$ for
$p>r$ and $\alpha=-a$ for $p<r$. The new offset
$\varepsilon+\eta-\alpha$ is a two-switch word. Row reversal exchanges
the two edge orders and changes $a$ to $-a$ while preserving $\alpha$.
It therefore suffices to check degree preservation for $p>r$.
In the predecessor column the changed prefix ends at an actual
$-a$ step at $r$ and has form $a^*(-a)^*$. It gains exactly one
extremum. For an incoming one-switch gap its switch has current step
$-a$, at or after the prefix extremum. For an end column its minimal
available-step interval contains $p$ but not $r$, so the prefix is all $-a$
and the added extremum is the top row. A uniform end predecessor has
the same conclusion from its one-turn word. Reversing both rows and
columns gives one new extremum in the successor. The removed column had two, so degree
and occupation are preserved. Deletion reduces the column count by one,
so the cost rises to four while $q$ is unchanged.
There are at most two source inverses. For a two-switch
offset with first sign $b$ and edges $L<R$, they are
$(a,p,r)=(b,L,R)$ or $(-b,R,L)$.

\medskip\noindent
The eligible pairs are consecutive one-switch nonuniform gaps with
opposite first signs, separated by no uniform gap or by the single
canonical bridge. Toggling the first eligible pair inserts that bridge
when absent and collapses it when present.
Insertion separates the two nonzero outer intervals and collapse
reunites them, preserving every extremum and reversing $(-1)^q$.
The first eligible pair remains the same after the toggle, so this
is an involution. The number of actual source inverses, including
their multiplicity, is $\nu$. A source inverse requires a two-switch gap whose
two sides spend both slack units in~\eqref{eq:cost-budget},
$\zeta-2=\lambda_L+\lambda_R+\sum_{\rm middle}(e_j-1)+q_{\rm int}$.
An eligible pair without a bridge has an extra middle-column extremum.
With a bridge it has an internal uniform gap. Either adds a unit
distinct from the two source charges, so an array containing both
would have cost at least five. Every toggled cost-four array therefore
has $\nu=0$, and the involution cancels its full signed weight.
For each remaining cost-four array, source deletion preserves $q$ and
the distinct-edge part of $-P_d(-1;n)$ subtracts its $\nu$ source contributions.
Its resulting weight is $(-1)^q(1-\nu)$.

\medskip\noindent\emph{One-switch classes.}
In the common-sign one-switch class, the initial available-step set has at
most $\lambda_L+1$ components and the complement of the final set at most
$\lambda_R+1$. The cost-four budget gives
$\lambda_L+\lambda_R+2\le4$, so
Lemma~\ref{lem:interval-word-count} bounds these nested sets and their
consumption records by $O_d(n^3)$.
After the eligible pairs have cancelled, a change of phase sign occurs
at a noncanonical uniform bridge or across two consecutive uniform gaps.
Each spends a budget unit, so there are at most two such changes.

\medskip\noindent
With one sign change, the vertical edge set is $\{1,\ldots,N\}$,
where $N=n-1$. Negation and row reversal permit an incoming
first sign $+1$, a bridge offset $-1$ and switch edges $p\ge r$
in the incoming and outgoing passages. The junction columns have the
same vertical word, with negative steps at $p,r$ by compatibility.
Its negative-step set is $J$, and the overlap of
$[1,p]$ and $[r,N]$ is $I=[r,p]$. The adjacent extremum counts are
$e_{\rm left},e_{\rm right}$. The first junction column has one extremum
for each component of $J$ meeting $[1,p]$. The second has one for
each meeting $[r,N]$. These are the $+1$ to $-1$ and $-1$ to
$+1$ tests of Lemma~\ref{lem:boundary-interval}, including its
virtual boundary steps. If $b_J$ counts the components of $J$ and
$b_I$ those meeting $I$, their adjacent extremum counts satisfy
\[
 e_{\rm left}+e_{\rm right}=b_J+b_I\ge b_J+1.
\]
Indeed, every component meets at least one of the two edge intervals,
and precisely the components meeting $I$ are counted twice. The
endpoints $r,p$ lie in $J$, so $b_I\ge1$.
The bridge spends one of the two budget units,
so $e_{\rm left}+e_{\rm right}\le3$ and $b_J\le2$.

\medskip\noindent
The positive-step sets at the array's left and right
ends are $A,C$. They contain $J^c$, with $|A\setminus J^c|$ and
$|C\setminus J^c|$ bounded by the phase lengths, hence by $d$.
Their component counts are at most $\lambda_L+1$ and
$\lambda_R+1$. If $J$ has one component, applying
Lemma~\ref{lem:interval-word-count} to $J^c\subseteq A$ and
$J^c\subseteq C$ gives combined exponent at most
$\lambda_L+\lambda_R+2\le3$, even before matching the shared $J^c$.
If it has two, the displayed identity makes the adjacent count three.
Together with the bridge this exhausts the two-unit budget, so both
array endpoints are minimal and $A,C$ are intervals.
If $J$ misses both outer edges, $J^c$ contains $1,N$, forcing
both $A$ and $C$ to equal $\{1,\ldots,N\}$.
Both component lengths are bounded, so
their two left endpoints give $O_d(N^2)$ choices.
If $J$ meets one outer edge, each interval containing $J^c$ fills
the interior component of $J$, whose length is therefore bounded.
The boundary-component endpoint, the interior-component
left endpoint, bounded lengths and endpoint extensions determine
$J,A,C$. If $J$ meets both outer edges, $J^c$ is one interval.
Its two endpoints and the bounded extensions determine the same data.
Each case has $O_d(N^2)$ choices. Within each phase, the bounded
event records of Lemma~\ref{lem:interval-word-count} add only a
$d$-dependent factor.
Two consecutive uniform gaps have opposite signs and a one-turn
intermediate word, reducing to the one-component case.
Two separate sign changes leave three nested pairs, each with
$O_d(n)$ choices. The whole one-switch cost-four class is cubic.

\medskip\noindent\emph{Multi-switch classes.}
For $h$ multi-switch gaps, disjoint charges before the first,
between consecutive ones and after the last give
$\zeta-2\ge h+1$. Each region has the column-excess or uniform-gap
charge used in the cost-three budget argument, and distinct regions
have distinct charges. Cost four therefore has at most one multi-switch gap.
Four or more switches would cost at least two units on each side,
contradicting the budget two. A three-switch gap cannot touch a
boundary or a uniform gap, since either uses two units on that side.
Its neighbours have one switch and the same first sign.
The opposite sign would give three changed intervals, and the sole possible
coalescence forces incompatible vertical steps. Its three switches
give two consumptions and one restoration. Minimal endpoints and
the bounded-event encoding of Lemma~\ref{lem:interval-word-count}
give at most cubic growth.

\medskip\noindent
For a two-switch gap, equal side-phase signs give one restoration
and the same cubic bound. If a phase is missing, horizontal reflection
and root normalization place it on the right. The initial set then
has at most one component and the final complementary set at most two.
The sharper bounded-event encoding again gives at most cubic growth.
Two valid source inverses require that this be the only nonuniform
gap, since another gap fixes the phase sign. Then the band count is
two and the cost identity gives $d=6$, already covered by the integer
Newton law. Thus $\nu\le1$ in the range $d\ge9$ under consideration.
Only a two-switch gap with two nonempty, opposite-sign phases remains.
It has at least three nonuniform gaps and therefore requires $d\ge8$.

\medskip\noindent\emph{Side resolutions and signed weights.}
In that remaining case, all internal uniform gaps are immediately
adjacent to the exceptional two-switch gap, since its two sides
already spend the two slack units. Each has the first sign $b$ of
that gap. Sign $-b$
would force two nonzero outer intervals as well as paying for the
gap. On either side of an adjacent uniform gap, the two completed columns
have one extremum each. Their collapsed column has at least two
from its adjacent one- and two-switch offsets and at most two since
splitting through a constant cannot decrease their sum. Thus degree
and occupation are preserved.
The two inverse resolutions affect distinct vertical words, commute
and recover a unique base with no internal uniform gap.
If $r_0\in\{0,1,2\}$ resolutions are admissible and $q_0$ counts
end gaps, choosing subsets of the resolutions gives signed weight
$(-1)^{q_0}(1-1)^{r_0}$. Refinements with an internal gap have no
source inverse, since a cost-three paired source has none.
Subtracting the base's true inverse-source multiplicity therefore gives
\[
 (-1)^{q_0}\bigl((1-1)^{r_0}-\nu\bigr).
\]
The remaining cancellation needs $\nu=1$ whenever $r_0=0$.
This is the only step supplied by the finite-state certificate.

\medskip\noindent\emph{The finite-state local implication.}
The local implication concerns four columns with consecutive offset
switch counts $1,2,1$, first signs $a,b,-a$ and two extrema in each
middle column. If neither constant-$b$ side resolution is admissible,
the phase-compatible inverse factorization must have middle-column
counts $1,2,1$. Those counts make it a source inverse, and
$\nu\le1$ then gives $r_0=0\Rightarrow\nu=1$.

\medskip\noindent
Negation permits $a=1$. For one row transition, the offsets are
$\varepsilon,\theta,\eta$ and their changes are
$\Delta\varepsilon,\Delta\theta,\Delta\eta$.
If the first middle column has vertical step $\rho\in\{-1,1\}$,
unit steps in the other original columns require
\[
 \rho-\Delta\varepsilon,\quad \rho+\Delta\theta,\quad
 \rho+\Delta\theta+\Delta\eta\in\{-1,1\}.
\]
The inserted column has horizontal offset $\chi$ from the first
middle column, so its vertical step is $\rho+\Delta\chi$ and must
also be unit.
The inverse factorizations specify that $\chi$ switches at the first
$\theta$ switch when $b=a$, and at the second when $b=-a$.
Since both inverse choices have $\alpha=-b$, the second one-switch
offset is $\chi'=\theta-b-\chi$ by the source-deletion identity.
The factorized middle-column counts correspond to the offset pairs
\[
 (\varepsilon,\chi),\qquad(\chi,\chi'),\qquad(\chi',\eta).
\]

\medskip\noindent
The side-resolution tests use the same local extrema.
For the split column beside the two-switch gap, the maximal nonzero
offset interval is that gap's middle block. For the other split column,
it is the grid-end block of the adjacent one-switch gap.
Each interval forces an extremum by Lemma~\ref{lem:boundary-interval}.
Splitting through $b$ adds exactly
the vertical extrema where both relevant offsets equal $-b$.
Since the base column has two extrema and each split column has at
least one, both split counts equal one exactly when no extremum is
added. A Boolean flag for each side records whether an added extremum
has made that resolution fail.

\medskip\noindent
A state records the three switch counters, the previous vertical
steps of the two middle columns and inserted column, the two original
and three factorized extremum counts, and the two failure flags.
The switch counters are
bounded by $1,2,1$. Original counts above two are rejected, since
later rows cannot decrease them. The factorized counts are saturated
at $2,3,2$, preserving exactly whether they equal $1,2,1$.
The initial state has zero counters and counts, no previous vertical
step and unset flags. A terminal state has switch counters $1,2,1$.
The initial and terminal row tests omit the absent vertical neighbour.

\medskip\noindent
Every compatible four-column array gives a state path, and the offsets
and vertical steps of a path reconstruct its columns up to normalization.
For each sign $b/a$, the retained certificate has 301 states and 846
transitions. It contains the initial state, is closed under every
legal successor, and has no terminal state with original counts $2,2$,
both failure flags and factorized counts different from $1,2,1$.
Induction on appended rows keeps every legal prefix in the retained
set. A complete array with the prescribed switch counts reaches a
terminal state when its final row is counted. The forbidden terminal
condition therefore proves the implication for every number of rows.
The computation checks finite closure with four columns fixed, not
arrays below a row cutoff.

\medskip\noindent\emph{Surviving opposite-phase bases.}
The implication gives $r_0=0\Rightarrow\nu=1$. Those bases cancel,
as do bases with $r_0>0,\nu=0$. Only $r_0>0,\nu=1$ survive,
each with weight $-(-1)^{q_0}$.
The unique source of a surviving base has available-step set $V$ in the inserted
column. Outside its two exceptional edges its word has forms
$a^*(-a)^*$ above and $(-a)^*a^*$ below.
An admissible side resolution restricts the intervening word to one
extremum, so $V$ has at most one prefix, one interior interval and
one suffix. The initial and final $a$-step sets $A,C$ are intervals
containing $V$, with bounded differences.
If $V$ is empty, $A,C$ have bounded lengths, so their two left
endpoints suffice. If $V$ is nonempty and touches neither end, it is one
interval. Its two endpoints and the bounded extensions determine
$A,C$. If it touches just the left end, it is a prefix and at most
one interior interval. The gap between them is filled by $A,C$ and
has bounded length. The two component right endpoints and that
bounded length determine $V$. The extensions to $A,C$ are bounded.
Reflection treats a set touching just the right end.
If $V$ touches both ends, both $A$ and $C$ equal $\{1,\ldots,N\}$.
The at most two complementary intervals of $V$ have bounded total
length. Their
left endpoints and bounded lengths determine those intervals.
There are $O_d(n^2)$ such triples, and all event orders have only
a $d$-dependent number of choices. A source determines its base
uniquely by the degree-preserving deletion.

\medskip\noindent
The surviving opposite-phase contribution is therefore quadratic.
Every other uncancelled cost-four class is at most cubic, and the
marked unpaired and equal-edge corrections are quadratic.
Thus $R_d(n)=O_d(n^3)$ for $n\ge d-1$. The extraction identity makes
$R_d$ a polynomial there, so its growth bounds its degree by three.
Adding the linear Baxter term proves the cubic correction bound.
\end{proof}

\begin{theorem}[Boundary correction degree bounds, with a computer-assisted cubic case]\label{thm:boundary-degree-bounds}
For integers $d\ge4$, $2\le s\le d-2$ and $n\ge d-1$,
the correction polynomial satisfies $\deg_n B_d^{(s)}\le D_{d,s}$.
In addition,
\[
 \deg_n B_d^{(d-3)}\le2\quad(d\ge5),\qquad
 \deg_n B_d^{(d-4)}\le3\quad(d\ge6).
\]
\end{theorem}

\begin{proof}
Lemma~\ref{lem:switch-degree-bound} bounds the correction degree by
$d-2-\lceil(s-1)/2\rceil$.
Lemma~\ref{lem:top-layer-bound} gives the other term,
$d-1-\min(s,4)$, for $d\ge8$.
For $d<8$, that term follows from these deductive inputs
\[
\begin{array}{ccl}
 d&s&\text{bound used}\\
 4&2&\text{integer Newton law},\\
 5&2&\text{integer Newton law},\\
 5&3&\text{first excluded dimension, degree one},\\
 6&2&\text{integer Newton law},\\
 6&3&\text{top primitive layer, degree at most two},\\
 6&4&\text{first excluded dimension, degree one},\\
 7&2&\text{integer Newton law},\\
 7&3&\text{top primitive layer, degree at most three},\\
 7&4&\text{quadratic correction bound},\\
 7&5&\text{first excluded dimension, degree one}.
\end{array}
\]
The integer Newton law is Theorem~\ref{prop:meta}. The first-excluded
formula is Theorem~\ref{thm:first-boundary}. Thus
$\deg B_d^{(s)}\le d-1-\min(s,4)$ throughout the stated domain.
Taking its minimum with Lemma~\ref{lem:switch-degree-bound} gives
$D_{d,s}$. Lemmas~\ref{lem:quadratic-correction-bound}
and~\ref{lem:cubic-correction-bound} give the two additional bounds.
\end{proof}

Together with Theorem~\ref{thm:first-boundary}, these upper bounds
give the cap $d-s-1$ for $s\le4$ or $s\ge d-4$, covering every
fixed short dimension through degree nine. No equality is asserted
in these degree bounds. The question is whether the same cap holds
at every intermediate short dimension.

\begin{conjecture}[Sharp correction degree]\label{conj:sharp-boundary}
For every integer $d\ge4$ and $2\le s\le d-2$, the polynomial
$B_d^{(s)}$ has degree at most $d-s-1$.
\end{conjecture}

These bounds do not prove the cap in the central range
$d\ge10$ and $5\le s\le d-5$.
Section~\ref{sec:discussion} states the coefficient question for the
complete signed sum in that range.

\section{Distance and farthest pairs}\label{sec:diameter}
\subsection{Distance and the known diameter}\label{subsec:distance-inputs}

The extrema of a height determine its legal flips, whereas the difference
of two heights determines the distance between their states. For an
integer-valued function $f$ on a finite set $V$, its dispersion is the
minimum total absolute deviation from an integer constant. The function
$(i,j)\mapsto i+j$ has dispersion $D(m,n)$, with the two quantities defined by
\begin{align*}
  \disp(f)&=\min_{K\in\Z}\sum_{v\in V}|f(v)-K|,\\
  D(m,n)&=\min_{K\in\Z}\sum_{i=1}^m\sum_{j=1}^n|i+j-K|.
\end{align*}

The underlying recolouring-distance formula is already applied to
Miura-ori in~\cite[Lemma~7 and Theorem~17]{Christensen2025origami}.
Combining the colouring correspondence with a shortest-recolouring algorithm
finds a minimum-length face-flip sequence in polynomial time~\cite[Theorem 4.3]{Akitaya2020faceflips}.
Passing to normalized heights requires minimizing over colour rotations,
giving the following quotient form.

\begin{proposition}[Quotient distance]\label{prop:quotient-distance}
For integers $m,n\ge2$ and normalized height functions $h,h'$
representing two vertices of $\OFG(\Mm)$,
\[
  \ofgdist(h,h')=\frac12\min_{K\in\Z}
  \sum_{v\in V(\Gmn)}|h'(v)-h(v)-K|.
\]
\end{proposition}

\begin{proof}
The correspondence~\eqref{eq:iso} identifies flip distance with distance in the
colour-rotation quotient.
The corresponding pinned colourings are $\sigma=h\bmod3$ and $\tau=h'\bmod3$.
Projecting a recolouring path to rotation orbits cannot increase its length.
Global colour rotation preserves recolouring edges, since it preserves
properness and single-cell differences. Conversely, an orbit path from
$[\sigma]$ to $[\tau]$ lifts from $\sigma$ by rotating each successive edge
representative so that its first endpoint is the current colouring.
The lift preserves length and ends at a rotation of $\tau$, giving
\[
  \ofgdist(h,h')=\min_{a\in\{0,1,2\}}
  d_{\mathcal R_3(\Gmn)}(\sigma,\rho^a\tau).
\]

\medskip\noindent
The zero root value and unit edge steps give $h$ and $h'$ the parity
of grid distance from $(1,1)$. A lift here is an integer-valued function with
unit edge steps, the given colouring modulo three and this grid parity.
Its root value need not be zero. Recolouring distance is half the minimum
$\ell_1$ difference between a lift of $\sigma$ and a lift of $\rho^a\tau$
\cite[Lemma 4.5]{Goldberg2004sampling}. Since $\sigma$ fixes the edge increments, its lifts
on the connected grid differ from $h$ by constants. Matching colours and
parity make such a constant $6q$ for some $q\in\Z$.
Subtracting $6q$ from both lifts fixes the first at $h$, preserves the second
lift's colours, parity and edge steps, and leaves their difference unchanged.
Rotation preserves edge increments, so each lift of a rotated $\tau$ equals
$h'+C$ for a constant $C$. Parity forces $C$ to be even. Conversely, each
even $C$ preserves parity and gives the rotation $\rho^{C\bmod3}\tau$.
Writing $K=-C$ gives
\[
  \ofgdist(h,h')=\frac12\min_{K\in2\Z}
  \sum_{v\in V(\Gmn)}|h'(v)-h(v)-K|.
\]

\medskip\noindent
All differences $h'(v)-h(v)$ are even. A central entry of the sorted
differences is an even median, with either choice valid when there are two
central entries. A median minimizes the sum of absolute deviations over the
real numbers. This even median therefore attains both the even-offset and
integer-offset minima, proving the formula.
\end{proof}

\begingroup\emergencystretch=1em
The factor of one half in the quotient formula suggests using the half-height
difference. Since both heights have the same parity at each cell, this
difference is integer-valued and puts the distance in dispersion form.
\par\endgroup

\begin{corollary}[Half-height reduction]\label{cor:half-height}
For integers $m,n\ge2$ and normalized heights $h,h'$ of two
vertices of $\OFG(\Mm)$, the half-difference $\phi=(h'-h)/2$ is integer-valued
and $1$-Lipschitz on $\Gmn$, and
\[
  \ofgdist(h,h')=\min_{L\in\Z}
  \sum_{v\in V(\Gmn)}|\phi(v)-L|.
\]
\end{corollary}

\begin{proof}
Both normalized heights have the parity of grid distance from $(1,1)$,
since their root values are zero and their edge increments are $\pm1$.
Their difference is therefore even, so $\phi$ is integer-valued.
Across an edge, $\phi$ changes by $-1$, $0$ or $1$. Summing these changes
along a shortest grid path gives the $1$-Lipschitz bound.
The proof of Proposition~\ref{prop:quotient-distance} shows that the minimum over
integer offsets is unchanged when restricted to even offsets.
Writing $K=2L$ with $L\in\Z$ and using $h'-h=2\phi$ gives
\begin{align*}
  \ofgdist(h,h')
    &=\frac12\min_{L\in\Z}\sum_{v\in V(\Gmn)}|2\phi(v)-2L|\\
    &=\min_{L\in\Z}\sum_{v\in V(\Gmn)}|\phi(v)-L|.\qedhere
\end{align*}
\end{proof}

Dispersion depends only on the multiset of values, so sorting can simplify
the function without changing the distance objective. Sorting rows and then
columns makes it nondecreasing in both coordinates and preserves the
Lipschitz bound.

\begin{lemma}[Monotone reduction]\label{lem:monotone-reduction}
Every integer $1$-Lipschitz function on $\Gmn$ has the same value multiset as a
coordinatewise nondecreasing integer $1$-Lipschitz function on $\Gmn$.
\end{lemma}

\begin{proof}
The rows are first sorted in nondecreasing order. An integer
$1$-Lipschitz row contains every integer between its extremes, so
its sorted neighbours still differ by at most one. Sorting also
preserves pointwise inequalities. For equal-length lists $x$ and $y$
with $x_i\le y_i$ at each index, at least $k$ entries of $y$ are
at most its $k$th smallest entry. Their matching entries of $x$ are
no larger, so neither is its $k$th smallest entry. Sorting commutes with adding constants. For adjacent rows, sorting therefore preserves the vertical bound
\[
 y-1\le x\le y+1.
\]

\medskip\noindent
The columns are then sorted in nondecreasing order. Each column is
still integer $1$-Lipschitz, so it contains every integer between its
extremes and remains $1$-Lipschitz after sorting. Before this second sort,
adjacent columns are pointwise ordered and differ by at most one.
The order-statistic comparison preserves both inequalities, hence the
horizontal bounds and row order. Both sorts preserve the value multiset,
giving the required function.
\end{proof}

The earlier reduction~\cite[Version~2, inequality~(4)]{Gupta2026companion}
asks for an upper bound on every integer $1$-Lipschitz grid function,
including functions not realizable as half-height differences. The bound is
\[
 \disp(f)\le D(m,n).
\]

The opposite degree-two distance and its parity-dependent closed form were
already established in~\cite[Theorem~17]{Christensen2025origami}.
The matching diameter upper bound is established
\cite[Theorem~1.2]{Watts2026diameter}. For $m,n\ge2$, the shorter and longer sides are
$r=\min(m,n)$ and $t=\max(m,n)$; the indicator $\epsilon$ is one when $m+n$ is odd
and zero otherwise. The antidiagonal-sum evaluation is given in~\cite[Proposition~5.3]{Watts2026diameter}. The known rectangular diameter is
\[
 \diameter\OFG(\Mm)=D(m,n)
 =\frac{r\bigl(3t^2+r^2-4+3\epsilon\bigr)}{12}.
\]
The two-row specialization is the earlier exact
diameter~\cite[Corollary~19]{Christensen2025origami}.

The grid bound also has the optional strict-chain derivation in
Appendix~\ref{app:diameter-routes}, using
\cite[Theorem~1.1]{Watts2026diameter}. To classify equality rather than
only bound the distance, the argument below uses the separately
communicated covariance comparison.

\subsection{Dispersion maximizers}\label{subsec:dispersion-equality}

The known diameter leaves its attaining pairs to be identified. This requires
every maximizing half-height difference and every height pair realizing it.
Sorting first reduces the dispersion problem to nondecreasing functions.
Equality in the covariance comparison restricts their coordinate dependence,
and reversing the sorting recovers the original functions.

The comparison and its proof are due to Haobo Ma and Wenlin Zhang, who
communicated them in a private note. The calculation below is the communicated
proof.
The equality criterion and the classifications extend the comparison.
Covariance and monotone rearrangement already appear in classical path
comparisons~\cite[Lemmas~4.8 and~4.9]{Bobkov2006subgaussian}, while maximal
centred exponential moments tensorize on graph products~\cite[Theorem~3.6]{Alon1998isoperimetric}.
Those moment comparisons concern a different objective. The question here
is which integer functions attain the median-deviation bound and which
height pairs realize them.

All averages and covariances below use the uniform measure on the grid.
A function depends on a coordinate if changing that coordinate while
holding the other fixed can change the function. For a function $u$ on
$G_{m,n}$, its restriction and average on row $i$ are $u_i(j)=u(i,j)$ and
$\bar u_i=n^{-1}\sum_{j=1}^n u_i(j)$.
The communicated note proves covariance nonnegativity using the chain
identity and row conditioning below. The additional assertion here
characterizes when that covariance vanishes.

\begin{lemma}[Zero covariance on two chains]\label{lem:equality-zero-covariance}
For integers $m,n\ge2$ and coordinatewise nondecreasing
real-valued functions $u,v$ on $G_{m,n}$, $\operatorname{Cov}(u,v)\ge0$,
with equality iff the functions depend on disjoint coordinates.
\end{lemma}

\begin{proof}
Conditioning on the row separates covariance into two nonnegative terms.
Equality forces one function to be independent of the row and, unless that
function is constant, the other to be independent of the column.
For nondecreasing real functions $p,q$ on a uniform chain of length $L$, products of paired increments are nonnegative. If both functions are nonconstant, the first and last indices give a strictly positive product. If one function is constant, every product is zero. The covariance therefore vanishes exactly when at least one function is constant, by the identity
\[
 \operatorname{Cov}(p,q)
 =\frac{1}{2L^2}\sum_{r=1}^L\sum_{s=1}^L
 (p(r)-p(s))(q(r)-q(s)).
\]

\medskip\noindent
The covariance decomposition on the grid is
\[
  \operatorname{Cov}(u,v)
  =\frac1m\sum_{i=1}^m\operatorname{Cov}(u_i,v_i)
    +\operatorname{Cov}(\bar u,\bar v).
\]
The row restrictions and row averages are nondecreasing, so both terms are
nonnegative by the chain identity. If $u$ depends on the row coordinate, some later row has a strictly larger value at some column and no
smaller value at every column. Its average is therefore strictly larger.
The same holds for $v$. If both functions depend on the row, the covariance
of their nonconstant row averages is positive.

\medskip\noindent
Suppose that the grid covariance is zero. At least one function is
independent of the row. Symmetry permits this to be $u$. If $u$ is
constant, it depends on neither coordinate. Otherwise, every row restriction
$u_i$ is the same nonconstant function of the column. Each nonnegative
within-row covariance must vanish because their sum is zero. The chain
identity then makes every $v_i$ constant, so $v$ is independent of the
column. The functions depend on disjoint coordinates.

\medskip\noindent
Conversely, if the functions depend on disjoint coordinates, either one function is
constant or each depends on a different coordinate. In the latter case the
two coordinates are independent under the uniform product measure, so the
functions are independent. In both cases their covariance is zero.
\end{proof}

The corner-distance function $A_0(i,j)=i+j-2$ satisfies $A_0(1,1)=0$ and $\disp(A_0)=D(m,n)$.
For a coordinatewise nondecreasing integer $1$-Lipschitz function $F$,
the function $R=A_0-F$ is also nondecreasing because each forward increment
of $F$ is zero or one. The sign convention is $\operatorname{sgn}(0)=0$. The absolute-deviation cost at a real centre $t$, the mean and the sign about the mean are
\[
  C_F(t)=\sum_{v\in V(\Gmn)}|F(v)-t|,
  \qquad \mu=\mathbb EF,
  \qquad s=\operatorname{sgn}(F-\mu).
\]
The communicated comparison is
\begin{equation}\label{eq:quoted-dispersion-comparison}
\begin{aligned}
  \disp(F)&\le C_F(\mu)\le D(m,n),\\
  D(m,n)-C_F(\mu)&\ge mn\operatorname{Cov}(s,R).
\end{aligned}
\end{equation}

\begin{proof}[Proof of the covariance comparison]
An integer-valued
sample has an integer median, so $\disp(F)\le C_F(\mu)$. Reflection across
the grid centre replaces $A_0$ by $m+n-2-A_0$, making its mean a median.
An integer median has the same cost, so
\[
  \sum_v|A_0(v)-\mathbb E A_0|=\disp(A_0)=D(m,n).
\]
Since $|s|\le1$, this gives
\[
  mn\operatorname{Cov}(s,A_0)
  =\sum_v s(v)\bigl(A_0(v)-\mathbb E A_0\bigr)
  \le D(m,n).
\]
Writing $F=A_0-R$ then gives
\[
\begin{aligned}
  C_F(\mu)
  &=\sum_v s(v)\bigl(F(v)-\mu\bigr)\\
  &=mn\operatorname{Cov}(s,F)\\
  &=mn\operatorname{Cov}(s,A_0)-mn\operatorname{Cov}(s,R)\\
  &\le D(m,n)-mn\operatorname{Cov}(s,R).
\end{aligned}
\]
Both $s$ and $R$ are nondecreasing, so
Lemma~\ref{lem:equality-zero-covariance} gives
$\operatorname{Cov}(s,R)\ge0$. This proves both inequalities
in~\eqref{eq:quoted-dispersion-comparison}.

\end{proof}

\begin{lemma}[Nondecreasing dispersion maximizers]\label{lem:sorted-dispersion-equality}
For integers $m,n\ge2$ and a coordinatewise nondecreasing integer
$1$-Lipschitz function $F$ on $G_{m,n}$, $\disp(F)=D(m,n)$ iff, for some
integer $c$, at least one of the following descriptions holds:
\[
\begin{aligned}
  F(i,j)&=i+j-2+c\\
  F(i,j)&=i-1+c &&\text{with }n=2\text{ and }m\text{ even}\\
  F(i,j)&=j-1+c &&\text{with }m=2\text{ and }n\text{ even}.
\end{aligned}
\]
\end{lemma}

\begin{proof}
Equality in the covariance comparison separates the coordinates of the sign
from those of the residual. A nonconstant residual can leave only one
coordinate in $F$, whose dispersion determines the allowed rectangles.
The corner function is $A_0(i,j)=i+j-2$, its residual is $R=A_0-F$,
the mean is $\mu=\mathbb EF$ and the sign is $s=\operatorname{sgn}(F-\mu)$.
Suppose first that $\disp(F)=D(m,n)$. Since $D(m,n)>0$, the function $F$
is nonconstant and its values occur on both sides of their mean. Thus
the sign function $s$ is nonconstant.
The covariance is nonnegative by Lemma~\ref{lem:equality-zero-covariance}.
The comparison~\eqref{eq:quoted-dispersion-comparison} therefore gives
\[
 D(m,n)=\disp(F)\le C_F(\mu)
 \le D(m,n)-mn\operatorname{Cov}(s,R)\le D(m,n).
\]
Equality throughout forces $\operatorname{Cov}(s,R)=0$.
If $R$ is constant, then $F=A_0+c$, giving the first family.

\medskip\noindent
If $R$ is nonconstant, Lemma~\ref{lem:equality-zero-covariance} makes the
coordinates on which $s$ and $R$ depend disjoint. Both functions are nonconstant,
so each depends on exactly one coordinate. Suppose that $s$ depends only
on the row and $R$ only on the column. Then
$F(i,j)=i-1+g(j)$ for a nondecreasing integer function $g$.
If $g$ is nonconstant, its range has length at least one. The range
of $F$ is therefore covered by the overlapping or touching intervals
\[
  [i-1+\min g,\ i-1+\max g],\qquad 1\le i\le m.
\]
The mean $\mu$ belongs to one of
these intervals. In that row, values at the minimum and maximum of $g$ have different
signs relative to the mean. This remains true when the mean is an endpoint,
since the two signs are then zero and nonzero. This contradicts the
column independence of $s$, so $g$ is constant and $F=i-1+c$.
Transposition gives $F=j-1+c$ for the other coordinate choice.

\medskip\noindent
For $F=i-1+c$, each of $m$ consecutive integers occurs $n$ times.
Pairing their distances from a central median gives
\[
  \disp(F)=n\left\lfloor\frac{m^2}{4}\right\rfloor.
\]
The indicator $\delta_m$ is one for odd $m$ and zero for even $m$;
$\epsilon$ is the parity indicator in the displayed closed form of $D(m,n)$.
If $m\ge n$, that formula gives
\[
  D(m,n)-\disp(F)
  =\frac{n}{12}\bigl(n^2-4+3\epsilon+3\delta_m\bigr).
\]
For $n\ge3$ this is positive. For $n=2$, the indicators satisfy
$\epsilon=\delta_m$, so it vanishes exactly when $m$ is even.
If $m<n$, then $m\ge2$ and $n>m$. The same formula therefore gives
\[
  D(m,n)-\disp(F)
  =\frac{m\bigl(3n(n-m)+m^2-4+3\epsilon\bigr)+3n\delta_m}{12}>0.
\]
Hence the row-coordinate function attains the
bound exactly when $n=2$ and $m$ is even. Transposition gives the last family.

\medskip\noindent
Conversely, $F=A_0+c$ has the same dispersion as $A_0$ by translation.
The coordinate calculation gives equality in the second family, and
transposition gives it in the third. This proves sufficiency for every
listed function.
\end{proof}

Sorting preserves the value multiset, but an attaining multiset can have
more than one Lipschitz arrangement. The next result reverses that step.

\begin{theorem}[Dispersion maximizers]\label{thm:dispersion-equality}
For integers $m,n\ge2$ and an integer $1$-Lipschitz function $f$ on
$G_{m,n}$, $\disp(f)=D(m,n)$ iff $f(v)=\dgrid(p,v)+c$ for a grid
corner $p$ and an integer $c$, or $m=n=2$ and $f$ takes two consecutive
integer values twice each, or one of the following descriptions holds for
an integer $c$ and a sign $\eta\in\{-1,1\}$:
\[
\begin{aligned}
  f(i,j)&=\eta(j-1)+c &&\text{with }m=2,\ n\ge4\text{ even}\\
  f(i,j)&=\eta(i-1)+c &&\text{with }n=2,\ m\ge4\text{ even}.
\end{aligned}
\]
\end{theorem}

\begin{proof}
The nondecreasing classification leaves the coordinate-sum multiset and,
on even thin rectangles, a repeated coordinate multiset. The first is
rigid because its range equals the face grid's maximum distance $m+n-2$.
The second requires
checking the possible positions of a minimum and a maximum.
Suppose first that $\disp(f)=D(m,n)$. By
Lemma~\ref{lem:monotone-reduction}, sorting produces a nondecreasing
integer $1$-Lipschitz function with the same multiset and dispersion.
Lemma~\ref{lem:sorted-dispersion-equality} lists its possibilities.

\medskip\noindent
\emph{Coordinate-sum multiset.}
If the sorted function is $A_0+c$, the range of $f$ is $m+n-2$.
The finite grid has a minimum cell $p$ and a maximum cell $q$.
Their value difference is the range, so Lipschitzness requires
$\dgrid(p,q)\ge m+n-2$. Grid distance is at most $m+n-2$, and equality
requires both coordinate differences to be maximal, so $p,q$ are opposite
corners. Every cell $v$ lies on a shortest path between them. The bounds
from both ends give
\[
  f(q)-\dgrid(v,q)\le f(v)\le f(p)+\dgrid(p,v).
\]
The two bounds coincide because $f(q)-f(p)=\dgrid(p,q)$ and
$\dgrid(p,v)+\dgrid(v,q)=\dgrid(p,q)$. Thus
$f(v)=f(p)+\dgrid(p,v)$, giving the first family.

\medskip\noindent
\emph{Even thin rectangles.}
Transposition permits taking a $2\times n$ rectangle with even $n\ge4$.
If the sorted function is $j-1+c$, subtracting the minimum leaves all
increments unchanged and makes each value $0,\ldots,n-1$ occur twice.
For a minimum cell $p$ and a maximum cell $q$,
their distance is at least $n-1$.

If they are in the same row, they are its endpoints, and equality of the
range and endpoint distance forces that row to be $0,\ldots,n-1$ up to
reflection. The other row contains each value once. Adjacent differences
are nonzero integers of absolute value at most one, hence are $+1$ or $-1$.
A change of sign would immediately repeat a value, so the row is monotone.
The reversed order violates a vertical edge bound at an endpoint because
$n-1>1$. Both rows are therefore equal, giving the long coordinate up to sign.

If $p,q$ are opposite corners, axis reflections place $p=(1,1)$ and
$q=(2,n)$ without changing an edge bound. Forward edges now point right
or down. Every such edge belongs to a shortest $p$--$q$ path, whose length
is $n$. On each path, the increments of $f$ are $\delta_1,\ldots,\delta_n$. Since each
$\delta_t\le1$ and their sum is $n-1$,
\[
  \sum_{t=1}^n(1-\delta_t)=1.
\]
These are nonnegative integer deficits, so exactly one is one and all
others are zero. Every forward increment is consequently zero or one,
making $f$ nondecreasing. Lemma~\ref{lem:sorted-dispersion-equality}
leaves only the long coordinate. The function $A_0+c$ has range $n$, not $n-1$, and the
short-coordinate exception would require $n=2$.

In the remaining case, $p,q$ lie in different rows and their columns differ
by $n-2$. Their distance $n-1$ equals the range. On the intervening
$2\times(n-1)$ rectangle, the endpoint bounds therefore force a corner
distance. Its extreme values occur once each and every intermediate value
twice. The omitted column must contain the missing minimum and maximum.
Those two cells are adjacent but differ by $n-1>1$, a contradiction.
The possibilities are exhaustive because different-row cells at distance
at least $n-1$ have column difference $n-2$ or $n-1$.

\medskip\noindent
\emph{The square.}
On $2\times2$, the range is at most two. Range two gives a corner distance
by the endpoint argument. For range one, if the smaller value occurs $k$
times, dispersion is $\min(k,4-k)$. This equals $D(2,2)=2$ exactly when
$k=2$. Range zero has dispersion zero. These are exactly the stated
two-value square functions and the corner distances.

\medskip\noindent
Conversely, every listed function has the multiset of a nondecreasing
maximizer in Lemma~\ref{lem:sorted-dispersion-equality}. Corner reflections,
additive constants and the listed square arrangements preserve dispersion.
Every listed function therefore attains $D(m,n)$.
\end{proof}

\subsection{Height-pair reconstruction}\label{subsec:farthest-reconstruction}

The maximizing functions determine the possible half-height differences,
but not yet the pairs of assignments. Outside the two-by-two square,
reconstructing those pairs returns the classification to the four
degree-two heights.
For $\varepsilon=(\varepsilon_1,\varepsilon_2)\in\{-1,1\}^2$, the corner cone is
\[
  h_\varepsilon(i,j)=\varepsilon_1(i-1)+\varepsilon_2(j-1).
\]
These are the four degree-two states identified in
Proposition~\ref{prop:degrees-two-three}.
For an integer grid function $\phi$ with $\phi(1,1)=0$, a height-pair
reconstruction means a pair of normalized heights $h,h'$ satisfying
$(h'-h)/2=\phi$. Different height pairs realizing the same $\phi$ are
distinct reconstructions.
On $G_{2,2}$, the function $\chi(i,j)$ is the residue of $i+j-2$ modulo two, taking values in
$\{0,1\}$. The corner cones and the functions $\chi,-\chi$ are normalized heights, so they represent
flip-graph vertices. The known equality $\diameter\OFG(\Mm)=D(m,n)$ and
Corollary~\ref{cor:half-height} make a pair farthest exactly when its
half-height difference attains the dispersion bound.

\begin{maintheorem}[Farthest Miura-ori pairs]\label{cor:farthest-pairs}
For integers $m,n\ge2$, the unordered farthest pairs in $\OFG(M_{m,n})$
are classified as follows. If $m=n=2$, they are all six unordered pairs
of distinct corner cones and the additional pair $\{\chi,-\chi\}$.
Otherwise, they are the two pairs $\{h_\varepsilon,h_{-\varepsilon}\}$
with $\varepsilon\in\{-1,1\}^2$, together with the following pairs for
$\eta\in\{-1,1\}$ under the displayed
dimension conditions:
\[
\begin{aligned}
  &\{h_{(\eta,1)},h_{(\eta,-1)}\}
    &&\text{if }m=2,\ n\ge4\text{ even}\\
  &\{h_{(1,\eta)},h_{(-1,\eta)}\}
    &&\text{if }n=2,\ m\ge4\text{ even}.
\end{aligned}
\]
There is no additional pair when neither dimension condition holds.
Thus the number of unordered farthest pairs is
\[
\begin{cases}
  7,&m=n=2,\\
  4,&\min(m,n)=2\text{ and }\max(m,n)\ge4\text{ is even},\\
  2,&\text{otherwise}.
\end{cases}
\]
\end{maintheorem}

\begin{proof}
An extremal half-height difference fixes the two height increments on
every edge where it changes. On its constant edges, the common increments
must satisfy the square-cycle conditions. Tracking these choices retains
all height-pair reconstructions, including distinct pairs with the same difference.
For normalized heights $h,h'$, their half-height difference is $\phi=(h'-h)/2$. The known diameter
and Corollary~\ref{cor:half-height} give
\[
  \ofgdist(h,h')=D(m,n)\quad\Longleftrightarrow\quad
  \disp(\phi)=D(m,n).
\]
The function $\phi$ is integer $1$-Lipschitz with $\phi(1,1)=0$, so
Theorem~\ref{thm:dispersion-equality} supplies every possibility.
For an oriented edge, the relation between increments is
\[
  \Delta h'-\Delta h=2\Delta\phi,
  \qquad \Delta h,\Delta h'\in\{-1,1\}.
\]

\medskip\noindent
\emph{Corner differences.}
If $\phi$ is a corner distance plus its root-normalizing constant, each
edge has $|\Delta\phi|=1$. The increment equation forces
$\Delta h=-\Delta\phi$ and $\Delta h'=\Delta\phi$.
The common zero root therefore gives $h=-\phi$ and $h'=\phi$.
The four possible normalized differences are the four corner cones.
Reversal identifies each with its negative, giving the two opposite pairs.

\medskip\noindent
\emph{Even thin rectangles.}
On a $2\times n$ rectangle with even $n\ge4$, the normalized difference
$\phi(i,j)=j-1$ forces horizontal increments of $h,h'$ to be $-1,+1$.
Vertical increments agree for the two
heights. For $h$, equality of the sums along the two paths around a unit
square makes the vertical increments in its two columns equal, since
the horizontal increment is $-1$ in both rows. Thus the vertical increment
is one constant $\eta\in\{-1,1\}$ across all columns. The two reconstructions are
\[
  h(i,j)=\eta(i-1)-(j-1),\qquad
  h'(i,j)=\eta(i-1)+(j-1).
\]
They give the two additional pairs in the statement. Their differences
agree, but their short-coordinate signs distinguish the pairs.
The difference $-\phi$ reverses them and adds no unordered pair.
Transposition gives the other thin orientation.

\medskip\noindent
\emph{The square.}
For $m=n=2$, a two-value difference takes values $0,1$ or $-1,0$, since
it is zero at the root. If its equal-value cells form rows or columns,
the difference is one of $\pm(i-1)$ and $\pm(j-1)$.
The square-cycle calculation above also holds at side length two, giving
two reconstructions for each positive axial difference. The negative
differences reverse those ordered pairs, so these give the four
non-opposite pairs among the corner cones. Otherwise, the equal-value
cells lie on diagonals, giving $\chi$ or $-\chi$. Every edge then has
$|\Delta\phi|=1$, so the increments are forced and give the unique
unordered pair $\{\chi,-\chi\}$.
Together with the two opposite-cone pairs, these give seven pairs.

\medskip\noindent
Every reconstructed pair consists of normalized heights and has a listed
extremal difference, so the half-height formula gives distance $D(m,n)$.
An ordered pair has exactly one difference function, so pairs reconstructed
from distinct differences cannot coincide. Reversing an ordered pair negates
its difference. Since $D(m,n)>0$, a farthest pair has
nonzero difference and is not fixed by reversal. This accounts for every
unordered pair and gives the displayed counts.
\end{proof}

Outside $2\times2$, every farthest pair therefore has both endpoints among
the four degree-two states. Only two or four of their six unordered pairs
are farthest.

\enlargethispage{2\baselineskip}
\section{Discussion}\label{sec:discussion}

The counting and distance results give complementary information about
the same assignments. Degree records how many legal moves are available,
but not where those moves occur. The degree-five example and interior cone
in Section~\ref{sec:geometry} have the same degree on the same grid and
different flip locations. Nor does minimum degree alone characterize a
farthest pair. Outside the two-by-two square, only two or four of the six
pairs among the four degree-two assignments attain the diameter. The
reconstruction step distinguishes all attaining pairs, including those
with the same height difference.

For the counting results, the remaining question concerns the degree of
the short-side corrections. For fixed $2\le s\le d-2$ and $n\ge d-1$,
Conjecture~\ref{conj:sharp-boundary} proposes the cap $d-s-1$ on their degree
in $n$. The proved bounds and Baxter identity establish this cap for
$s\le4$ and $s\ge d-4$, covering every short dimension through degree nine.
It remains open for $d\ge10$ and $5\le s\le d-5$. Because the correction
is a signed sum, the question is whether every coefficient above this cap
vanishes in the complete sum. Cancellation within a restricted class of
contributions does not settle that question.

The rectangular face grid is also part of the mathematical scope.
For another crease pattern, a colouring correspondence alone does not
ensure that consistent integer heights exist around cycles. Even when
such heights exist, polynomiality, onset and diameter require separate
arguments for the new face graph. The present results therefore do not
establish those properties for every origami flip graph.

The state space imposes a separate limit on physical interpretations.
Locally flat-foldable assignments need not admit globally compatible layer
orders~\cite{GineproHull2014counting}. Removing such assignments can
invalidate a shortest path, so distances in the full flip graph need not
describe conversions within the globally flat-foldable subset. Realizing
a graph edge as a physical operation further requires a folding motion.
Neither stability nor energy cost is specified by the counts or distances.

\section{Conclusion}\label{sec:conclusion}

Integer height functions give a common description of local moves and
maximum separation in the rectangular Miura-ori flip graph. Exact counting
laws apply once both dimensions are large enough for a fixed number of
available flips, and explicit corrections cover a fixed short dimension
as the other grows. The farthest-pair classification identifies every
maximally separated pair on every rectangle. These complementary
descriptions make local choices and extreme distances explicit without
constructing the full flip graph.

\FloatBarrier
\hypertarget{maintext-end}{}\label{sec:maintext-end}

\appendix
\section{Worked maximum-test illustration}\label{app:low-degree}

\subsection{Maximum tests for geometric classifications}\label{subsec:maximum-tests}

The Maxima Criterion distinguishes a tie between cones from a strict local
maximum by testing every existing neighbour direction.
In Figure~\ref{fig:max-criterion}, the three
cones from Figure~\ref{fig:kcone-envelope} attain height $4$ at
$v=(3,3)$. The cones based at $p_1,p_2$ decrease on the upward step,
the cone based at $p_3$ on the downward step, the cones based at $p_1,p_3$
on the leftward step, and the cone based at $p_2$ on the rightward step.
For each neighbour of $v$, some attaining cone decreases on the step
from $v$, so Lemma~\ref{lem:maxima} makes $v$ a strict local maximum.

\begin{figure}[ht]
\centering
\begin{tikzpicture}[x=1.1cm, y=-1.1cm, font=\small]
  \foreach \gridrow in {1,...,5} {
    \draw[gray!50] (1,\gridrow) -- (5,\gridrow);
  }
  \foreach \gridcolumn in {1,...,5} {
    \draw[gray!50] (\gridcolumn,1) -- (\gridcolumn,5);
  }
  \foreach \i/\j/\h in {
    1/1/0, 1/2/1, 1/3/2, 1/4/1, 1/5/0,
    2/1/1, 2/2/2, 2/3/3, 2/4/2, 2/5/1,
    3/1/2, 3/2/3, 3/3/4, 3/4/3, 3/5/2,
    4/1/1, 4/2/2, 4/3/3, 4/4/4, 4/5/3,
    5/1/0, 5/2/1, 5/3/2, 5/4/3, 5/5/4
  }
    \node[circle, draw=black, fill=white, minimum size=7mm, inner sep=0pt]
      at (\j,\i) {$\h$};
  \foreach \i/\j/\h in {1/1/0, 1/5/0, 5/1/0}
    \node[circle, draw=black, fill=gray!30, minimum size=7mm, inner sep=0pt]
      at (\j,\i) {$\h$};
  \foreach \i/\j/\h in {2/3/3, 4/3/3, 3/2/3, 3/4/3}
    \node[circle, draw=black, fill=gray!15, minimum size=7mm, inner sep=0pt]
      at (\j,\i) {$\h$};
  \node[rectangle, draw=black, very thick, fill=white,
        minimum size=7mm, inner sep=0pt] at (3,3) {$4$};
  \foreach \i/\j/\h in {4/4/4, 5/5/4}
    \node[rectangle, draw=black, very thick, fill=white,
          minimum size=7mm, inner sep=0pt] at (\j,\i) {$\h$};
  \node[font=\footnotesize, anchor=east] at (0.6,1) {$p_1$};
  \node[font=\footnotesize, anchor=west] at (5.4,1) {$p_2$};
  \node[font=\footnotesize, anchor=east] at (0.6,5) {$p_3$};
  \node[font=\footnotesize, anchor=west] at (5.4,3) {$v=(3,3)$};
  \node[font=\scriptsize, anchor=north] at (3, 2.32) {$p_1,p_2$};
  \node[font=\scriptsize, anchor=north] at (3, 4.32) {$p_3$};
  \node[font=\scriptsize, anchor=north] at (2, 3.32) {$p_1,p_3$};
  \node[font=\scriptsize, anchor=north] at (4, 3.32) {$p_2$};
\end{tikzpicture}
\caption{Three-cone envelope from Figure~\ref{fig:kcone-envelope}.
Dark-grey circles mark minima, light-grey circles mark the four neighbours of
$v=(3,3)$, and boxes mark maxima. The tag below each light-grey circle lists
the attaining cones that decrease from $v$ toward that neighbour.}
\label{fig:max-criterion}
\end{figure}

At $(4,4)$ and $(5,5)$, only the cones based at $p_2,p_3$ attain height $4$.
At $(4,4)$, the cone based at $p_2$ decreases toward the upper and right
neighbours, while the cone based at $p_3$ decreases toward the lower and
left neighbours. At the corner $(5,5)$, only the upper neighbour toward
$p_2$ and the left neighbour toward $p_3$ exist. Both cells therefore
satisfy the Maxima Criterion.

For $m,n\ge2$, a height $h$ has configuration $(A,c)$ of
strict local minima and their values given by Theorem~\ref{thm:envelope},
so $h=E_{A,c}$.
By Lemma~\ref{cor:corners}, $h$ has at most four single-cone maxima,
all at corners. Every remaining maximum $v$ has at least two attaining
cones, so for some distinct indices $s,t$ it satisfies
\[
  c_s+\dgrid(p_s,v)=c_t+\dgrid(p_t,v)=E_{A,c}(v).
\]
The maxima in Figure~\ref{fig:max-criterion} illustrate these attaining ties.
A tie alone does not ensure maximality. For each existing neighbour,
some attaining cone must decrease on the step to that neighbour.
Figure~\ref{fig:max-criterion} also shows a nonmaximum tie at $(2,3)$,
whose neighbour $(3,3)$ has greater height.
For admissible configurations with two apexes in distinct rows and columns,
Lemma~\ref{lem:two-apex-ridge} characterizes tie-cell maxima by requiring each
coordinate to be an endpoint of its grid range or lie strictly between
the corresponding apex coordinates. The same direction test is used in
the Presburger construction in Appendix~\ref{app:alternative-routes}.

\section{Explicit polynomial data}\label{app:polynomials}

\subsection{Closed forms on the high region}\label{subsec:high}

\label{cor:low-degree-counts}
Proposition~\ref{thm:degree-two} gives $p_2\equiv4$ on $m,n\ge2$.
For $d\ge3$, the formulas below agree with $E_d$ on
$H_d=\{m,n\ge d-1\}$ by Theorem~\ref{thm:poly}.
The degree-three formula is $p_3=4(m+n)-16$.
The exact count grids and computational warrants for the remaining
formulas are given in Section~\ref{app:coefficient-data}.

{\allowdisplaybreaks\setlength{\jot}{2pt}
\begin{align*}
  p_4 &= 2(m^2+n^2) + 6mn - 10(m+n) - 4,\\
  p_5 &= \tfrac{2}{3}(m^3+n^3) + 2(m^2+n^2) + 50mn - \tfrac{392}{3}(m+n) + 264,\\
  p_6 &= \tfrac{1}{6}(m^4+n^4) + \tfrac{1}{3}(m^3+n^3) + 38(m^2n+mn^2)\\
      &\quad{} - \tfrac{229}{6}(m^2+n^2) - 272mn - \tfrac{103}{3}(m+n) + 1176,\\
  p_7 &= \tfrac{1}{30}(m^5+n^5) + \tfrac{37}{3}(m^3n+mn^3) + 18m^2n^2
         + \tfrac{7}{6}(m^3+n^3) + 25(m^2n+mn^2)\\
      &\quad{} - 496(m^2+n^2) - \tfrac{3818}{3}mn + \tfrac{14354}{5}(m+n) - 904,\\
  p_8 &= \tfrac{1}{180}(m^6+n^6)\\
      &\quad{} - \tfrac{1}{60}(m^5+n^5) + 3(m^4n+mn^4) + \tfrac{14}{3}(m^3n^2+m^2n^3)\\
      &\quad{} + \tfrac{95}{36}(m^4+n^4) + \tfrac{91}{3}(m^3n+mn^3) + 210 m^2n^2\\
      &\quad{} - \tfrac{1399}{12}(m^3+n^3) - \tfrac{2780}{3}(m^2n+mn^2)\\
      &\quad{} - \tfrac{41609}{45}(m^2+n^2) - \tfrac{11822}{3}mn
        + \tfrac{125353}{5}(m+n) - 56788,\\
  p_9 &= \tfrac{1}{1260}(m^7+n^7)\\
      &\quad{} - \tfrac{1}{180}(m^6+n^6) + \tfrac{11}{15}(m^5n+mn^5)
        + (m^4n^2+m^2n^4) + 2m^3n^3\\
      &\quad{} + \tfrac{97}{180}(m^5+n^5) + \tfrac{7}{2}(m^4n+mn^4)
        + 99(m^3n^2+m^2n^3)\\
      &\quad{} - \tfrac{155}{36}(m^4+n^4) + \tfrac{449}{3}(m^3n+mn^3) - 400 m^2n^2\\
      &\quad{} - \tfrac{68852}{45}(m^3+n^3) - \tfrac{16657}{2}(m^2n+mn^2)\\
      &\quad{} + \tfrac{640274}{45}(m^2+n^2) + \tfrac{220166}{5}mn
        + \tfrac{4101458}{105}(m+n) - 368848,\\
  p_{10} &= \tfrac{1}{10080}(m^8+n^8)\\
         &\quad{} - \tfrac{1}{840}(m^7+n^7) + \tfrac{13}{90}(m^6n+mn^6)
           + \tfrac{4}{15}(m^5n^2+m^2n^5) + \tfrac{7}{18}(m^4n^3+m^3n^4)\\
         &\quad{} + \tfrac{67}{720}(m^6+n^6) - \tfrac{21}{20}(m^5n+mn^5)
           + \tfrac{91}{3}(m^4n^2+m^2n^4) + \tfrac{472}{9}m^3n^3\\
         &\quad{} + \tfrac{101}{30}(m^5+n^5) + \tfrac{2101}{18}(m^4n+mn^4)
           + \tfrac{3407}{18}(m^3n^2+m^2n^3)\\
         &\quad{} - \tfrac{421393}{1440}(m^4+n^4) - \tfrac{120535}{36}(m^3n+mn^3)
           - \tfrac{10004}{3}m^2n^2\\
         &\quad{} - \tfrac{562697}{120}(m^3+n^3) - \tfrac{159259}{5}(m^2n+mn^2)\\
         &\quad{} + \tfrac{455062963}{2520}(m^2+n^2) + \tfrac{19134272}{45}mn\\
         &\quad{} - \tfrac{137300147}{210}(m+n) - 885004.
\end{align*}}%

Table~\ref{tab:topdeg} collects the leading terms for comparison across
the displayed degrees. Its split columns distinguish pure powers from
the degree-four mixed term.

\begin{table}[H]
\centering
\renewcommand{\arraystretch}{1.4}
\begin{tabular}{@{}cccc@{}}
\toprule
$d$ & $C(d)=4/(d-2)!$ & \shortstack{Non-mixed\\top-degree part} &
\shortstack{Mixed\\top-degree part} \\
\midrule
$2$ & \multicolumn{3}{l@{}}{Constant $p_2\equiv4$} \\
\midrule
$3$ & $4$              & $4(m+n)$                    & $0$   \\
$4$ & $2$              & $2(m^2+n^2)$                & $6mn$ \\
$5$ & $\tfrac{2}{3}$   & $\tfrac{2}{3}(m^3+n^3)$     & $0$   \\
$6$ & $\tfrac{1}{6}$   & $\tfrac{1}{6}(m^4+n^4)$     & $0$   \\
$7$ & $\tfrac{1}{30}$  & $\tfrac{1}{30}(m^5+n^5)$    & $0$   \\
$8$  & $\tfrac{1}{180}$   & $\tfrac{1}{180}(m^6+n^6)$   & $0$   \\
$9$  & $\tfrac{1}{1260}$  & $\tfrac{1}{1260}(m^7+n^7)$  & $0$   \\
$10$ & $\tfrac{1}{10080}$ & $\tfrac{1}{10080}(m^8+n^8)$ & $0$   \\
\bottomrule
\end{tabular}
\caption{Top-degree parts of $p_d$. Degree two is the constant
four on $m,n\ge2$. For $d\ge3$, every displayed row is identified on $H_d$.
The last two combine the general degree theorem with reduced-node counts.}
\label{tab:topdeg}
\end{table}

The transfer generating functions of Section~\ref{sec:period} express
fixed-width polynomial coefficients through their principal parts at
$z=1$. The binomial basis relates these pole coefficients to the leading
terms in Table~\ref{tab:topdeg}. For $d\ge3$ and
fixed width $m\ge d-1$, the polynomial $p_d(m,n)$ has degree $d-2$ in
$n$ and agrees with $E_d(m,n)$ for $n\ge d-1$. Consequently,
$\sum_{n\ge2}E_d(m,n)z^n$ and $\sum_{n\ge0}p_d(m,n)z^n$ differ by a
polynomial in $z$ and have the same principal part at $z=1$. The binomial
expansion of $p_d(m,n)$ uses the basis $\binom{n+j}{j}$ for $j=0,\ldots,d-2$.
Each basis coefficient is equivalently a pole coefficient, by
\[
  \sum_{n\ge0}\binom{n+j}{j}z^n=(1-z)^{-(j+1)}.
\]
The coefficient of $n^{d-2}$ is $C(d)=4/(d-2)!$ from
Theorem~\ref{thm:total-degree}, so the leading pole term is
$4(1-z)^{-(d-1)}$. Subtracting $4\binom{n+d-2}{d-2}$ leaves a
polynomial of degree at most $d-3$. Its coefficient of $n^{d-3}$,
multiplied by $(d-3)!$, is the subleading pole coefficient at
$(1-z)^{-(d-2)}$.

This conversion uses the polynomial formulas, not an independent counting
check or a premise of the degree proofs. At degree three, the pure term
$4m$ makes the subleading pole coefficient width-dependent. At degree four,
the mixed term $6mn$ does so. For $d\ge5$, total degree $d-2$ allows width-dependence in
the coefficient of $n^{d-3}$ only through $mn^{d-3}$.
Theorem~\ref{thm:total-degree} excludes that mixed top-degree term, so the
subleading pole coefficient is independent of width for $m\ge d-1$.
The displayed formulas give $-4$, $-8$, $-12$, $-16$, $-20$ and $-24$
for degrees five through ten, in that order.

Boundary corrections can change the pole coefficients. At degree six
and width two, Table~\ref{tab:corr} gives, for $n\ge5$,
\[
  E_6(2,n)=p_6(2,n)+B_6^{(2)}(n)
    =4\binom{n+4}{4}-30n^2+78n-100.
\]
The finite terms with $n<5$ do not affect the principal part.
The quadratic contributes no pole term above order three, so the
leading pole coefficient is $4$ and the fourth-order coefficient is
zero. Thus the value $-8$ for $m\ge5$ does not hold at every width
$m\ge2$. The pole-coefficient comparison must therefore retain its
high-region width condition.

\subsection{The three-row total}\label{app:three-row-total}

The three-column total and transposition symmetry
in~\cite[Section~3 and Corollary~5]{GineproHull2014counting} give the known
three-row count. The column transfer below re-derives that total without
marking extrema. The resulting total is a target for checking a complete
degree enumeration rather than a premise of the polynomial certificates.

\begin{proposition}[Three-row vertex count]\label{prop:three-row-count}
The sequence has initial values $b_1=4$, $b_2=18$, and satisfies
\[
  b_{n+2}=5b_{n+1}-2b_n\qquad(n\ge1).
\]
For every $n\ge2$,
\[
  \lvert V\bigl(\OFG(M_{3,n})\bigr)\rvert=b_n.
\]
\end{proposition}

\begin{proof}
For $n\ge1$, the number of proper $3$-colourings of the $3\times n$ grid is
$C_n$. For $n\ge2$, the rectangular colouring correspondence identifies
the flip-graph vertices with global colour-rotation orbits. Each orbit has
three elements, since a nonidentity rotation changes every cell's colour,
so the vertex count is $C_n/3$.

\medskip\noindent
A proper column is of type $E$ when its first and third colours agree and
of type $D$ when its three colours are distinct. There are six columns of
each type. A type-$E$ column has three type-$E$ successors, two repeating its
middle colour and one the unused colour. Its two type-$D$ successors put its
repeated colour in the middle and exchange the other two colours.
A type-$D$ column has two type-$E$ successors, both repeating its middle
colour, and two type-$D$ successors, the derangements of its three colours.
With rows indexing the old type and columns the new type, in order $E,D$,
the transition matrix and count are
\[
  T=\begin{pmatrix}3&2\\2&2\end{pmatrix},
  \qquad
  C_n=(6,6)T^{n-1}\begin{pmatrix}1\\1\end{pmatrix}.
\]
The identity $T^2-5T+2I=0$, multiplied by $(6,6)T^{n-1}$ on the left and
by the column of ones on the right, gives $C_{n+2}=5C_{n+1}-2C_n$ for $n\ge1$.
Since $C_1=12$ and $C_2=54$, dividing by three gives the stated recurrence
and initial values for $b_n=C_n/3$.
\end{proof}

The sequence $18,82,374,1706,\ldots$ is recorded as OEIS A052913\footnote{\url{https://oeis.org/A052913}}.
The auxiliary value $b_1=4$ is a colour-quotient count for the $3\times1$ path,
not a Miura-ori flip-graph count.

Every height on $G_{3,n}$ has between two and $3n$ distinct extremal cells. Thus the known total and the degree-specific counts satisfy
\[
 \sum_{d=2}^{3n} E_d(3,n)=b_n\qquad(n\ge2).
\]
Complete histograms on $3\times2$, $3\times3$ and $3\times4$ were recomputed
by the pinned-colouring transfer and checked against direct integer-height
enumeration. Their totals at $n=2,3,4$ are $18$, $82$ and $374$, agreeing with
the recurrence. These finite checks validate extremum bookkeeping, not the
general polynomial certificates.

\section{Alternative proof routes}\label{app:alternative-routes}

The main counting and distance proofs do not use the arguments in this
appendix. The path correspondence extends the envelope construction to
the positive grids needed by the integer-point count. That count gives
piecewise quasi-polynomiality directly from cone constraints, while a
determinant calculation isolates rationality from frozen structure in
Appendix~\ref{app:coefficient-series}.
The converse and joint-type reductions give alternative certificates
for split-remainder bounds. The final two subsections compare the
degree-eight certificates and complete the earlier diameter reduction.

\subsection{Counts outside the rectangular model}\label{app:grid-extension}

For the independent lattice-point construction below, a height on a path
$G_{1,n}$ or $G_{m,1}$ has the same unit edge differences and root
normalization as on a rectangle. Its counts $N_{(a,b)}$ and $E_d$ use
the same numbers of distinct extremal cells. This extends the counting
domain, not the rectangular Miura model.

\begin{proposition}[Envelope correspondence on paths]\label{prop:path-envelope}
For positive integers $m,n$ with $\min(m,n)=1$ and $mn\ge2$, a height
on $G_{m,n}$ is the lower envelope of the distance cones based at its
strict local minima and the upper envelope of those based at its strict
local maxima. Admissible configurations and normalized heights correspond
bijectively by taking the lower envelope and recovering its minima and
their values.
\end{proposition}

\begin{proof}
The maps recover the minima and their values from a height and reconstruct
the height from an admissible configuration. For a normalized path height
$h$, its strict local minima form $P_h$, and
$\Phi(h)=(P_h,h|_{P_h})$. For a cell $v$, strict descent cannot revisit a
cell and therefore ends at a strict local minimum $p$ on the finite grid.
The descending path has length $h(v)-h(p)$. The unit edge differences
also give the Lipschitz bound, so
\[
  \dgrid(p,v)\le h(v)-h(p)\le\dgrid(p,v).
\]
Thus the cone based at $p$ attains $h(v)$, while the Lipschitz bound
puts every minimum's cone at or above $h(v)$. The lower envelope is $h$.
Equality in the Lipschitz bound between distinct minima would make the
first step from the higher minimum toward the other descend, contradicting
minimality. Hence the activity inequalities are strict. Summing unit
increments along the path between minima gives the parity condition.
The reconstructed envelope is zero at the root, so $\Phi(h)$ is admissible.

\medskip\noindent
Conversely, an admissible configuration maps to
$\Psi(A,c)=E_{A,c}$. The proof of Lemma~\ref{lem:parityvalid} uses
only the triangle inequality and the change of distance parity across
an edge, so it gives unit edge differences on the path. Root normalization
is part of admissibility. A cell away from the apexes has a lower neighbour
one step toward an attaining apex, as in Lemma~\ref{lem:minatapex}.
At an apex, activity and unit edge differences make every neighbour one
higher, as in the forward direction of Lemma~\ref{lem:activemin}.
Thus the envelope's minima are exactly $A$, with values $c$.
Both maps land in their stated classes, and the two recoveries give
\[
  \Psi(\Phi(h))=h,
  \qquad \Phi(\Psi(A,c))=(A,c).
\]
Negation preserves normalization and unit edge differences and exchanges
minima with maxima. Applying the lower-envelope identity to $-h$ gives
the upper-envelope identity for $h$.
\end{proof}

On a path with $mn\ge2$ cells, Lemma~\ref{lem:binom} gives
$E_d(m,n)=2\binom{mn-2}{d-2}$ after summing the extremum splits.
The unique normalized height on $G_{1,1}$ has one distinct extremal
cell, although that cell is vacuously both a strict minimum and a strict
maximum. Thus the auxiliary grid-height count has $E_d(1,1)=0$ for $d\ge2$.

\subsection{Apex encoding and Presburger counting}
The envelope bounds how many integers determine a height. The encoding
below stores the less numerous extremum set and supplies the growth bound
in Corollary~\ref{cor:degree-nine-reduction}. Its role in the alternative
degree-eight certificate~\cite{Watts2026} is described in
Appendix~\ref{app:independent-eight}. The Presburger construction instead
keeps all minima and their offsets. Neither encoding is a premise of
Theorem~\ref{thm:poly}.

\begin{corollary}[Parameter count]\label{cor:Odparams}
For integers $m,n\ge1$ and $d\ge2$ with $mn\ge2$, height functions
$h$ on the fixed grid $\Gmn$ with $d$ extrema admit an injective
encoding by apex coordinates and relative height values, using at most
$3\lfloor d/2\rfloor-1$ integers. The encoding uses as apexes the strict
local minima of whichever of $h$ and $-h$ has fewer minima, choosing
$h$ in a tie. The tuple stores their coordinates and, for each apex
other than the lexicographically first, the chosen height's value at
that apex minus the chosen height's value at the first apex.
A bit recording whether $h$ or $-h$ was chosen is packed into the
first apex's row entry.
\end{corollary}

\begin{proof}
Negation exchanges minima and maxima while preserving root normalization
and unit edge differences. A height $h$ with $d$ extrema has $a$ minima
and $b$ maxima. Every cell has a neighbour, so the two sets are disjoint
and $a+b=d$. The encoding uses $g=h$ and $\epsilon=0$ when $a\le b$,
and $g=-h$ and $\epsilon=1$ otherwise. Both extremum sets are
nonempty by Lemma~\ref{lem:coneenv} on rectangles and
Proposition~\ref{prop:path-envelope} on paths, so the normalized height
$g$ has $k=\min(a,b)$ minima, with $k\ge1$.
Lexicographic order indexes the apexes as
$p_1,\dots,p_k$, with $p_s=(r_s,\kappa_s)$ and offsets $q_s=g(p_s)$.
Theorem~\ref{thm:envelope} on rectangles and
Proposition~\ref{prop:path-envelope} on paths give
\[
  g(v)=\min_{1\le s\le k}\bigl(q_s+\dgrid(p_s,v)\bigr)
  \qquad\text{for }v\in\Gmn.
\]

\medskip\noindent
The relative offsets are $\delta_1=0$ and $\delta_s=q_s-q_1$ for $2\le s\le k$.
The single entry $\theta=2(r_1-1)+\epsilon$ records both the row $r_1$
of $p_1$ and the sign bit $\epsilon$. Concatenating $\theta$,
the column $\kappa_1$ of $p_1$ and the triples $(r_s,\kappa_s,\delta_s)$
for $2\le s\le k$ defines the encoding
\[
  h\longmapsto
  \begin{cases}
    (\theta,\kappa_1), & k=1,\\
    (\theta,\kappa_1,r_2,\kappa_2,\delta_2,\ldots,r_k,\kappa_k,\delta_k),
      & k\ge2.
  \end{cases}
\]
The two initial entries and $k-1$ triples contain
$2+3(k-1)=3k-1$ integers.

\medskip\noindent
For a tuple produced by this encoding, its number of entries determines
$k$, and the first entry gives $r_1=\lfloor\theta/2\rfloor+1$ and
$\epsilon=\theta\bmod2$. The remaining entries give $\kappa_1$ and
the other apex coordinates and relative offsets. At the root, the
envelope representation and $q_s=q_1+\delta_s$ give
\begin{align*}
  0=g(1,1)&=\min_{1\le s\le k}
    \bigl(q_s+\dgrid(p_s,(1,1))\bigr)\\
    &=q_1+\min_{1\le s\le k}
    \bigl(\delta_s+\dgrid(p_s,(1,1))\bigr).
\end{align*}
Solving for $q_1$ recovers all offsets as
\[
  q_1=-\min_{1\le s\le k}
       \bigl(\delta_s+\dgrid(p_s,(1,1))\bigr),
  \qquad q_s=q_1+\delta_s\quad\text{for }1\le s\le k.
\]
The recovered apexes and offsets reconstruct $g$ on $\Gmn$, and
the bit gives $h=(-1)^\epsilon g$. These steps depend only on the
tuple and the fixed grid and recover the original $h$. Thus decoding
is a left inverse of the encoding, so the encoding is injective.
Since $k\le\lfloor d/2\rfloor$, the tuple has at most
$3\lfloor d/2\rfloor-1$ entries.
\end{proof}

\label{sec:quasipoly}

For the Presburger count, the tuple retains every minimum and its value.
The Maxima Criterion fixes the other extremum count. For a fixed split,
these conditions involve a fixed number of integer coordinates and
offsets, even as the grid dimensions vary.

The grid dimensions are the parameters of this integer-point count.
A Presburger formula is built from integer-coefficient linear equalities
and inequalities, Boolean connectives and quantifiers over natural numbers
\cite[Definition 1.1 and Example 1.2]{Woods2015}.
In a counting application, $x$ gives the counted coordinates, $t$ the
parameters and $y$ the auxiliary quantified coordinates, all with
natural-number entries. With fixed integer coefficients, a linear
inequality involving these vectors has the form
\[
  \alpha\cdot x+\lambda\cdot y\le\beta+\gamma\cdot t.
\]
Here a Presburger family consists of
sets $S_t$ of counted points $x$, indexed by $t\in\mathbb N^k$,
defined by such formulas with $x,t$ free and the auxiliary variables bound.

A function $g:\mathbb N^k\to\mathbb Q$ is a quasi-polynomial if it agrees with a
polynomial on each residue class modulo some period. It is piecewise
quasi-polynomial if $\mathbb N^k$ has a finite partition into the integer
points of polyhedra, with $g$ agreeing with a quasi-polynomial on each part
\cite[Definitions 1.8--1.9]{Woods2015}. For a Presburger family with every
$S_t$ finite, its counting function $g(t)=|S_t|$ is piecewise quasi-polynomial
\cite[Definition 1.6 and Theorem 1.10]{Woods2015}.

\begin{theorem}[Joint quasi-polynomiality]\label{thm:quasipoly}
For each fixed $d\ge2$, $E_d(m,n)$ is piecewise quasi-polynomial in $(m,n)$
on $\{m,n\ge1\}$.
\end{theorem}

\begin{proof}
For each split, the envelope gives a fixed-length tuple of apex coordinates and
offsets. A linear shift makes the counted variables nonnegative, and the
configuration conditions become a Presburger formula with finite fibres.
The finite sum over the splits then gives $E_d$.

\medskip\noindent
\emph{Configuration space.}
For positive integers $a,b$ with $a+b=d$,
Theorem~\ref{thm:envelope} for rectangles and
Proposition~\ref{prop:path-envelope} for paths identify the heights counted by $N_{(a,b)}$
with the admissible configurations having $a$ apexes and $b$ maxima.
Lexicographic order fixes the apex indexing. With $L=m+n-2$, the shifted
offset coordinates give the encoding
\[
  x=(r_1,\kappa_1,z_1,\dots,r_a,\kappa_a,z_a)\in\mathbb N^{3a},
  \qquad c_s=z_s-L.
\]
The defining formula includes the parameter guard as a conjunct,
\[
  m\ge1\ \wedge\ n\ge1\ \wedge\ (m\ge2\ \vee\ n\ge2).
\]
This is equivalent to $m,n\ge1$ with $mn\ge2$. On this domain, the finite
offset bounds below make the shift a bijective encoding. All $3a$ variables
are retained. Normalization is imposed as a constraint, not used to reduce
the number of counted coordinates.

\medskip\noindent
\emph{Presburger encoding.}
After substituting $c_s=z_s-L$, each defining condition is a Presburger
formula in $x$ and the parameters $(m,n)$.
Each apex cell obeys the grid bounds $1\le r_s\le m$ and $1\le\kappa_s\le n$, and the apexes are listed in lexicographic order, $(r_s,\kappa_s)\prec(r_{s+1},\kappa_{s+1})$ for $s<a$, so each configuration is a single point.
There is one activity inequality~\eqref{eq:active},
$c_s<c_t+\dgrid(p_t,p_s)$, for each apex $s$ against each of the other
$a-1$ apexes, giving $a(a-1)$ in all. Each involves a grid distance
$\dgrid$, a sum of absolute values, hence expands to a Boolean combination
of linear inequalities. The parity congruences~\eqref{eq:par},
$c_s-c_t\equiv\dgrid(p_s,p_t)\pmod2$, give one condition for each of
the $\binom{a}{2}$ unordered apex pairs. They are symmetric under
$s\leftrightarrow t$, as $\dgrid$ is symmetric and
$c_s-c_t\equiv c_t-c_s\pmod2$.
The common shift cancels in the activity and parity conditions.
Since addition and subtraction agree modulo two, each parity condition
is equivalent to
\[
  \exists\,q_{st}\in\mathbb N:\quad
  z_s+z_t+\dgrid(p_s,p_t)=2q_{st}.
\]
The left side is nonnegative, so a natural-number witness suffices.
After splitting the absolute values, this is an integer-linear equality
with one auxiliary quantified variable.
Normalization
at the root, where $\dgrid(p_s,(1,1))=r_s+\kappa_s-2$, becomes
\[
  \bigwedge_s(z_s+r_s+\kappa_s\ge m+n)
  \quad\wedge\quad
  \bigvee_s(z_s+r_s+\kappa_s=m+n).
\]
The split further requires exactly $b$ maxima. Under the imposed parity,
Lemma~\ref{lem:maxima} identifies a maximum by requiring every existing
direction to lead toward an attaining apex. The predicate $\Phi(A,c,v)$
combines these four conditions with the grid bounds on $v$.
Each witness is a pair of auxiliary natural-number coordinates.
The two conditions for exactly $b$ maxima are
\[
\begin{aligned}
  &\exists\,v_1,\dots,v_b\ \text{distinct}\
    \bigwedge_{i=1}^{b}\Phi(A,c,v_i),\\
  &\neg\,\exists\,w_1,\dots,w_{b+1}\ \text{distinct}\
    \bigwedge_{i=1}^{b+1}\Phi(A,c,w_i).
\end{aligned}
\]
Attainment is a finite conjunction of cone
comparisons, and the Maxima Criterion imposes its four existing-neighbour
conditions by Boolean connectives. Each absolute value is eliminated by its
two linear sign branches, and distinct cells differ in at least one
coordinate. Thus both requirements are Presburger formulas.
The parity and maximum witnesses are quantified variables, not additional
counted coordinates. They occupy the auxiliary vector $y$ in the displayed
linear form. Their order does not multiply the count.
Conjoining the parameter guard, these conditions and the bounds
$0\le z_s\le2L$ gives a Boolean combination of integer-linear conditions
with $x,(m,n)$ free and all auxiliary natural-number variables bound.
Thus its sets $S_{(m,n)}$ of counted points form a Presburger family
with parameters $(m,n)$, in the defining form stated above.

\medskip\noindent
\emph{Finiteness.}
The box bounds preserve every admissible
configuration, because activity gives $c_s=E_{A,c}(p_s)$ and the normalized
$1$-Lipschitz envelope gives
\[
  |c_s|\le\dgrid(p_s,(1,1))\le L.
\]
Thus $z_s=c_s+L$ lies in the box. Conversely, decoding $c_s=z_s-L$
restores all configuration conditions. The finite coordinate and offset
boxes give a finite fibre, and the parameter guard gives an empty fibre
elsewhere. For $m,n\ge1$ with $mn\ge2$,
$|S_{(m,n)}|=N_{(a,b)}(m,n)$.

\medskip\noindent
\emph{Assembly.}
The counting function $(m,n)\mapsto|S_{(m,n)}|$ is piecewise
quasi-polynomial by the stated counting theorem. Summing these functions
over the finitely many splits of $d$ preserves piecewise quasi-polynomiality,
using a common refinement of the partitions and a common period.
For $m,n\ge1$ with $mn\ge2$, equation~\eqref{eq:Ed} gives
$E_d(m,n)=\sum_{a+b=d}N_{(a,b)}(m,n)$, so the sum of family counts
is $E_d(m,n)$. At $(1,1)$ the
parameter guard makes each family empty, so the sum is zero, agreeing
with $E_d(1,1)=0$ for $d\ge2$. The sum therefore agrees with $E_d$ on
all positive grid sizes, proving the claim. Only $m,n\ge2$ has the
Miura-OFG interpretation.
\end{proof}

The Presburger route gives a finite-piece counting law directly from the
cone constraints, with tuple length depending on $d$ rather than the
dimensions. It does not construct a transfer or classify zero-extremum
states. Conversely, the transfer proof in Section~\ref{sec:counts} does
not use this integer-point count and gives the stronger period-one law.

For fixed $d\ge3$, Theorem~\ref{thm:poly} gives $p_d$ on $H_d$.
Lemma~\ref{lem:uniform} also gives a polynomial on every strip with fixed
$2\le m<d-1$ and $n\ge d-1$, and transposition gives the corresponding
column strips. Each remaining rectangle with both dimensions below
$d-1$ can form a single-point piece. Lemma~\ref{lem:binom} covers the
paths, and $E_d(1,1)=0$ covers the singleton. For $d=2$, rectangles have
count four, paths have count two and the singleton has count zero.
These pieces give a finite piecewise-polynomial description on all
positive grid sizes, strengthening Theorem~\ref{thm:quasipoly} without
using it as a premise.

\subsection{Rationality without frozen structure}\label{app:coefficient-series}

A determinant argument proves coefficient rationality using only a finite
matrix transfer and polynomial endpoint weights. Unlike the rotation-quotient
argument of Lemma~\ref{lem:quotient}, it does not identify which spectral
poles survive the scalar count. It therefore supplies a separate
rationality check when no frozen-state classification is available.

\begin{lemma}[Coefficient rationality and poles]\label{lem:poles}
For integers $m\ge2$ and $d\ge2$, the width series
$\sum_{n\ge2}E_d(m,n)z^n$ is rational and regular at $z=0$.
Its finite poles lie among the reciprocals of the nonzero eigenvalues
of $T_m(0)$.
\end{lemma}

\begin{proof}
The transfer identity~\eqref{eq:transfer-enumerator} gives the width
series $F=z^2\mathbf u^{\!\top}(I-zT_m(x))^{-1}\mathbf v$.
Its determinant and adjugate polynomials are
\[
\begin{aligned}
 D(x,z)&=\det(I-zT_m(x)),\\
 \mathcal N(x,z)&=z^2\mathbf u^{\!\top}
   \operatorname{adj}(I-zT_m(x))\mathbf v.
\end{aligned}
\]
The endpoint vectors are polynomial in $x$, so $D$ and $\mathcal N$
belong to $\Z[x,z]$, and the adjugate formula gives $DF=\mathcal N$.
For $j\ge0$, the coefficients are $f_j=[x^j]F$,
$D_j=[x^j]D$ and $\mathcal N_j=[x^j]\mathcal N$.
Coefficient extraction gives the recurrence
\[
  D_0f_j=\mathcal N_j-\sum_{k=1}^{j}D_kf_{j-k},
  \qquad \mathcal N_j,D_k\in\Q[z],\qquad D_0(0)=1.
\]
The empty sum is zero, and $D_0(z)=\det(I-zT_m(0))$.
Induction gives $f_j=P_j(z)/D_0(z)^{j+1}$ for a polynomial $P_j$.
Indeed, $f_0=\mathcal N_0/D_0$. At step $j$, each lower-index coefficient on the
right has denominator dividing $D_0^j$. Division by $D_0$ gives the
claimed form. The unit constant term makes each fraction regular at
zero and fixes its unique formal expansion there. Thus the finite poles of
$f_d=\sum_{n\ge2}E_d(m,n)z^n$ can occur only at zeros of $D_0$.

\medskip\noindent
The eigenvalues $\lambda_i$ of $T_m(0)$, with algebraic multiplicity, give
\[
  D_0(z)=\prod_i(1-z\lambda_i).
\]
A zero eigenvalue contributes one. Every other factor vanishes only at
$z=1/\lambda_i$. The same determinant, adjugate and coefficient recurrence
hold for finite transfers and polynomial endpoint weights over $\Q[x]$.
\end{proof}

\subsection{The finite-certificate converse}\label{app:finite-converse}

Corollary~\ref{conj:G2} already proves the split-remainder bound for every
$d\ge5$. The converse below gives an alternative certificate degree by
degree. An independent determination of the top-degree part can certify
the individual remainder bounds without using the uniform run theorem.
Nonnegative remainder counts prevent higher-degree terms from cancelling.

\begin{proposition}[Degree bound from the top-degree part]\label{prop:converse}
For an integer $d\ge5$, suppose that $p_d$ has total degree at most $d-2$ and its
total-degree-$(d-2)$ part is $C(d)\,(m^{d-2}+n^{d-2})$. Then every eventual
split-remainder polynomial $\rho_{(a,b)}$, for $a+b=d$ and $a\le b$, has total
degree at most $d-3$.
\end{proposition}

\begin{proof}
The split decomposition gives an identity independently of the
remainder bound. On $K_d$, Lemma~\ref{lem:splitedge} partitions each split
into its single-side families and its remainder. Summing with weights
$2-\delta_{ab}$ gives $E_d-S_d=\sum(2-\delta_{ab})R_{(a,b)}$.
Here $E_d=p_d$ and $R_{(a,b)}=\rho_{(a,b)}$, so
\[
  p_d-S_d
  =\sum_{\substack{1\le a\le b\\ a+b=d}}\bigl(2-\delta_{ab}\bigr)\rho_{(a,b)}.
\]
Lemma~\ref{lem:splitpoly} makes the right side polynomial, as is the left
side. Agreement on the infinite rectangle $K_d$ therefore makes this a
polynomial identity. By Proposition~\ref{prop:singleside} and the hypothesis,
$S_d$ and $p_d$ have the same degree-$(d-2)$ homogeneous part and no term
above it. Hence $\deg(p_d-S_d)\le d-3$.

\medskip\noindent
Each $\rho_{(a,b)}$ is a polynomial by Lemma~\ref{lem:splitpoly}, with
degree-$k$ homogeneous part $\rho^{[k]}_{(a,b)}$. Suppose some such
polynomial has total degree at least $d-2$. The largest total degree among
them is then $D\ge d-2$. Since $\rho_{(a,b)}$ agrees with the nonnegative
remainder count on $K_d$, it is nonnegative at every integer point of $K_d$.
For positive integers $m_0,n_0$ and all sufficiently large integers $t$,
$t(m_0,n_0)$ lies in $K_d$. Dividing the nonnegative remainder value by
$t^D$ and taking the limit gives
\[
  \rho^{[D]}_{(a,b)}(m_0,n_0)
  =\lim_{t\to\infty}\frac{\rho_{(a,b)}(t m_0,t n_0)}{t^{D}}\ \ge\ 0.
\]
Clearing denominators scales positive rational pairs to positive integer pairs.
Homogeneity of degree $D$ preserves the sign of $\rho^{[D]}_{(a,b)}$ under
positive scaling, and continuity extends nonnegativity to the positive cone $m,n>0$.
At least one degree-$D$ part is nonzero, so it is strictly positive
somewhere in this open cone. All other parts are nonnegative there and
each weight $2-\delta_{ab}$ is positive. Their weighted sum is therefore
a nonzero degree-$D$ component of $p_d-S_d$, contradicting
$\deg(p_d-S_d)\le d-3$.
\end{proof}

\subsection{Joint types and finite reductions}\label{app:finite-ridge}

The two-apex ridge criterion specifies the maxima and counts those on its
central arm. Pairwise ties give this information only when no other cone
lies below the pair. The joint criterion instead reads degree bounds
from the positive leading terms of the expansion counts and gives a
fixed-grid formulation at degree nine. These geometric and split-specific
refinements are not open steps in the general degree theorem.

\begin{lemma}[Two-apex ridge criterion]\label{lem:two-apex-ridge}
For an admissible configuration $(A,c)$ with apex set
$A=\{p_1,p_2\}$, where $p_k=(r_k,\kappa_k)$, $r_1\ne r_2$, and
$\kappa_1\ne\kappa_2$, the envelope $h=E_{A,c}$ has tie set $T$ consisting of the cells where the
two cones attain the same value. Outside $T$, a strict local maximum is
exactly a grid corner whose unique attaining apex lies on neither boundary
line through that corner. A tie cell $(i,j)\in T$ is a strict local maximum
iff both coordinate conditions hold
\[
\begin{aligned}
 &i\in\{1,m\}\ \text{or}\ \min(r_1,r_2)<i<\max(r_1,r_2),\\
 &j\in\{1,n\}\ \text{or}\ \min(\kappa_1,\kappa_2)<j<\max(\kappa_1,\kappa_2).
\end{aligned}
\]
\end{lemma}

\begin{proof}
At a tie cell $(i,j)$, both cones attain the envelope, so the directional
tests in Lemma~\ref{lem:maxima} can use either apex. In an interior row,
both vertical directions exist. The tests require an apex above and an apex
below, which is equivalent to
\[
  \min(r_1,r_2)<i<\max(r_1,r_2).
\]
On the first row, only the downward test remains. The distinct apex rows
ensure that at least one apex lies below. On the last row, the same
distinctness ensures an apex above. This proves the row condition.

\medskip\noindent
In an interior column, the leftward and rightward tests require an apex on
each side, which is equivalent to
\[
  \min(\kappa_1,\kappa_2)<j<\max(\kappa_1,\kappa_2).
\]
On the first column, the distinct apex columns ensure an apex to the right.
On the last column, they ensure an apex to the left. These are the only
horizontal tests there, proving the column condition. Lemma~\ref{lem:maxima}
requires both coordinate conditions, giving the tie-cell characterization.

\medskip\noindent
Outside $T$, one cone uniquely attains the envelope.
Lemma~\ref{cor:corners} identifies its strict local maxima as exactly the
corners whose attaining apex lies on neither boundary line through the
corner, proving the characterization outside $T$.
\end{proof}

\begin{corollary}[Central ridge arm]\label{cor:central-ridge-arm}
For an admissible two-apex configuration $(A,c)$ with apexes
$p_1=(r_1,\kappa_1)$ and $p_2=(r_2,\kappa_2)$ with $r_1<r_2$ and
$\kappa_1\ne\kappa_2$, the envelope is $h=E_{A,c}$. Inside the open rectangle between
the two apex rows and columns, the tie cells form a possibly empty diagonal
segment. It lies on a line of constant $i+j$ when $\kappa_1<\kappa_2$
and constant $i-j$ when $\kappa_1>\kappa_2$. If this segment has $L$ cells
and there are $B$
strict local maxima on the grid boundary, counting each corner once, then
$h$ has $L+B$ strict local maxima.
\end{corollary}

\begin{proof}
Inside the open rectangle, the signs of the coordinate differences from each
apex are fixed. When $\kappa_1<\kappa_2$, the tie equation and its collected
form are
\[
\begin{aligned}
  c_1+(i-r_1)+(j-\kappa_1)&=c_2+(r_2-i)+(\kappa_2-j),\\
  2(i+j)&=c_2-c_1+r_1+r_2+\kappa_1+\kappa_2.
\end{aligned}
\]
When $\kappa_1>\kappa_2$, they are
\[
\begin{aligned}
  c_1+(i-r_1)+(\kappa_1-j)&=c_2+(r_2-i)+(j-\kappa_2),\\
  2(i-j)&=c_2-c_1+r_1+r_2-\kappa_1-\kappa_2.
\end{aligned}
\]
Each equation fixes $i+j$ or $i-j$, so its solutions in the open rectangle
form a possibly empty segment on the corresponding line.

\medskip\noindent
Lemma~\ref{lem:two-apex-ridge} makes every cell of this segment a strict local
maximum. Conversely, an interior maximum cannot have a unique attaining cone
by that lemma. It must be a tie cell, and neither boundary alternative in the
lemma applies, so both coordinates lie strictly between the apex coordinates.
Thus the segment contains exactly the interior maxima. The $L$ interior
maxima and the $B$ distinct boundary maxima are disjoint and exhaust all
strict local maxima, giving $L+B$.
\end{proof}

For a pair of apexes sharing neither a row
nor a column, the tie equation from the proof of
Corollary~\ref{cor:central-ridge-arm} involves only those two cones.
The tie equation therefore also applies within a larger configuration.
Inside the open rectangle between these two apexes, their tie cells form
a possibly empty diagonal segment. When present, this ridge constrains the cones'
relative positions and offsets, coupling the coordinate choices.

With three or more apexes, a meeting of two cones need not be visible in the
envelope. On $G_{5,5}$, the apexes $(1,5)$, $(3,3)$ and $(5,1)$ all have
offset $-4$. Their common parity, strict activity and root value make the
configuration admissible, and its apex set is not a product grid, so it is
non-separable. The first and third cones agree at height zero along the main
diagonal. At $(i,i)$ the middle cone has value $-4+2|i-3|$, strictly below zero
at the interior diagonal cells. The diagonal ridge of the pair therefore does not give a ridge of the
three-cone envelope. Pairwise tie equations therefore do not determine
the maxima when another cone lies below the pair.

For the joint criterion, an eligible type is a non-separable joint type
with positive row and column run counts. Lemma~\ref{lem:joint-frozen}
excludes zero-run types on $K_d$, so precisely the eligible types
contribute to its non-separable counts.

\begin{proposition}[Joint run-count criterion]\label{prop:joint-criterion}
For every integer $d\ge5$, every split remainder has total degree at most $d-3$ iff
every eligible degree-$d$ joint type $H$ satisfies $r_H+c_H\le d-1$.
\end{proposition}

\begin{proof}
Negation exchanges minima and maxima and preserves normalization,
separability and the occupied levels in both axes. It preserves the run
counts, so it suffices to consider splits $(a,b)$ with $a\le b$.
For such a split, Lemma~\ref{lem:joint-frozen} excludes zero-run types
on $K_d$ and partitions the non-separable heights into eligible types.
Their finite set is $\mathcal J^{\mathrm{ns}}_{(a,b)}$, and
its eventual count polynomial is $\rho^{\mathrm{ns}}_{(a,b)}$. Each type $H$
has dimensions $m_H,n_H$ and run counts $r_H,c_H$.
On $K_d$, the joint expansion gives
\[
  \rho^{\mathrm{ns}}_{(a,b)}(m,n)
  =\sum_{H\in\mathcal J^{\mathrm{ns}}_{(a,b)}}
    \binom{m-m_H+r_H-1}{r_H-1}
    \binom{n-n_H+c_H-1}{c_H-1}.
\]
Each summand has total degree $r_H+c_H-2$ and highest term
\[
  \frac{m^{r_H-1}n^{c_H-1}}{(r_H-1)!(c_H-1)!}.
\]
The coefficients in these highest terms are positive, so the terms of
maximal total degree cannot cancel. For a nonempty type set,
$\deg_{\mathrm{tot}}\rho^{\mathrm{ns}}_{(a,b)}=\max_H(r_H+c_H-2)$.
An empty set contributes the zero polynomial. The separable remainder
has degree at most $d-3$ by Proposition~\ref{thm:sepcase}. Thus the full
remainder bound is equivalent to $r_H+c_H\le d-1$ for every eligible
type, as claimed.
\end{proof}

At degree nine, one-apex geometry excludes split $(1,8)$, and the
injective encoding bounds split $(2,7)$. The joint criterion reduces
the remaining splits to two conditions on a fixed square. These
deductions do not use the twenty-node interpolation.

\begin{corollary}[Degree-nine finite reduction]\label{cor:degree-nine-reduction}
On $K_9$,
\[
 \rho_{(1,8)}=0,\qquad \deg_{\mathrm{tot}}\rho_{(2,7)}\le5.
\]
Moreover, every degree-nine split remainder has total degree at most six iff, for each split
$(a,b)\in\{(3,6),(4,5)\}$, every non-separable normalized height of
that split on $G_{17,17}$ has at most eight maximal extremum-free row
and column runs in total.
\end{corollary}

\begin{proof}
\emph{One minimum.}
Corollary~\ref{cor:one-apex} gives at most four maxima for a height
with one minimum, so no height has split $(1,8)$. Its remainder is a
subset of that empty class, giving $R_{(1,8)}=0$ and hence
$\rho_{(1,8)}=0$ by polynomial agreement on $K_9$.

\medskip\noindent
\emph{Two minima.}
For split $(2,7)$,
Corollary~\ref{cor:Odparams} gives an injective encoding by four apex
coordinates and one relative offset. Along a ray $(m,n)=t(m_0,n_0)$
with positive integers $m_0,n_0$, each coordinate has $O(t)$ choices.
The relative offset has magnitude at most $m+n-2$ by the grid Lipschitz
bound, so it also has $O(t)$ choices. Hence $N_{(2,7)}=O(t^5)$.
Lemma~\ref{lem:binom} gives $f_{(2,7)}=0$, and
Lemma~\ref{lem:splitedge} therefore gives $R_{(2,7)}=N_{(2,7)}$.
Its eventual polynomial on $K_9$ has degree at most five. Otherwise,
its nonzero highest homogeneous part, of degree $D\ge6$, would be
nonzero in some positive integer direction, as in the argument of
Proposition~\ref{prop:converse}. Evaluation on that ray would grow in
absolute value as $t^D$, contradicting the $O(t^5)$ bound.

\medskip\noindent
For an eligible type $H$ of split $(2,7)$,
Lemma~\ref{lem:joint-frozen} gives a preimage polynomial of total degree
$r_H+c_H-2$. Along an arbitrary ray $(m,n)=t(m_0,n_0)$ with
positive integers $m_0,n_0$, this polynomial has a positive leading
coefficient. For all sufficiently large $t$, these preimages form a subset
of the $N_{(2,7)}$ heights. The $O(t^5)$ bound therefore forces
\[
  r_H+c_H-2\le5,\qquad r_H+c_H\le7.
\]
There is no joint type of split $(1,8)$, because such a type would itself
be a height of that split.
Thus every eligible type of either lower split satisfies the required
budget of eight runs. Negation gives the same conclusion for the reversed
splits.

\medskip\noindent
\emph{Fixed-square equivalence.}
It remains to treat the ordered splits $(3,6)$ and $(4,5)$ and their
reversals. By Proposition~\ref{prop:joint-criterion}, failure of the degree-nine split-remainder bound is therefore equivalent to an eligible type of one
of these remaining splits having at least nine runs in total.
Negating a type interchanges minima and maxima at the same cells while
preserving dimensions, non-separability and both run counts. Thus a violating
type of one of the reversed splits gives a violating type of split
$(3,6)$ or $(4,5)$.
Lemma~\ref{lem:joint-frozen} bounds its two dimensions by seventeen.
Since both run counts are positive, the type expands to $17\times17$
with its split, non-separability and run counts unchanged. Conversely,
a non-separable height of either stated split on $17\times17$ with
at least nine runs contracts to a violating type. Both run counts are
positive, because each axis has at most eight runs. This proves the
equivalence.
\end{proof}

Theorem~\ref{thm:uniform-run} excludes the violating heights in both
fixed-square conditions, so neither is an unresolved target. The
reduction is an alternative formulation, not a premise of that theorem
or of the twenty-node coefficient identification. Its two-minimum
remainder bound has degree at most five, stronger than the general
degree-nine bound of six.

\subsection{A separate degree-eight certificate}\label{app:independent-eight}
The note~\cite[Version~1.0]{Watts2026} treated degree eight using the framework
of an earlier public version of this counting
work~\cite[Version~2]{Gupta2026sequel}, in which that case remained open.
Its argument certifies the polynomial $p_8$ already displayed there,
combining a growth bound from the parameter count of
Corollary~\ref{cor:Odparams} with exact counts on six rows of the tensor grid.
The certificate here instead uses the full grid $\{7,\dots,13\}^2$,
without that growth bound. Both certificate routes use split-refined
polynomiality and nonnegative remainder counts to pass from the aggregate
degree drop to the individual split bounds.

Theorem~3.1 of that note states per-split polynomiality for every $d\ge3$
on $m,n\ge d-1$, a larger region than $K_d$. At degree eight, the note states
that every split remainder agrees on $m,n\ge7$ with a polynomial of total
degree at most five, and is piecewise polynomial on all of $m,n\ge2$ with
pieces of total degree at most five, using certificates at fixed widths two
through six. These degree-eight domains contain $K_8=\{m,n\ge15\}$, where
Lemma~\ref{lem:splitpoly} and Corollary~\ref{conj:G2} apply. The wider
domains are results of the cited note, not premises here.

For $s=2,\ldots,6$, its aggregate remainder is
$A_8(s,n)=E_8(s,n)-S_8(s,n)$. The correction in this article is
$B_8^{(s)}(n)=E_8(s,n)-p_8(s,n)$, so the two data sets are related by
\[
 A_8(s,n)=B_8^{(s)}(n)+p_8(s,n)-S_8(s,n).
\]
The explicit correction polynomials here have their own transfer-count
certificates. No result of the note is a counting premise.

\subsection{The strict-chain completion}\label{app:diameter-routes}

Counting integer level cuts gives a second route to the grid-dispersion
bound. It converts dispersion into a sum over sublevel-set sizes, to which
the cited strict-chain theorem applies.

\begin{lemma}[Median level cuts]\label{lem:median-level-cuts}
For a nonconstant integer-valued function $f$ on a finite set of size $N$,
the extreme values are $a=\min f$ and $b=\max f$, and $c_\ell$ counts the values at or
below $\ell$. Its dispersion is
\[
  \disp(f)=\sum_{\ell=a}^{b-1}\min(c_\ell,N-c_\ell).
\]
\end{lemma}

\begin{proof}
Counting integer level cuts expresses each absolute deviation as a sum of
unit contributions. An integer $K$ outside the interval from $a$ to $b$
cannot minimize the sum, since moving $K$ to the nearest extreme decreases
every distance. For an integer $K$ between $a$ and $b$, a value below $K$
crosses the cuts from its own level through $K-1$, while a value above $K$
crosses the cuts from $K$ through its level minus one. Summing gives
\[
  \sum_v|f(v)-K|
  =\sum_{\ell=a}^{K-1}c_\ell+
    \sum_{\ell=K}^{b-1}(N-c_\ell).
\]
Each cut therefore contributes at least $\min(c_\ell,N-c_\ell)$.
A median $K$ chosen among the integer values has at most $N/2$ values
strictly on either side. When $\ell<K$, the $c_\ell$ values are strictly
below $K$, so $c_\ell\le N/2$. When $\ell\ge K$, the $N-c_\ell$ values
are strictly above $K$, so $N-c_\ell\le N/2$. Every cut now contributes
its smaller count, attaining the lower bound and proving the formula.
\end{proof}
Lemma~\ref{lem:monotone-reduction} preserves the value multiset and makes
every sublevel set an order ideal of the coordinatewise grid order.
For $m\ge n\ge2$, the set $N[S]$ consists of an order ideal $S$ and its grid
neighbours, and $\nu(t)$ is the minimum of $|N[S]|$ over order ideals
of size $t$. A strict $\nu$-chain is a strictly increasing list
$1\le c_1<\cdots<c_k\le mn-1$ with $c_{j+1}\ge\nu(c_j)$.
The strict-chain theorem proves that every such list has
$\sum_j\min(c_j,mn-c_j)\le D(m,n)$, settling the earlier
slowest-chain conjecture~\cite[Version~2, Conjecture~4.6]{Gupta2026companion}
by a cell-weight certificate~\cite[Theorem~1.1]{Watts2026diameter}.

For a nonconstant nondecreasing integer $1$-Lipschitz function $F$, its sublevel set at $\ell$ is
$S_\ell=\{v:F(v)\le\ell\}$. The edge bound gives
\[
  N[S_\ell]\subseteq S_{\ell+1}.
\]
Connectedness ensures that no integer level between the extremes is skipped.
The nonempty proper sublevel sets therefore have strictly increasing sizes,
each at least $\nu$ of the preceding size. They form a strict $\nu$-chain,
whose score is $\disp(F)$ by Lemma~\ref{lem:median-level-cuts}.
The strict-chain theorem bounds this score by $D(m,n)$. A constant function
has dispersion zero, so the bound holds for every integer $1$-Lipschitz
grid function. Transposition covers $m<n$.

This gives the grid inequality stated
in~\cite[Corollary~5.1]{Watts2026diameter} and completes the earlier diameter
reduction. The equality classification in
Section~\ref{subsec:dispersion-equality} instead uses covariance.

\section*{Acknowledgements}
Simon Watts, independent researcher, is thanked for resolving the earlier
slowest-chain conjecture~\cite[Version~2, Conjecture~4.6]{Gupta2026companion},
whose application proves the grid bound used in the diameter reduction,
and for establishing the exact rectangular
diameter~\cite[Theorems~1.1 and~1.2]{Watts2026diameter}. He is also thanked
for the mathematical correspondence and for his independent
degree-eight work~\cite[Version~1.0]{Watts2026}.
Haobo Ma and Wenlin Zhang are thanked for communicating their covariance
proof of the same grid bound in private correspondence.

\bibliographystyle{amsplain}
\bibliography{references}

\section*{Declarations}
No funding was received for this work. The author declares no conflicts of
interest.

\section*{AI Disclosure}
GitHub Copilot was used in mathematical exploration and proof development,
implementation and checking of computational code, and drafting and revision
of the manuscript. These uses materially influenced the counting and distance
arguments, their computational support and the exposition throughout the
article. The author retains responsibility for the correctness, originality,
attribution and presentation of the work.

\end{document}